\documentclass[oneside]{amsart}
\usepackage[english]{babel}
\usepackage[T1]{fontenc}
\usepackage[latin9]{inputenc}
\usepackage[dvipsnames]{xcolor}
\usepackage{units}
\usepackage{mathrsfs}
\usepackage{mathtools}
\usepackage{amsbsy}
\usepackage{amsthm}
\usepackage{amssymb}
\usepackage{setspace}
\usepackage{esint}
\usepackage[hidelinks]{hyperref}
\usepackage{xcolor}
\usepackage{makecell}
\usepackage{float}
\usepackage{upgreek}
\usepackage{enumitem}
\usepackage{bbm} 
\usepackage{tikz}

\makeatletter
\numberwithin{equation}{section}
\numberwithin{figure}{section}

\newtheorem{theorem}{Theorem}[section]
\newtheorem*{theorem*}{Theorem}
\newtheorem{corollary}[theorem]{Corollary}

\newtheorem{lemma}[theorem]{Lemma}
\newtheorem{proposition}[theorem]{Proposition}

\theoremstyle{definition}
\newtheorem{definition}{Definition}

\newtheorem*{example*}{Example}

\theoremstyle{remark}
\newtheorem{remark}[theorem]{Remark}

\def\moverlay{\mathpalette\mov@rlay}
\def\mov@rlay#1#2{\leavevmode\vtop{
   \baselineskip\z@skip \lineskiplimit-\maxdimen
   \ialign{\hfil$\m@th#1##$\hfil\cr#2\crcr}}}
\newcommand{\charfusion}[3][\mathord]{
    #1{\ifx#1\mathop\vphantom{#2}\fi
        \mathpalette\mov@rlay{#2\cr#3}
      }
    \ifx#1\mathop\expandafter\displaylimits\fi}

\makeatother

\newcommand{\cupdot}{\charfusion[\mathbin]{\cup}{\cdot}}
\newcommand{\bigcupdot}{\charfusion[\mathop]{\bigcup}{\cdot}}

\usepackage{babel}
\newcommand{\SO}{\operatorname{SO}} 

\newcommand{\SL}{\operatorname{SL}} 

\newcommand{\diag}{\operatorname{diag}} 

\newcommand{\Lie}{\operatorname{Lie}} 

\newcommand{\Span}{\operatorname{Span}} 

\newcommand{\Gr}{\operatorname{Gr}}

\newcommand{\dtt}{\mathtt{d}} 

\newcommand{\att}{\mathtt{a}} 

\newcommand{\BC}{\operatorname{BC}} 

\newcommand{\covol}{\operatorname{covol}}

\newcommand{\Ext}{\bigwedge} 

\newcommand{\cusp}{\mathcal{N}}

\newcommand{\Eigen}{\operatorname{Eigen}} 

\newcommand{\Leb}{\operatorname{Leb}} 

\newcommand{\compact}{\operatorname{Compact}} 

\newcommand{\SLdRZ}{\SL_{\dtt}(\mathbb{R})/\SL_{\dtt}(\mathbb{Z})} 

\newcommand{\RBS}{\operatorname{RBS}} 

\newcommand{\Par}{\operatorname{Par}} 

\newcommand{\Weyl}{\operatorname{Weyl}} 

\newcommand{\height}{\operatorname{height}} 

\newcommand{\err}{\operatorname{err}} 

\newcommand{\Opt}{\operatorname{Opt}} 

\newcommand{\stab}{\operatorname{stab}} 

\newcommand\numberthis{\addtocounter{equation}{1}\tag{\theequation}}

\newcommand\mathitem{\item\leavevmode\vspace*{-\dimexpr\baselineskip+\abovedisplayskip\relax}}

\begin{document}

\title
[Bounding entropy using linear functionals]
{Bounding entropy for one-parameter diagonal flows on $\SL_{\mathtt{d}}(\mathbb{R})/ \SL_{\mathtt{d}}(\mathbb{Z})$ using linear functionals}
\author{Ron Mor}

\begin{abstract}
We give a method to bound the entropy of measures on  $\SL_{\dtt}(\mathbb{R})/\SL_{\dtt}(\mathbb{Z})$ which are invariant under a one parameter diagonal subgroup, in terms of entropy contributions from the regions of the cusp corresponding to different parabolic groups. These bounds depend on an auxiliary linear functional on the Lie algebra of the Cartan group. In follow-up papers we will show how to optimize this functional to get good bounds on the cusp entropy and prove that in many cases these bounds are sharp.
\end{abstract}
\thanks{This work was supported by ERC 2020 grant HomDyn (grant no.~833423)}
\address{The Einstein Institute of Mathematics\\
	Edmond J. Safra Campus, Givat Ram, The Hebrew University of Jerusalem
	Jerusalem, 91904, Israel}
	
\maketitle

\tableofcontents

\section{Introduction}

\subsection{Motivation and background}
This paper is the first in a sequence of papers in which we study the entropy of
any one-parameter diagonal subgroup $\att_{\bullet}=\{\mathbf{a}_{t}\}_{t\in\mathbb{R}}$ of $\SL_{\mathtt{d}}(\mathbb{R})$, with positive entries on the diagonal, on the space $\SL_{\mathtt{d}}(\mathbb{R})/ \SL_{\mathtt{d}}(\mathbb{Z})$ of unimodular lattices in $\mathbb{R}^{\mathtt{d}}$.
In this paper we prove general results considering upper bounds on the entropy of $\att_{\bullet}$-invariant probability measures, which combine entropy contributions from the different parts of the cusp. 
In subsequent papers we use these results to deduce good upper bounds for the entropy in the cusp in various cases~\cite{mor2022b}, as well as prove that these bounds are tight~\cite{mor2022c}.

Questions of this flavor were addressed in various settings by some authors. In particular we note the works of Einsiedler, Lindenstrauss, Michel and Venkatesh~\cite{einsiedler2011distribution} who worked on  $\SL_{2}(\mathbb{R})/ \SL_{2}(\mathbb{Z})$, Einsiedler, Pohl and Kadyrov \cite{einsiedler2015escape} who generalized these results to diagonal actions on spaces $G/\Gamma$ where $G$ is a connected semisimple real Lie group of rank 1 with
finite center, and $\Gamma$ is a lattice, Einsiedler and Kadyrov \cite{einsiedler2012entropy} who considered the action of $\att_t=\diag(e^{t/2},e^{t/2},e^{-t})$ on $\SL_{3}(\mathbb{R})/ \SL_{3}(\mathbb{Z})$, 
 Kadyrov, Kleinbock, Lindenstrauss and Margulis \cite{kadyrov2017singular} for the action of
 \[\att_{t}=\diag(e^{nt},\ldots,e^{nt},e^{-mt},\ldots,e^{-mt}),\]
 on $\SL_{m+n}(\mathbb{R})/\SL_{m+n}(\mathbb{Z})$, 
as well as~\cite{mor2021excursions} where the frame flow on geometrically finite hyperbolic orbifolds is studied. 

These types of estimates give in particular an estimate of the entropy in the cusp, which is the maximal entropy that may be obtained by a limit of $\mathbf{a}_{\bullet}$-invariant probability measures which converge to the zero measure.
Such estimates are inherently included in the results of the aforementioned authors, but this question is interesting on its own and was also addressed directly by Iommi, Riquelme and Velozo (in two papers with different sets of coauthors~\cite{iommi2015entropy,riquelme2017escape}) who considered entropy in the cusp for geometrically finite Riemannian manifolds with pinched negative sectional curvature
and uniformly bounded derivatives of the sectional curvature.

A last family of related works is the study of the Hausdorff dimension of the set of singular systems of linear forms. 
This set is related to dynamics of one parameter diagonal groups --- a relation known as Dani's correspondence \cite{dani1985divergent}. Dani's correspondence was generalized by Kleinbock~\cite{kleinbock1998bounded,kleinbock1998flows}. In particular we note in this direction the work of Yitwah Cheung \cite{cheung2011hausdorff} and his joint work with Chevallier \cite{cheung2016hausdorff}, the work of Liao, Shi, Solan and Tamam \cite{liao2016hausdorff}, and the work by Das, Fishman, Simmons and Urba{\'n}ski \cite{das2019variational}. Via Dani's correspondence this aspect is also discussed in the papers \cite{einsiedler2012entropy,kadyrov2017singular} quoted above. The most recent work is of Solan~\cite{solan2021parametric} who studied the Hausdorff dimension with respect to the expansion metric. Solan's results are related to one of the results we will report in~\cite{mor2022b} regarding explicit computations of the entropy in the cusp, i.e.\ entropy after going to the limit; we make use of some of Solan's results in~\cite{mor2022c} to prove our estimates are sharp.

\medskip

In this paper we prove a general theorem which does
not only give an upper bound for the entropy in the cusp, but rather a bound for the entropy before going to the limit. In the higher rank case, the cusp has important finer structure: in the natural compactification of $G/\Gamma$ --- the reductive Borel-Serre compactification --- one adds a subvariety at infinity to $G/\Gamma$ for each standard parabolic proper subgroup of $G$. Our results give an estimation for the entropy contribution depending on the measures given to each of these different cusp regions, rather than only for the cusp as a whole, a feature that was not discussed in the previous literature on the subject.
The approach we employ in this paper is quite different from that considered in previous works, an approach that enables us to compute upper bounds for the entropy under certain constraints without having to consider all of the different trajectories in the cusp, by that simplifying the question considerably.

\subsection{Settings}\label{section:introduction}
Let $G=\SL_{\mathtt{d}}(\mathbb{R})$ and $\Gamma=\SL_{\mathtt{d}}(\mathbb{Z})$. 
The quotient space $G/\Gamma$ may be identified with the space of unimodular lattices in $\mathbb{R}^{\mathtt{d}}$. Let $A\leq G$ be the subgroup of diagonal matrices with positive entries on the diagonal, and let $\{\att_{t}\}_{t\in\mathbb{R}}\leq A$ be a one-parameter diagonal flow on $G/\Gamma$, i.e.\ $\att_t=\exp(t\mathbf{\upalpha})$ for all $t$, for some $\upalpha=\diag(\upalpha_1,\ldots,\upalpha_\dtt)$ with $\sum_{i=1}^{\dtt} \upalpha_i=0$, i.e.\ $\upalpha\in \Lie(A)$.

Let $\mathcal{P}$ be the set of real loci of the
standard $\mathbb{Q}$-parabolic subgroups of $\SL_{\mathtt{d}}$. Then $\mathcal{P}$ consists of all groups of upper-triangular block matrices with real entries and determinant $1$.

Let $T$ be the full diagonal subgroup of $G$ (so $A$ is the connected component 
of the identity matrix in $T$). 
For any $P\in\mathcal{P}$ let $W(T,P)=N_P(T)/C_{P}(T)$, where $N_P(T)$ and $C_{P}(T)$ are the normalizer and centralizer of $T$ in $P$, respectively.
We define $W=W(T,G)$ to be the Weyl group of $G$. 
We consider the right action of $W$ on $T$ by conjugation, where the action of $w=nC_{G}(T)$ on $a\in T$ is denoted by $a^{w}$ and is given by $a^{w}\coloneqq n^{-1}an$, independently of the representative $n$. 
In our settings, 
$W$ is isomorphic to the symmetric group $S_{\mathtt{d}}$,
and
the action of $\sigma\in S_{\dtt}\cong W$ on $A$ reads as 
\[\diag(a_1,\ldots,a_{\dtt})^{\sigma}=\diag(a_{\sigma_1},\ldots,a_{\sigma_\dtt}).\]
Note that $W$ acts on $\Lie(A)$ by permutations as well.

We consider quotients of $W$ by two different subgroups. First, in the case that there are multiplicities in the multiset of eigenvalues of $\att_{t}$, different elements in the Weyl group can conjugate $\att_{t}$ to the same element. To compensate for this it will be convenient to divide $W$ by 
$\stab_{W}(\att)$, i.e.\ the 
stabilizer in $W$ of the
time-one map $\att\coloneqq \att_{1}=\exp(\upalpha)$.
Next, for a parabolic subgroup $P\in\mathcal{P}$, we also divide $W$ by the subgroup $W(T,P)$ of permutations which preserve the block structure of $P$, in the sense that $\sigma,\tau\in W$ are identified in $W/W(T,P)$
if their actions on any $a\in A$ are the same up to a permutation which preserves the multiset of values of $a$ in each block of $P$. 
Finally, we define the double-quotient 
\[W_{P,\att}=\stab_{W}(\att)\backslash W/W(T,P).\]
Explicitly, let $P$ be the standard parabolic group with block sizes $(n_1,\ldots, n_{k})$. For any $0\leq i\leq k$, let $m_i=\sum_{l=1}^{i}n_l$. Then $\sigma,\tau\in S_{\dtt}\cong W$ are identified in $W_{P,\att}$ if and only if $\{\upalpha_{\sigma_{j}}\}_{j=m_{i-1}+1}^{m_{i}}=\{\upalpha_{\tau_{j}}\}_{j=m_{i-1}+1}^{m_{i}}$ as multisets, for all $1\leq i\leq k$.

We refer to $W_{P,\att}$ as the set of \textit{orientations} for $P$ and $\att$.
For $w\in W$, we let $[w]_P\in W_{P,\att}$ stand for the double-coset in $W_{P,\att}$ corresponding to $w$ (since $\att$ is considered fixed, we omit it from this notation).
We consider the sets $W_{P,\att}$ for different $P\in\mathcal{P}$ as disjoint.

\medskip 

Consider any subgroup $H_{S}\subseteq G$ of the form \[H_{S}=(I+\bigoplus_{(i,j)\in S} U_{ij}) \cap G\] for some set $S\subseteq\{1,\ldots,\dtt\}^{2}$, where $U_{ij}$ is the set of matrices with zeros at all entries except possibly the $(i,j)$'th entry. This includes the standard parabolic subgroups of $G$. 
For such a subgroup $H_{S}$, and for any diagonal matrix $a=\exp(\alpha)\in A$, where $\alpha=\diag(\alpha_1,\ldots,\alpha_{\mathtt{d}})\in \Lie(A)$, we define the entropy of $a$ on $H_{S}$ as the sum of (positive) Lyapunov exponents
\begin{align*}
    &h(H_S,a)=\sum_{(i,j)\in S} (\alpha_i-\alpha_j)^{+}, 
\end{align*}
where $z^{+}=\max(z,0)$ for any $z\in\mathbb{R}$.

Next, for a parabolic subgroup $P\in\mathcal{P}$, we let $A_P<A$ be the identity component of the centralizer of the Levi part of $P$, or more concretely the subgroup of $P$ consisting of block scalar matrices on the diagonal, with positive entries. 
Then, for any $P\in\mathcal{P}$ and $[w]_{P}\in W_{P,\att}$ we define 
the projection of $\upalpha^{w}$ from $\Lie(A)$ to $\Lie(A_P)$ by
\[\pi_P(\upalpha^{w})=\frac{1}{| W(T,P)|}\sum_{u\in W(T,P)}\upalpha^{wu}.\]
Then,
for $\phi\in \Lie(A)^{\ast}$, we give the notation
\begin{equation}\label{eq:1.1}(h-\phi)([w]_P)=h(P,\att^{w})-\phi(\pi_P(\upalpha^{w})),\end{equation}
where $\att$ stands for the time-one map of $\att_{\bullet}$.
Note that $h(P,\att^{w})$ and $\pi_{P}(\upalpha^{w})$ depend only on $[w]_{P}$.

\subsection{Cusp neighborhoods}\label{subsec:1.3}
We will present some of our results using the reductive Borel-Serre (RBS) compactification of $G/\Gamma$. General descriptions of several related compactifications of $G/\Gamma$ can be found in the book \cite{borel2006compactifications} by Borel and Ji. Without getting into details (see~\S\ref{subsec:8.2} for more), this compactification may be described as the disjoint union \[\overline{G/\Gamma}^{\RBS}=\coprod_{P\in\mathcal{P}}e_{\infty}(P)\] where $e_{\infty}(G)=G/\Gamma$ is an open dense subset.

Before stating our results, we now briefly discuss our definition for cusp regions in $G/\Gamma$, where further details can be found in~\S\ref{subsec:structure}. 
For $x\in \SLdRZ$, consider 
the set of indices $I$ where $i\in I$ if and only if the ratio $\frac{\lambda_i(x)}{\lambda_{i+1}(x)}$ between adjacent successive minima of Minkowski is smaller than $\delta$ (the successive minima are, up to constants, the heights of $x$ in the different directions of the Siegel domain).
These indices determine a unique flag of $\mathbb{R}^{\dtt}$ for $x$, with $I$ the set of dimensions of the flag. It also determines a parabolic subgroup $P$ which is the stabilizer of the flag. Lastly, the flag will typically have some orientation, 
corresponding to an element $[w]_P\in W_{P,\att}$. Then $\cusp_{\delta}(P,[w]_P)$ is defined as the set of points with large jumps corresponding to $P$, and orientation $[w]_P$.
It is also contained in the intersection with $G/\Gamma$ of a neighborhood of $e_{\infty}(P)$ in $\overline{G/\Gamma}^{\RBS}$.
Note that these regions are bounded if and only if $P=G$ (see Proposition~\ref{proposition:3.15}). 

\subsection{Results}\label{sec:introduction}

Our main result is an upper bound on the entropy of an invariant probability measure, with contributions from the regions of the cusp corresponding to the different parabolic subgroups and Weyl group elements as in~\S\ref{subsec:1.3}. These contributions include not only the entropy of each part of the cusp, but also the addition of a linear functional $\phi\in\Lie(A)^{\ast}$ as in Equation~\eqref{eq:1.1}. 

The inclusion of the linear functional is the main idea of this paper. As will be indicated in Corollaries~\ref{corollary:1.4}-\ref{corollary:1.5}, it allows to compute the entropy in the cusp in a simple manner, by a maximum over finitely many terms, without having to account for all of the different trajectories in the cusp. Given our results, one only needs to carefully choose the linear functional to use, which we do in a follow-up paper.

To add the linear functional we need to make sure that it would average out over any trajectory which starts and ends at a given compact set. This requires us to `declare' very carefully in which region of the cusp the trajectory is located, for every point in time (or alternatively, to which boundary component $e_{\infty}(P)$ the trajectory is close), a process we call `coding'. 
Finding a suitable coding is a main idea of this work. 
This coding yields a finite partition of $G/\Gamma$ as in the following Theorem~\ref{theorem:1.1new}, parametrized by the parabolic subgroups and the Weyl group elements, whose every element contributes a different amount to the entropy bound. The  partition is constructed using Theorem~\ref{theorem:1.1}.
\begin{theorem}\label{theorem:1.1new}
For all $\epsilon>0$, there is an open cover $\{U_{P}\}_{P\in\mathcal{P}}$ of $\overline{G/\Gamma}^{\RBS}$ satisfying
\[\overline{U_P}\cap e_{\infty}(Q)=\emptyset \qquad \text{for all $Q\in\mathcal{P}$ with $P\not\subseteq Q$}
\]
so that the following holds.
For all $\att\in A$ and $\phi\in \Lie(A)^{\ast}$ there 
 is a partition
\[\{V_{P,[w]_P}:\ P\in\mathcal{P},\ [w]_P\in W_{P,\att}\}\] 
 of $G/\Gamma$ 
with $V_{P,[w]_P}\subseteq U_{P}$ so that
\[h_{\mu}(\att)\leq \sum_{P\in \mathcal{P}}\sum_{[w]_{P}\in W_{P,\att}}\mu(V_{P,[w]_P})\cdot(h-\phi)([w]_P)+D_{\att,\phi}\epsilon\]  
for any $\att$-invariant probability measure $\mu$ on $G/\Gamma$,
where $D_{\att,\phi}>0$ is some constant which depends only on $\att$ and $\phi$.
\end{theorem}

Next, we take the limit of our entropy bound as the mass escapes to the cusp. 
We consider the entropy in the cusp, defined by
\[h_{\infty}(\att)=\sup\left\{\limsup_{i\to\infty} h_{\mu_i}(\att):\ \mu_i\in M_{1}(G/\Gamma)\ \att\text{-invariant},\ \mu_i\rightharpoonup 0\right\},\]
where the notation $M_{1}(Y)$ is introduced for the set of probability measures over some space $Y$.
We can also try to be more precise and bound the entropy in the cusp corresponding to the different parts of the cusp.
We suggest the following definition for this notion, using the reductive Borel-Serre (RBS) compactification of $G/\Gamma$ discussed in~\S\ref{subsec:1.3}. 
For $P\in\mathcal{P}$ we consider the limit of measures in the RBS compactification and define
\[h_{\infty,P}(\att)=\sup\left\{\limsup_{i\to\infty} h_{\mu_i}(\att):\ \mu_i\in M_{1}(G/\Gamma)\ \att\text{-invariant},\ \mu_i\rightharpoonup \nu\in M_{1}(e_{\infty}(P))\right\}.\]

Using the RBS compactification, we can bound the entropy in the limit.
\begin{theorem}\label{corollary:1.3}
Let $(\mu_i)_{i=1}^{\infty}\subset M_{1}(G/\Gamma) $ be a sequence of $\att$-invariant probability measures which converges (in the weak-$\star$ topology) to some $\nu\in M_{1}(\overline{G/\Gamma}^{\RBS})$.
Then
\[\limsup_{i\to\infty}h_{\mu_i}(\att)\leq \sum_{P\in\mathcal{P}}\nu(e_{\infty}(P))\max_{[w]_P\in W_{P,\att}}(h-\phi)([w]_P).\]
\end{theorem}

Using Theorem~\ref{corollary:1.3}, we can deduce upper bounds for $h_{\infty}(\att)$ and $h_{\infty,P}(\att)$. 
These bounds depict in them the main idea of this paper,
which is that the entropy of the cusp may be bounded by a maximum over the different (finitely many) Weyl elements. It is, of course, meaningful only if $\phi$ is chosen wisely.
\begin{corollary}\label{corollary:1.4}
\[h_{\infty}(\att)\leq \max_{P\in\mathcal{P}\smallsetminus\{G\}}\max_{[w]_P\in W_{P,\att}}\Big(h(P,\att^w)-\phi(\pi_P(\upalpha^{w}))\Big).\]
\end{corollary}
\begin{corollary}\label{corollary:1.5}
For all parabolic subgroups $P\in \mathcal{P}$,
\[h_{\infty,P}(\att)\leq \max_{[w]_P\in W_{P,\att}}\Big(h(P,\att^w)-\phi(\pi_P(\upalpha^{w}))\Big)\]
\end{corollary}

\begin{example*}
To demonstrate the use of our method, consider the simplest case where $X=\SL_{2}(\mathbb{R})/\SL_{2}(\mathbb{Z})$ and $\att_{t}=\exp(t/2,-t/2)$. In~\cite{einsiedler2011distribution}, it was shown that $h_{\infty}(\att)\leq\frac{1}{2}$. This bound was obtained essentially by arguing that when a trajectory goes through the cusp for a long period of time, while starting and ending roughly at the same height (namely, in the same fixed compact set), it accumulates entropy $1/2$ on average. The motivation behind was that the trajectory accumulates $0$ entropy going up the cusp and entropy $1$ going down, so the result $h_{\infty}(\att)\leq \frac{1}{2}$ is obtained by understanding that such a trajectory spends equal time in each of the two orientations of the cusp.

To see this result using our method, let $\phi=\frac{1}{2}\psi_{1}\in \Lie(A)^{\ast}$, where $\psi_1\in \Lie(A)^{\ast}$ is the root defined by $\psi_1(\diag(h_1,h_2))=h_1-h_2$.
In this case, the group $P$ of upper triangular matrices
is 
the only standard parabolic subgroup except for $G$.
Then, for $[w]_{P}\in W_{P_,\att}$ any of the two elements of $W_{P,\att}$, we have \[(h-\phi)([w]_{P})= \frac{1}{2}.\] 
We see that including $\phi$ in fact averages the entropy between the two directions, so the bound $\frac{1}{2}$ for $h_{\infty}(\att)$ follows immediately.
One should note that in this particular $\dtt=2$ case, it turns out that adding the linear functional does not make a material change to the bound for $h_{\mu}(\att)$ we obtain from Theorem~\ref{theorem:1.1new}, taking into account the observation that the measure of the upwards and downwards cusp regions are roughly the same.
However, in the general high-rank case, understanding the relations between the masses of the different cusp regions is a complicated task, since there are non-obvious restrictions on what sequences of orientations are possible for trajectories in the cusp in this case. This is where using the auxiliary linear functional can simplify the question considerably.
\end{example*}

As we mentioned, adding the linear functional is the main idea of this paper, as it allows to compute the entropy in the cusp in a simpler way. Still, even the case $\phi=0$ gives an interesting (and simpler) bound for the entropy of measures. In this case, where  a linear functional does not have to be included in the computations, constructing a coding is an easier task (see~\S\ref{sec:linear functionals} for details) and so we give the entropy bound using the explicit partition by the cusp regions as in~\S\ref{subsec:1.3}.
\begin{theorem}\label{corollary:1.2}
For all $\delta>0$ small enough,
\begin{align*}
h_{\mu}(\att)&\leq
\sum_{\substack{P\in\mathcal{P}}}\sum_{[w]_{P}\in W_{P,\att}}\mu(\cusp_{\delta}(P,[w]_{P}))\cdot h(P,\att^{w})+C_{\att}
\frac{1}{|\log\delta|^
{1/2}},
\end{align*}
for some constant $C_{\att}>0$ which depends only on $\att$.
\end{theorem}

In a subsequent paper~\cite{mor2022b}, we utilize the method constructed in this current paper, to compute upper bounds for the entropy of the cusp in various cases.
We are able to bound $h_{\infty}(\att)$, $h_{\infty, B}(\att)$ for $B$ the Borel subgroup, and $h_{\infty,P}(\att)$ for $P$ any maximal parabolic subgroup. 
Furthermore, in another paper~\cite{mor2022c} we prove that these bounds are in fact tight. We highlight these bounds here without a proof. Of these results, Item~\ref{item:1.5_2} is the result we mentioned has some similarities with Solan's work~\cite{solan2021parametric}.
In the following theorem, $P_{k}$ stands for the (maximal parabolic) group of upper triangular block matrices with two blocks, the first one being of size $k$ and the other of size $\dtt-k$.
\begin{theorem}[\cite{mor2022b,mor2022c}]
\begin{enumerate}
\item\label{item:1.5_1}

Assume without loss of generality that $\upalpha_{i}\geq \upalpha_{i+1}$ for all $1\leq i\leq \dtt-1$. For all $1\leq k\leq \dtt/2$, let $m_k$ be the minimal integer such that \[\sum_{i=1}^{k} \upalpha_{m+2(i-1)}\geq0\geq\sum_{i=1}^{k}\upalpha_{m+2(i-1)+1}.\]
Then
    \[h_{\infty,P_{k}}(\att)=h_{\infty,P_{\dtt-k}}(\att)=
h(G,\att)-k\sum_{i=1}^{m_k}\upalpha_{i}-\sum_{i=1}^{k-1}(k-i)(\upalpha_{m_k+2i-1}+\upalpha_{m_k+2i}).\]

\mathitem\label{item:1.5_2}
\[h_{\infty}(\att)= h(G,\att)-\sum_{i=1}^{\dtt}\upalpha_{i}^{+}.\]
This 
value
is
also
equal to 
$h_{\infty,P_1}(\att)$ and $h_{\infty,P_{\dtt-1}}(\att)$
as
in Item~\ref{item:1.5_1}.
\item
For all $\att$,
\[h_{\infty,B}(\att)\leq\frac{1}{2}h(G,\att).\]
Furthermore, if all the eigenvalues of $\att$ are of multiplicity $1$, or if there are only two distinct eigenvalues for $\att$, then \[h_{\infty,B}(\att)=\frac{1}{2}h(G,\att)\]
\end{enumerate}
\end{theorem}
\begin{remark}
We believe that the upper bound for $h_{\infty,B}(\att)$ is tight for any $\att$.
\end{remark}

\subsection{Acknowledgements} This paper is a part of the author's PhD studies in The Hebrew University of Jerusalem. I would like to thank my advisor, Prof.\ Elon Lindenstrauss, to whom I am grateful for his guidance and support throughout this work.

\section{Linear subspaces}\label{section:linear subspaces}
As will be explained in \S\ref{section:cusp}, the structure of the cusp in $X=\SL_{\dtt}(\mathbb{R})/\SL_{\dtt}(\mathbb{Z})$ is closely related to rational linear subspaces of $\mathbb{R}^{\dtt}$. Specifically, the cusp may be divided to different regions, where the region containing some lattice $x\in X$ is coded by the unique flag of $x$-rational subspaces of $\mathbb{R}^{\dtt}$, of small covolume, for the various dimensions for which they exist. It is a key ingredient in this work.
Before explaining this relation in full, we set in this section the necessary preliminaries regarding linear subspaces, and particularly describe the (simple) dynamics of the $\att_{\bullet}$-action on the Grassmanian manifold.
\subsection{Metrics on $\Gr(l,\dtt)$ and $\bigwedge^{l}\mathbb{R}^{\dtt}$}
When working with linear spaces, we usually take one of two approaches. The first is the trivial, which is considering them as subsets of $\mathbb{R}^{\dtt}$. 
The second is considering them as elements of the projectivization of the exterior algebra $\bigwedge^{l}\mathbb{R}^{\dtt}$, namely as elements of $\mathbb{P}(\bigwedge^{l}\mathbb{R}^{\dtt})$. 
In either case, we require metrics on these spaces. 
First, we define a metric $d_{\Gr}$ on $\Gr(l,\dtt)$ by
\begin{equation*}
    d_{\Gr}(V,W)=\sup_{v\in V,\ \|v\|=1}\inf\{\|v-w\|: w\in W\}.
\end{equation*}
Next, we define $\|\cdot\|_{\Ext}$ to be the norm on $\bigwedge^{l}\mathbb{R}^{\dtt}$ derived from the standard inner product $\langle \cdot, \cdot\rangle_{\Ext}$ on $\bigwedge^{l}\mathbb{R}^{\dtt}$, which on simple $l$-vectors is given by \[\langle v_1\wedge\cdots\wedge v_l, w_1\wedge\cdots\wedge w_l\rangle_{\Ext} =\det(\langle v_i,w_j\rangle_{\mathbb{R}^{\dtt}})_{i,j}.\]

It is important for us to note that these two notions of distance are related. The proof is omitted. 
\begin{proposition}\label{proposition:3.1}
Let $V,W\in\Gr(l,\dtt)$ be $l$-dimensional subspaces of $\mathbb{R}^{\dtt}$, and let $v_1,\ldots,v_l$ and $w_1,\ldots,w_l$ be bases for $V$ and $W$ respectively. Set $v=v_1\wedge\cdots\wedge v_l$ and $w=w_1\wedge\cdots\wedge w_l$. 
Then
\[d_{\Gr}(V,W)\leq \sqrt{l}\cdot \left\|\frac{v}{\|v\|_{\Ext}}-\frac{w}{\|w\|_{\Ext}}\right\|_{\Ext}\leq\sqrt{l}\cdot \frac{\|v-w\|_{\Ext}}{\|v\|_{\Ext}^{1/2}\|w\|_{\Ext}^{1/2}}.\]
\end{proposition}
\subsection{Dynamics on $\Gr(l,\dtt)$}
As was explained, we code the different parts of the cusp using linear subspaces. We use these subspaces to study the trajectory of points in the cusp under our flow $\att_{\bullet}$, so to do that we are first interested in the action of $\att_{\bullet}$ on these subspaces.

We require the following definition of an 
order on multisets.
\begin{definition}\label{def:1}
Let $E_1,E_2\subset\mathbb{R}_{>0}$ be multisets of the same finite cardinality. We say that $E_1\leq E_2$ if 
\[\prod_{x\in E_{1}}x \leq \prod_{x\in E_{2}} x.\]
\end{definition}

Furthermore, for an $\att_{\bullet}$-invariant subspace $M\subseteq\mathbb{R}^{\dtt}$, we let $\Eigen(\att|_M)$ be the multiset of eigenvalues of $\att|_{M}$ (where $\att\coloneqq \att_{1}$ is the time-one map). 
In this notation, we have $\Eigen(\att)=\{\exp(\upalpha_i)\}_{i=1}^{\dtt}$. 
Clearly $\Eigen(\att|_M)\subseteq\Eigen(\att)$ for all such $M$.

The following proposition explains the dynamics of the $\att_{\bullet}$-action on $\Gr(l,\dtt)$. 
\begin{proposition}\label{proposition:2.2}
For every $\att\in A$ and $\epsilon>0$, there is $C>0$ such that for all $V\in \Gr(l,\dtt)$ there exist finite sequences:
\begin{enumerate}
    \item $\{E_k\}_{k=1}^{m}$, where $E_k\subseteq\Eigen(\att)$ as multisets,  and $E_k$ is of cardinality $l$, 
    such that $E_1<\ldots< E_m$
    \item $\{W_k\}_{k=1}^{m}$, where $W_k\in \Gr(l,\dtt)$ is an $\att_{\bullet}$-invariant subspace with $\Eigen(\att|_{W_k})=E_k$
    \item $\{J_k\}_{i=1}^{m}$, where $J_k$ is an open interval (possibly an open ray) in $\mathbb{R}$, such that the intervals $\{J_k\}_{i=1}^{m}$ are mutually disjoint and their union covers $\mathbb{R}$ up to a set of Lebesgue measure at most $C$
\end{enumerate}
such that $d_{\Gr}(\att_{t}V,W_k)<\epsilon$ for all $t\in J_k,\ 1\leq k\leq m$.
\end{proposition}

\begin{remark}\label{remark:2.2}
Let $\epsilon_0\in(0,1)$ be small enough so that $\att$-invariant subspaces of $\mathbb{R}^{\dtt}$ of the same dimensions and with different multisets of eigenvalues, are at least $2\epsilon_0$ apart from each other with respect to $d_{\Gr}$, for any $\att\in A$. 
It would also be convenient, for technical reasons, to fix $\epsilon_0$ so that the Frobenius and the right-invariant metrics are equivalent on an $\epsilon_0$ neighborhood of the identity (see~\S\ref{subsec:metrics_on_SL}).
We will often turn to Proposition~\ref{proposition:2.2} with this fixed $\epsilon_0$ rather than vary $\epsilon$. In particular, the sets $\{E_k\}_{k=1}^{m}$ would be uniquely defined for a given $V$.
\end{remark}
\begin{remark}\label{remark:2.3}
We may extend the intervals $\{J_k\}_{k=1}^{m}$ to be as large as possible such that they still satisfy the proposition, hence we may treat the intervals as uniquely defined.
\end{remark}

\begin{proof}
Let $\{e_i\}_{i=1}^{\dtt}$ be the standard basis of $\mathbb{R}^{\dtt}$. 
Let $I$ be the set of strictly increasing $l$-tuples of elements of $\{1,\ldots,\dtt\}$, i.e.\ \[I=\{(i_1,\ldots,i_l):\ 1\leq i_1<\cdots<i_l\leq \dtt\}.\]
For $\mathbf{i}=(i_1,\ldots,i_{l})\in \{1,\ldots,\dtt\}^{l}$, let $u_{\mathbf{i}}=e_{i_1}\wedge\cdots\wedge e_{i_l}$, so that $\{u_{\mathbf{i}}\}_{\mathbf{i}\in I}$ is the standard basis for $\bigwedge^{l}\mathbb{R}^{\dtt}$.

Consider $V$ as an element of the projectivization of $\bigwedge^{l}\mathbb{R}^{\dtt}$, and write $V=\left[\sum_{\mathbf{i}\in I} c_{\mathbf{i}} u_{\mathbf{i}}\right]$, for some coefficients $c_{\mathbf{i}}\in\mathbb{R}$, where $[x]$ is the notation for the equivalence class of $x\in\bigwedge^{l}\mathbb{R}^{\dtt}$ in $\mathbb{P}(\bigwedge^{l}\mathbb{R}^{\dtt})$. 
Let $v=\sum_{\mathbf{i}\in I} c_{\mathbf{i}} u_{\mathbf{i}}$.

Each basis element $u_{\mathbf{i}}$ has an eigenvalue $\exp(\sum_{j=1}^{l} t \upalpha_{i_j})$ for the action of $\att_{t}$. For `most' actions $\att_t$, different basis elements have different eigenvalues, but it is not necessarily the case. So we group the basis elements according to their eigenvalues, that is write $I=\bigsqcup_{j=1}^{n} I_j$ where for any $1\leq j\leq n$, the vector $u_{\mathbf{i}}$ has the same $\att_{t}$-eigenvalue $\exp(t\beta_j)$ for all $\mathbf{i}\in I_j$. We write this union in such a way that the finite sequence $(\beta_j)_{j=1}^{n}$ is strictly increasing.
We may then write $v=\sum_{j=1}^{n} v_j$, where  $v_{j}=\sum_{\mathbf{i}\in I_j}c_{\mathbf{i}}u_{\mathbf{i}}$.
Then,  
in these notations,
$\att_{t} v=\sum_{j=1}^{n} \exp(t\beta_j) v_j$.

Note that $\{u_{\mathbf{i}}\}_{\mathbf{i}\in I}$ is an orthonormal basis for $\bigwedge^{l}\mathbb{R}^{\dtt}$, so in particular $\{v_j\}_{j=1}^{n}$ are orthogonal.
It follows that 
\begin{equation}\label{eq:2.0}\|\att_{t} v\|_{\Ext}^{2}=\sum_{j=1}^{n}\exp(2t\beta_j)\|v_j\|_{\Ext}^2.\end{equation}
Consider $\|\att_{t} v\|_{\Ext}^{2}$ as a function of $t\in\mathbb{R}$. It is just a linear combination of finitely many exponential functions, with different logarithmic derivatives, so it is clear that 
there are mutually disjoint intervals $J_1,\ldots,J_{m}$ which cover $\mathbb{R}$ up to Lebesgue measure $\ll_{\att} |\log\epsilon|$, and a strictly increasing finite sequence $(j_k)_{k=1}^{m}$, such that
\begin{equation}\label{eq:2.1}
    \sum_{j\not=j_k}\|v_{j}\|_{\Ext}^2 \exp(2t\beta_j)<\epsilon^2\cdot\|v_{j_k}\|_{\Ext}^2 \exp(2t\beta_{j_k})
\end{equation}
for all $1\leq k\leq m$ and all $t\in J_k$. In other words, on each interval there is a unique dominant component of the sum in the RHS of Equation~\eqref{eq:2.0}.
It follows immediately that 
\begin{equation}\label{eq:2.2}
    \frac{\|\att_{t}v-\exp(t\beta_{j_k})v_{j_k}\|_{\Ext}}{\|\att_{t}v\|_{\Ext}^{1/2}\|\exp(t\beta_{j_k})v_{j_k}\|_{\Ext}^{1/2}}<\epsilon.
\end{equation}

Note that $v_j$ is an eigenvector for the action of $\att_{t}$.
Then we would have liked to use Proposition~\ref{proposition:3.1} here to deduce that $\att_{t}v$ is close to an invariant subspace corresponding to $v_{j_k}$. However, $[v_{j_k}]$ is not necessarily an element of the Grassmanian as it may not satisfy the Pl\"ucker relations. We therefore perturb $v_{j_k}$ by a small amount, as follows. 

Let $D=\dim \bigwedge ^{l} \mathbb{R}^{\dtt}$.
For an element $[x]\in \mathbb{P}(\bigwedge^{l}\mathbb{R}^{\dtt})$, let $\iota([x])\in \mathbb{R}^{D}$ be the coordinates vector of $x$ relative to the standard basis $\{u_{\mathbf{i}}\}_{\mathbf{i}\in I}$ (with respect to some fixed ordering of the basis elements, say first the elements $u_{\mathbf{i}}$ for $\mathbf{i}\in I_1$, then for $\mathbf{i}\in I_2$ and so on). It is well defined only up to scalar, i.e.\ a normalization and a sign. For the normalization, we choose the sum of squares of $\iota([x])$ being equal to one.  For the sign, we may for instance choose the first non-zero component of $\iota([x])$ to be positive. 
With a slight abuse of notation, we may define \[\iota^{-1}:\ \mathbb{R}^{D}\smallsetminus\{0\}\to \mathbb{P}(\bigwedge\nolimits^{l}\mathbb{R}^{\dtt})\] sending $\{c_{\mathbf{i}}\}_{\mathbf{i}\in I}$ to $[\sum_{\mathbf{i}\in I} c_{\mathbf{i}}u_{\mathbf{i}}]$. It is the inverse of $\iota$ on its image, but is defined over all $\mathbb{R}^{D}\smallsetminus\{0\}$.
Moreover, let \[\nu:\ \mathbb{R}^{|I_{j_k}|}\smallsetminus\{0\}\to \Span\{\iota([u_{\mathbf{i}}])\}_{\mathbf{i}\in I_{j_k}}\subseteq\mathbb{R}^{D}\]
be the natural isomorphism defined by extending the coordinates of a vector in $\mathbb{R}^{|I_{j_k}|}$ by adding zeros before and after, so that the resulting $D$-tuple lies in $\Span\{\iota([u_{\mathbf{i}}])\}_{\mathbf{i}\in I_{j_k}}$.

We consider the Pl\"ucker relations as functions $f_1,\ldots,f_p$, where for any $1\leq q\leq p$ the function $f_q:\ \mathbb{R}^{D}\to\mathbb{R}$ is defined on the coordinates vector $\iota([x])$ for $x\in \bigwedge^{l}\mathbb{R}^{\dtt}$. Let $f:\ \mathbb{R}^{D}\to\mathbb{R}$ be defined by $f=\sum_{q=1}^{p} f_{q}^2$. It follows that $[x]\in \Gr(l,\dtt)$ if and only if $f(\iota([x]))=0$.
We consider the restrictions of $f$ to $\mathbb{R}^{|I_{j_k}|}$, namely define $\tilde{f}=f\circ\nu$.
Note that if $\tilde{f}(z)=0$ then $f(r\nu(z))=0$ for any $r\in \mathbb{R}$ (because $f_q$ is homogeneous), and so it follows that $\iota^{-1}(\nu(z))\in \Gr(l,\dtt)$.

Therefore, we are interested in studying the zero locus of $\tilde{f}$, which we denote by $Z_{\tilde{f}}$. 
Note that $\nu^{-1}(\iota([u_{\mathbf{i}}]))\in Z_{\tilde{f}}$ for all $\mathbf{i}\in I_{j_k}$ (because $[u_i]\in \Gr(l,\dtt)$, hence satisfies the relations), so in particular $Z_{\tilde{f}}\not=\emptyset$.
We consider $\tilde{f}$ as a real analytic function defined on the open subset $\mathbb{R}^{|I_{j_k}|}\smallsetminus\{0\}$ of $\mathbb{R}^{|I_{j_k}|}$.
Let $S$ be the unit sphere in $\mathbb{R}^{|I_{j_k}|}$. By the \L{}ojasiewicz inequality~\cite{lojasiewicz1965ensembles}, there is a constant $\alpha>0$ such that 
\begin{equation}\label{eq:2.4}\inf\{\|z-y\|:\ y\in Z_{\tilde{f}}\}\ll_{S}|\tilde{f}(z)|^{\alpha}\end{equation}
for any $z\in S$. As $S$ is constant throughout this proof, we will denote $\ll_{S}$ by $\ll$ for simplicity of notation.

Note that since $V=[v]$ is a subspace, $v$ satisfies the Pl\"ucker relations, i.e.\ $f(\iota([v]))=0$. 
Recall that the functions $f_q$ are given by an homogeneous relations involving sums of products of pairs of Pl\"ucker coordinates. As such, it is easy to see that perturbing the coordinates of a vector $h\in\mathbb{R}^D$ with $\|h\|=1$ by a difference of $\ll\epsilon$ changes the value of $f(h)$ by a difference of $\ll \epsilon$.  
Then, it follows from Equation~\eqref{eq:2.1} that $|f(\iota([v_{j_k}]))|\ll \epsilon$. Therefore $|\tilde{f}(z)|\ll \epsilon$, for $z=\nu^{-1}(\iota([v_{j_k}]))$.
By 
applying Equation~\eqref{eq:2.4} to this specific $z$,
there is some $y\in Z_{\tilde{f}}$ such that
$\|z-y\|\ll\epsilon^{\alpha}$, and so
\begin{equation}\label{eq:2.3}
\|\iota([v_{j_k}])-\nu(y)\|\ll \epsilon^{\alpha},
\end{equation}
i.e.\ the coordinates vector of $[v_{j_k}]$ is close to the coordinates vector of an element of the Grassmanian.

Consider $[v^{\prime}]=\iota^{-1}(\nu(y))$. Then $[v^{\prime}]\in \Gr(l,\dtt)$ because $y\in Z_{\tilde{f}}$.
We would like to use Equation~\eqref{eq:2.3} to bound $\|v_{j_k}-v^{\prime}\|_{\Ext}$, but it only bounds the difference of the coordinates vectors up to normalizations and signs. Still, by normalizing $v^{\prime}$ so that $\|v^{\prime}\|_{\Ext}=\|v_{j_k}\|_{\Ext}$, and perhaps by multiplying $v^{\prime}$ by $-1$,  
it 
follows that 
\[\|v_{j_k}-v^{\prime}\|_{\Ext}\ll \|v^{\prime}\|_{\Ext}\epsilon^{\alpha}.\]
Then, it follows from Equation~\eqref{eq:2.2} that for any $t\in J_{k}$
\begin{align*}\label{eq:2.100}
    \numberthis
    \|\att_{t}v-\exp(t\beta_{j_k})v^{\prime}\|_{\Ext}&\leq \|\att_{t}v-\exp(t\beta_{j_k})v_{j_k}\|_{\Ext}+\exp(t\beta_{j_k})\|v_{j_k}-v^{\prime}\|_{\Ext}\\&
    \ll (\epsilon+\epsilon^{\alpha})\|\att_{t}v\|_{\Ext}^{1/2}\|\exp(t\beta_{j_k})v^{\prime}\|_{\Ext}^{1/2} 
\end{align*}
and so it follows from Proposition~\ref{proposition:3.1} that \[d_{\Gr}(\att_{t}V,[v^{\prime}])\ll(\epsilon+\epsilon^{\alpha})\] as required. 

Note that $[v^{\prime}]$ is indeed an $\att_{t}$-invariant subspace, because $v^{\prime}\in\Span(\{u_{\mathbf{i}}\}_{\mathbf{i}\in I_{j_k}})$. So we let $W_{k}=[v^{\prime}]$ and $E_{k}=\Eigen(\att|_{W_k})$. 
This concludes the construction of $E_{k},W_{k},J_{k}$ as required in the proposition. Note that 
\[\exp(\beta_k)=\prod_{x\in E_{k}} x\]
so the sequence $(E_k)_{k=1}^{m}$ is indeed increasing with respect to the 
order `$<$'.
\end{proof}

We emphasize that a subspace being close to an invariant space $M$ with eigenvalues $E=\Eigen(\att|_M)$, implies that the action of $\att_{\bullet}$ changes the volume of the subspace roughly according to $E$.

\begin{corollary}\label{corollary:2.4}
Let $\epsilon>0$ be small enough.
Let $V\in\Gr(l,\dtt)$ and write $V=[v]$ for $v\in \bigwedge^{l}\mathbb{R}^{\dtt}$.
Let $E_k,J_k$, for $1\leq k\leq m$, be as in Proposition~\ref{proposition:2.2}. Let \[\beta_k=\log\Big(\prod_{x\in E_k} x\Big).\]
Then 
\begin{equation*}\|\att_{t}v\|_{\Ext}\asymp_{\epsilon,\att}\exp\Big(\sum_{k=1}^{m}\beta_{k}\cdot \Leb(J_k\cap [0,t])\Big)\cdot\|v\|_{\Ext}
\end{equation*}
for all $t\geq 0$, where $\Leb$ is the Lebesgue measure on $\mathbb{R}$. 
\end{corollary}

\begin{proof}
Let $a_k=\inf J_k$ and $b_k=\sup J_k$. 
Let $W_k=[v^{\prime}]$ as in the proof of Proposition~\ref{proposition:2.2}.
First, note that \[\|\att_{b_k}v\|_{\Ext}\asymp_{\epsilon,\att_{1}}\|\att_{a_{k+1}}v\|_{\Ext},\] 
for any $k<m$, because $|a_{k+1}-b_{k}|<C$ (and $C$ depends on $\epsilon$), and because $\att_{1}$ changes the norm by up to a bounded (from both sides) factor.
Because there are $\ll_{\dtt} 1$ many intervals $J_k$, it suffices to show that 
\begin{equation}\label{eq:2.102}
\|\att_{s_2}v\|_{\Ext}\asymp\exp(\beta_{k}\cdot (s_2-s_1))\cdot\|\att_{s_1}v\|_{\Ext}
\end{equation}
for any $s_1,s_2\in J_k$ with $s_2\geq s_1$.

Let $t\in J_{k}$. Note that
\begin{align*}\label{eq:2.103}
\numberthis
\|\att_{t}v\|_{\Ext}&\leq\|\att_{t}v-\exp(t\beta_{j_k})v^{\prime}\|_{\Ext}+
\|\exp(t\beta_{j_k})v^{\prime}\|_{\Ext}\\&
\leq C_0(\epsilon+\epsilon^{\alpha})\|\att_{t}v\|^{1/2}_{\Ext} \|\exp(t\beta_{j_k})v^{\prime}\|_{\Ext}^{1/2}+\|\exp(t\beta_{j_k})v^{\prime}\|_{\Ext}
\end{align*}
where $C_0\asymp 1$ is the implicit constant in Equation~\eqref{eq:2.100}.
Letting 
\begin{align*}
    &r=\|\att_{t}v\|_{\Ext}^{1/2}\\&
    b=C_0(\epsilon+\epsilon^{\alpha})\|\exp(t\beta_{j_k})v^{\prime}\|_{\Ext}^{1/2}\\&
    c=\|\exp(t\beta_{j_k})v^{\prime}\|_{\Ext}
\end{align*}
we see that Equation~\eqref{eq:2.103} is just a quadratic inequality
\[r^2\leq b r + c.\]
So it follows that
\begin{equation*}
    |r-b/2|\leq \frac{1}{2}\sqrt{b^2+4c}.
\end{equation*}
Then, assuming $\epsilon$ is small enough, it is easy to see that
\begin{equation}
r^{2}=\|\att_{t} v\|_{\Ext}\ll \|\exp(t\beta_{j_k})v^{\prime}\|_{\Ext}.
\end{equation}
Similarly,
\[\|\att_t v\|_{\Ext}\geq \|\exp(t\beta_{j_k})v^{\prime}\|_{\Ext}-\|\att_t v-\exp(t\beta_{j_k})v^{\prime}\|_{\Ext},\]
so $r^{2}\geq -br+c$,
from which we deduce
\[\|\att_{t} v\|_{\Ext}\gg \|\exp(t\beta_{j_k})v^{\prime}\|_{\Ext}.\]
All together we find
\begin{equation}\label{eq:2.101}\|\att_{t} v\|_{\Ext}\asymp \|\exp(t\beta_{j_k})v^{\prime}\|_{\Ext}.\end{equation}
Since Equation~\eqref{eq:2.101} holds for all $t\in J_{k}$, Equation~\eqref{eq:2.102} follows by substituting $t=s_1,s_2$ in Equation~\eqref{eq:2.101} and dividing. With it, the desired result follows as well.
\end{proof}

To conclude this section, we give the following definition of an orientation of a linear subspace, in the context of Proposition~\ref{proposition:2.2}.
\begin{definition}\label{definition:2}
Let
$V\in \Gr(l,\dtt)$.
Let $\{J_{k}\}_{k=1}^{m}$ and $\{E_{k}\}_{k=1}^{m}$ be the intervals of time and the multisets of eigenvalues, respectively, guaranteed from Proposition~\ref{proposition:2.2} for the constant $\epsilon_0$ set in Remark~\ref{remark:2.2}.
Then, if $0\in J_{k}$ for some $1\leq l\leq m$, we say that the linear subspace $V$ has \textbf{orientation} $E_{k}$ with respect to $\att$.
\end{definition}

\section{The structure of the cusp in $\SLdRZ$}\label{section:cusp}
In this section we explain the structure of the cusp, and set the required notations to state the results of this paper. This is the connection between the discussion of linear spaces in \S\ref{section:linear subspaces} and the space of lattices.

\subsection{The successive minima of Minkowski}
The discussion in this subsection is rather classical. First, we define some standard notions for lattices in $\mathbb{R}^{\dtt}$, and in particular for rational subspaces and their covolume.
We will often implicitly identify $X=\SL_{\mathtt{d}}(\mathbb{R})/\SL_{\mathtt{d}}(\mathbb{Z})$ and the space of unimodular lattices, by the map $\SLdRZ\ni[g]\mapsto g\mathbb{Z}^{\dtt}$.

\subsubsection{Definitions and basic concepts}\label{section:3.1.1}
\begin{definition}
Let $\Lambda$ be a lattice in $\mathbb{R}^{\dtt}$.
A subspace $V$ of $\mathbb{R}^{\dtt}$ is called $\Lambda$-rational if $V\cap\Lambda$ is a lattice in $V$.
In this case, we define $\covol_{\Lambda}(V)$ (the covolume of $V$ with respect to $\Lambda$) by the volume of a fundamental parallelepiped for $\Lambda\cap V$ in $V$, where the volume on $V$ is derived from the Euclidean volume.
\end{definition}
\begin{remark}\label{remark:3.1}
Given a $\Lambda$-rational subspace $V$ and a $\mathbb{Z}$-basis $\{v_1,\ldots,v_l\}$ for $\Lambda\cap V$, the set
\[\left\{\sum_{j=1}^{l}r_j v_j:\ r_j\in[0,1)\right\}\]
is a fundamental parallelepiped and so the covolume satisfies \[\covol_{\Lambda}(V)=\|v_1\wedge\cdots\wedge v_l\|_{\Ext}.\] 
\end{remark}

\begin{definition}
Let $x\in\SL_{\mathtt{d}}(\mathbb{R})/\SL_{\mathtt{d}}(\mathbb{Z})$. Then we define $\alpha_0(x)=1$, and for any $1\leq i\leq \dtt$ 
\[\alpha_i(x)=\min\{\textup{$\covol_{x}(V)$: $V\subset\mathbb{R}^{\dtt}$\textup{ an $i$-dimensional $x$-rational subspace}}\}.\]
\end{definition}
\begin{remark}
It is not difficult to show that the minimum in the definition of $\alpha_i(x)$ is indeed obtained.
\end{remark}
\begin{definition}
 Let $x\in\SL_{\mathtt{d}}(\mathbb{R})/\SL_{\mathtt{d}}(\mathbb{Z})$. Then the successive minima of Minkowski are defined, for any $1\leq i\leq\dtt$, by
    \[\lambda_i(x)=\min\{r\in\mathbb{R}:\ \exists v_1,\ldots,v_i\in x\textup{ linearly independent, with }\|v_j\|\leq r\textup{ for all }j\}.\]
\end{definition}
\begin{remark}
Clearly, $\lambda_1(x)\leq\cdots\leq\lambda_{\dtt}(x)$, for any $x\in\SLdRZ$.
\end{remark}

We will especially be interested in identifying when there is a big gap between the different successive minima of a lattice, so we give the following definitions.
\begin{definition}
Let $x\in\SLdRZ$.  For any $1\leq i\leq\dtt-1$, define
\[\eta_i(x)=\frac{\lambda_{i}(x)}{\lambda_{i+1}(x)}.\]
Moreover, for any $\epsilon>0$, let
\[\eta(x,\epsilon)=\{1\leq i\leq \mathtt{d}-1:\:\eta_i(x)<\epsilon\}.\]
\end{definition}

A classical theorem asserts that $\alpha_i(x)$ is roughly equal to the product of the successive minima $\lambda_1(x),\ldots,\lambda_i(x)$. A proof may be found for example in~\cite[Theorem 1.15, Corollary 1.16]{einsiedlerHomogeneousBook}.
\begin{theorem}\label{theorem:3.4}
Let $x\in\SL_{\dtt}(\mathbb{R})/\SL_{\dtt}(\mathbb{Z})$.
Then, for all $1\leq i\leq \dtt$,
\begin{equation*}
    \alpha_i(x)\asymp\lambda_1(x)\cdots\lambda_i(x).
\end{equation*}
Moreover, there exists a $\mathbb{Z}$-basis $v_1,\ldots,v_{\dtt}$ for $x$ such that $\|v_i\|\asymp \lambda_i(x)$ for all $1\leq i\leq \dtt$, and $\|\pi_{k}(v_k)\|\asymp \|v_k\|$ for all $2\leq k\leq \dtt$, where $\pi_k$ is the projection to $\Span\{v_1,\ldots,v_{k-1}\}^{\perp}$.
\end{theorem}
The following corollary follows using Remark~\ref{remark:3.1}.
\begin{corollary}\label{corollary:3.2}
Let $v_1,\ldots,v_{\dtt}$ be as in Theorem~\ref{theorem:3.4}.
Then for all $1\leq i\leq\dtt$, the subspace $W_{i}\coloneqq\Span\{v_1,\ldots,v_i\}$ is $x$-rational and satisfies $\covol_{x}(W_{i})\asymp\alpha_i(x)$. 
\end{corollary}
\begin{remark}
It follows that for all $x\in\SL_{\dtt}(\mathbb{R})/\SL_{\dtt}(\mathbb{Z})$ and for all $1\leq i\leq \dtt-1$,
\begin{equation*}
    \eta_i(x)\asymp\frac{\alpha_{i}^{2}(x)}{\alpha_{i-1}(x)\alpha_{i+1}(x)}.
\end{equation*}
In other words, informally, the function $i\mapsto \log\alpha_i(x)$ is convex, up to an additive constant, and the amount of deviation from linearity in its graph, at the point $i$, is roughly given by $\log\eta_i(x)$.
\end{remark}

Recall that by Mahler's compactness theorem, a closed subset $C\subset X$ is compact if and only if \[\inf\{\lambda_1(x):\:x\in C\}>0,\] i.e.\ there are no lattices in $C$ with arbitrarily short vectors. 
The following proposition shows that if $\lambda_1(x)$ is exceptionally small, i.e.\ a point is high up in the cusp, then there exists some jump $\eta_i(x)$ which is very small as well. Conversely, if $\eta_i(x)$ is small, then so is $\lambda_1(x)$. 
It follows that being high up in the cusp may be characterised by having a significant jump in the successive minima. 
Moreover, considering $\eta_1(x),\ldots,\eta_{\dtt-1}(x)$ instead of only $\lambda_1(x)$ gives more information about the reason that the point is high up in the cusp, as it characterises the existence of rational subspaces of small covolume, for some dimensions, as we will shortly see.

\begin{proposition}\label{proposition:3.4}
For all $1\leq i\leq \dtt-1$, the functions $\alpha_i$, $\lambda_1$ and $\eta_i$ satisfy
\begin{equation*}
    \lambda_1^{i}\ll\alpha_i\ll \eta_i^{i(\dtt-i)/\dtt}.
\end{equation*}
Moreover,
\begin{equation*}
    \min_{j=1,\ldots,\dtt-1}\eta_j\ll\lambda_{1}^{2/(\dtt-1)}.
\end{equation*}
\end{proposition}

\begin{proof}
Note that \[\alpha_i\asymp\lambda_1\cdots\lambda_i\geq\lambda_1^{i}\] so the first inequality of the first statement is clear.
Next, note that $\alpha_{\dtt}\equiv 1$, so
\[\frac{1}{\alpha_i}=\frac{\alpha_{\dtt}}{\alpha_i}\asymp\lambda_{i+1}\cdots\lambda_{\dtt}\geq\lambda_{i+1}^{\dtt-i}\]
so
\begin{equation*}
    \eta_i=\frac{\lambda_i}{\lambda_{i+1}}\gg\frac{(\lambda_1\cdots\lambda_i)^{1/i}}{\alpha_{i}^{-1/(\dtt-i)}}\gg\alpha_{i}^{\dtt/(i(\dtt-i))}
\end{equation*}
and therefore
\[\lambda_{1}^{i}\ll\alpha_i\ll\eta_{i}^{i(\dtt-i)/\dtt}\]
as desired.

For the second statement, let $i>1$. Note that
\begin{equation*}
    \lambda_i=\frac{\lambda_{i-1}}{\eta_{i-1}}=\cdots=\frac{\lambda_1}{\eta_1\cdots\eta_{i-1}}.
\end{equation*}
Then
\begin{equation*}
    \alpha_i\asymp\lambda_1\cdots\lambda_i=\frac{\lambda_{1}^{i}}{\eta_{1}^{i-1}\cdots\eta_{i-2}^{2}\eta_{i-1}}.
\end{equation*}

By substituting $i=\dtt$, we see that
\begin{equation*}
1\asymp\frac{\lambda_{1}^{\dtt}}{\eta_{1}^{\dtt-1}\cdots\eta_{\dtt-2}^{2}\eta_{\dtt-1}}\leq(\min_{i=1,\ldots,\dtt-1}\eta_{i})^{-\dtt(\dtt-1)/2}\lambda_{1}^{\dtt}
\end{equation*}
as desired.
\end{proof}

The importance of these definitions will be clear shortly. The key tool that we use is the following elementary result, which may be found for instance in~\cite[Lemma 5.6]{eskin1998upper}.
\begin{theorem}\label{theorem:3.8}
Let $\Lambda$ be a lattice in $\mathbb{R}^{\dtt}$, and let $L,M$ be $\Lambda$-rational linear subspaces.
Then
\[\covol_{\Lambda}(L)\covol_{\Lambda}(M)\geq \covol_{\Lambda}(L\cap M)\covol_{\Lambda}(L+M)\]
\end{theorem}

As we have seen in Proposition~\ref{proposition:3.4}, if $\eta_i(x)$ is small then $\alpha_i(x)$ is small as well, i.e.\ there is an $i$-dimensional rational subspace of small covolume. We now show that if $\eta_i(x)$ is small enough, then such small covolume subspace is unique, in the sense that all other subspaces are of much larger covolume.
\begin{proposition}\label{proposition:3.6}
Let $x\in X$ and let $1\leq i\leq\dtt-1$. Let $L$ be a $x$-rational subspace of dimension $i$ which satisfies $\covol_{x}(L)=\alpha_i(x)$.
Then for all $x$-rational subspaces $M$ of dimension $i$, such that $M\not=L$, the covolume satisfies
\begin{equation*}
    \covol_{x}(M)\gg\eta_i^{-1}(x)\alpha_i(x).
\end{equation*}
\end{proposition}

\begin{proof}
Let $L,M$ be $\Lambda$-rational subspaces of dimension $i$. 
Suppose that $L\not=M$, so there exists $k>0$ such that
\[\dim(L\cap M)=i-k,\ \dim(L+M)=i+k.\]

Then, using Theorem~\ref{theorem:3.8},
\begin{align*}
    \frac{\alpha_{i}^2(x)}{\covol_{\Lambda}(L)\covol_{\Lambda}(M)}&\leq\frac{\alpha_{i}^2(x)}{\covol_{\Lambda}(L\cap M)\covol_{\Lambda}(L+M)}\leq\frac{\alpha_{i}^2(x)}{\alpha_{i-k}(x)\alpha_{i+k}(x)}\\&
    \asymp\frac{\lambda_{i-k+1}(x)\cdots\lambda_{i}(x)}{\lambda_{i+1}(x)\cdots\lambda_{i+k}(x)}\leq\frac{\lambda_{i}^k(x)}{\lambda_{i+1}^k(x)}=\eta_{i}^{k}(x).
\end{align*}

Assume that $\covol_{\Lambda}(L)=\alpha_i(x)$. Then
\[\covol_{\Lambda}(M)\gg \eta_{i}^{-k}(x)\alpha_i(x).\]
\end{proof}
\begin{corollary}\label{corollary:3.7}
In Corollary~\ref{corollary:3.2} we saw the existence of a flag
\[\{0\}\leq W_1\leq\cdots\leq W_{\dtt-1}\leq \mathbb{R}^{\dtt},\]
where $W_i$ is an $i$-dimensional $x$-rational subspace, which satisfies
\[\covol_{x}(W_i)\asymp \alpha_i(x)\]
for all $i$.
It follows from Proposition~\ref{proposition:3.6} that there is some constant $0<\eta_0< 1$, which we fix for the rest of the paper, so that for all indices $i$ with $\eta_i(x)<\eta_0$, the subspace $W_i$ is defined uniquely, and it is the unique $x$-rational subspace of covolume $\alpha_i(x)$.
\end{corollary}

\begin{definition}\label{definition:6}
Let $x\in X$. Assume that $\eta_i(x)<\eta_0$ for some $i\in\{1,\ldots,\dtt-1\}$. Then, let $V_i(x)$ be the unique $i$-dimensional $x$-rational subspace of volume $\alpha_i(x)$.
\end{definition}

\subsubsection{Dynamics}
We use the results of~\S\ref{section:linear subspaces} with the definitions of~\S\ref{section:3.1.1} in order to study how $\alpha_i(x)$ and $V_i(x)$ change upon applying the flow $\att_{\bullet}$.
The following estimate is crude and not tight, but is still useful. We will later on improve on it, in Proposition~\ref{lemma:6.1}.
\begin{lemma}\label{lemma:3.11}
There is a constant 
$0\leq h\ll_{\att_{1}} 1$ 
such that
\[\covol_{\att_t x}(\att_t V)\leq \exp(h|t|) \covol_{x}(V)\]
for all $x\in\SLdRZ$, for all $x$-rational $l$-dimensional subspaces $V$, and for all  $t\in\mathbb{R}$, $1\leq l\leq \dtt$.
\end{lemma}

\begin{proof}
Let 
$h=\sum_{i=1}^{\dtt}\max(\upalpha_{i},0).$
Then it is clear that  
\begin{equation}\label{eq:}
\|\att_t v\|_{\Ext}\leq \exp(h|t|)\|v\|_{\Ext}
\end{equation}
for all $v\in\bigwedge^{l}\mathbb{R}^{\dtt}$, and for all $1\leq l \leq \dtt$, $t\in\mathbb{R}$. 

Now, to prove the claim, let $v\in \bigwedge^{l}\mathbb{R}^{\dtt}$ be a wedge product of a $\mathbb{Z}$-basis for $x\cap V$. Then $V=[v]$, $\att_{t} V=[\att_{t} v]$, and $\att_{t}v$ is a wedge product of a $\mathbb{Z}$-basis for $\att_{t} x\cap \att_{t} V$.
Then we get from Remark~\ref{remark:3.1} and  Equation~\eqref{eq:}
\[\covol_{\att_{t} x}(\att_{t} V)=\|\att_{t} v\|_{\Ext}\leq \exp(h|t|)\|v\|_{\Ext}=\exp(h|t|)\covol_{x}(V).\]
\end{proof}
\begin{corollary}\label{corollary:3.12}
Let $h$ be as in Lemma~\ref{lemma:3.11}. Then
\[\exp(-h|t|)\leq \frac{\alpha_l(\att_t x)}{\alpha_l(x)}\leq \exp(h|t|)\]
for all $x\in\SLdRZ$, $t\in\mathbb{R}$ and $1\leq l\leq \dtt$.
\end{corollary}
\begin{proof}
Let $V_1$ and $V_2$ be $x$-rational and $\att_{t} x$-rational subspaces, respectively, of dimension $l$, such that
\[\covol_{x}(V_1)=\alpha_{l}(x),\ \covol_{\att_t x} (V_2)=\alpha_{l}(\att_t x).\]
Then we apply Lemma~\ref{lemma:3.11} once for $x,V_1,t$ to obtain
\[\alpha_{l}(\att_t x)\leq \covol_{\att_t x}(\att_t V_1)\leq \exp(h|t|)\covol_{x}(V_1)=\exp(h|t|)\alpha_{l}(x) ,\]
and once for $\att_{t}x,V_2,-t$ to obtain
\begin{align*}
\alpha_{l}(x)&\leq\covol_{x}(\att_{-t} V_2)=\covol_{\att_{-t} (\att_t x)}(\att_{-t} V_2)\\&
\leq \exp(h|-t|)\covol_{\att_{t}x}(V_2)=\exp(h|t|)\alpha_{l}(\att_t x) .
\end{align*}
\end{proof}

The following simple proposition is essential to us, and asserts that when $\eta_l(x)$ is small, the unique small $x$-rational subspace of dimension $l$ does not change upon applying $\att_{\bullet}$. It will later on correspond to $x$ being high up in the cusp. 
\begin{proposition}\label{proposition:3.13}
Let $0<\delta<\eta_0$ and $1\leq l\leq \dtt-1$. 
Then for all $|t|\ll_{\att_1} |\log\delta|$ and for all $x\in \SLdRZ$ with $\eta_{l}(x)<\delta$, we have \[\eta_l(\att_t x)<\eta_0\] and \[V_l(\att_t x)= \att_t V_l(x).\]
\end{proposition}
\begin{proof}
Using Corollary~\ref{corollary:3.12}, note that
\[\eta_l(\att_t x)\asymp \frac{\alpha_l^2(\att_t x)}{\alpha_{l-1}(\att_t x)\alpha_{l+1}(\att_t x)}\leq \exp(4h|t|)\frac{\alpha_l^2(x)}{\alpha_{l-1}(x)\alpha_{l+1}(x)}\asymp\exp(4h|t|)\eta_l(x)\]
and so, since $\eta_l(x)< \delta$, the assertion $\eta_l(\att_t x)<\eta_0$ follows for all $t\ll\frac{1}{h}|\log\delta|$, as required.

Let $t\in\mathbb{R}$, and assume that $V_l(\att_t x)\not=\att_{t}V_l(x)$. then by Lemma~\ref{proposition:3.6},
\begin{equation}\label{eq:3.2}
\covol_{x}(\att_{-t}V_l(\att_t x))\gg \eta_l^{-1}(x)\covol_{x}(V_l(x))=\eta_{l}^{-1}(x)\alpha_l(x).
\end{equation}
On the other hand, by Lemma~\ref{lemma:3.11}, 
\begin{equation}\label{eq:3.3}
\covol_{x}(\att_{-t}V_l(\att_t x))\leq \exp(h|t|)\covol_{\att_{t}x}(V_l(\att_t x))=\exp(h|t|)\alpha_l(\att_{t}x),
\end{equation}
so together we get from Equations~\eqref{eq:3.2}-\eqref{eq:3.3}
\[\frac{\alpha_l(\att_t x)}{\alpha_l(x)}\gg(\eta_l(x)\exp(h|t|))^{-1}\]
and so, by Corollary~\ref{corollary:3.12}, $\exp(-2h|t|)\ll \delta$ so $|t|\gg\frac{1}{h}|\log \delta|$.
Therefore, a contradiction rises if $|t|\ll\frac{1}{h}|\log\delta|$, and we deduce $V_l(\att_t x)=\att_{t}V_l(x)$ for such $t$. 
\end{proof}

\subsection{The structure of the cusp of $\SLdRZ$}\label{subsec:structure}
In this section we define a natural structure of the cusp of $X=\SLdRZ$ which follows from the notions we introduced before. 

\subsubsection{$\mathbb{Q}$-parabolic subgroups and flags}
\label{subsec:3.3}

In~\S\ref{subsubsec:3.2.2}
we will give a partition of the cusp to regions characterised by the location in the cusp and the orientation of the lattices in them, as follows. 
First, we associate to a point $x\in X$ the set $I\subseteq \{1,\ldots,\dtt-1\}$ of places where the jumps in Minkowski's successive minima are significant, i.e.\ $\eta_i(x)$ is small.
Next, as we have seen in Definition~\ref{definition:6}, the set $I$ corresponds to a flag $\{V_{i}(x)\}_{i\in I}$ of exceptionally small $x$-rational subspaces of $\mathbb{R}^{\dtt}$, with $\dim V_{i}(x)=i$.
Then, if possible, we associate to each $V_{i}(x)$ its orientation $E_{i}\subset \Eigen(\att)$, in the sense of Definition~\ref{definition:2}. Then $\{E_{i}\}_{i\in I}$ is an increasing (with respect to $\subseteq$) sequence of multisets.

Note that these two pieces of information about a lattice, the set of jumps $I$ and the sequence of orientations $\{E_i\}_{i\in I}$, are precisely equivalent to the information given by indicating a pair of a parabolic subgroup $P\in\mathcal{P}$ and an element $[w]_{P}\in W_{P,\att}$. This correspondence is as follows. We will use it freely throughout this paper.
First, note that the set of standard parabolic subgroups $\mathcal{P}$ is in one-to-one correspondence with subsets $I\subseteq\{1,\ldots,\dtt-1\}$, where a set $I$ is mapped to the parabolic subgroup $P_I$, which is the stabilizer of flags of $\mathbb{R}^{\dtt}$, whose set of dimensions is $I$, with respect to an adapted basis. 
Concretely, $P_I$ is the subgroup of $G$ of block upper triangular matrices, with $I$ being the set of partial sums of the block sizes (not including $0$ or $\dtt$).
We will denote this 
correspondence by $\eta$, 
namely let $\eta(P)=I$ for the unique set $I$ so that $P=P_I$.

Next, consider the set \[\Eigen(\att)=\{\exp(\upalpha_1),\ldots,\exp(\upalpha_{\dtt})\}\] as a multiset, and let
$\{E_i\}_{i\in I}$ 
be an increasing (with respect to $\subseteq$) sequence of multisets, with each $E_i$ of cardinality $i$ (as a multiset), and all contained in $\Eigen(\att)$.
To this data we can attach an element $[w]_P$ of the relative Weyl group $W_{P,\att}$, uniquely determined by the requirement that 
for every $i \in I$ we have (as multisets)
\[
E_i = \{\exp(\upalpha_{w(1)}),\ldots,\exp(\upalpha_{w(i)})\}.
\]

With all of that in mind, we can now give the following definition for the orientation of a lattice. 
\begin{definition}
Let $x\in X$ and $P\in \mathcal{P}$. Assume $\eta_{l}(x)<\eta_0$ for all $l\in \eta(P)$.
Then we say that $x$ has orientation $[w]_{P}\in W_{P,\att}$ with respect to $P$ if for all $l\in \eta(P)$ the unique small $x$-rational subspace $V_{l}(x)$ has orientation \[\{\exp(\upalpha_{w(1)}),\ldots,\exp(\upalpha_{w(l)})\}\]
in the sense of Definition~\ref{definition:2}.
\end{definition}
\begin{remark}
Note that not all lattices have orientation, but as will be shown in Proposition~\ref{proposition:3.9}, most do.
Furthermore, recall that our notion of an orientation of a linear subspace in Definition~\ref{definition:2} was with respect to some constant $\epsilon_0$ which guarantees that subspaces with different orientations are far from each other, so in particular if the orientation of a lattice exists it is well defined.
\end{remark}

\subsubsection{Cusp regions}\label{subsubsec:3.2.2}
Let us define the cusp regions we use in this paper. 
First, we define the following partition of $X$ to regions based on jumps in the successive minima. Later we refine these regions according to the orientations of the lattices.

\begin{definition} For $P,Q\in\mathcal{P}$ with $Q\subseteq P$, and for $0<\delta<\delta^{\prime}<\eta_0$, let 
\begin{enumerate}
    \item $\cusp_{\delta}(P)=\left\{x\in X:\ \eta(x,\delta)=\eta(P)\right\}$
    \item $\cusp_{\delta,\delta^{\prime}}(P,Q)=\cusp_{\delta}(P)\cap\cusp_{\delta^{\prime}}(Q)$
    \item $\compact_{\delta}(X)=\{x\in X:\ \eta(x,\delta)=\emptyset$\}
\end{enumerate}
\end{definition}

\begin{remark}
We will mostly use these definitions for $P\not=G$, for which, as will be indicated in Proposition~\ref{proposition:3.15}, the sets $\cusp_{\delta}(P,Q)$ are unbounded. For the case $P=G$, we simply have $\cusp_{\delta}(G)=\compact_{\delta}(X)$. 
\end{remark}

The following are easy yet important characteristics of these sets.
\begin{proposition}\label{proposition:3.15}
Let $0<\delta<\delta^{\prime}<\eta_0$. Then
\begin{enumerate}
\item
The collection of sets
\[\{\compact_{\delta}(X)\}\cup\{\cusp_{\delta,\delta^{\prime}}(P,Q):\ Q\subseteq P\not=G\}\] is a partition of $X$.
\item
$\compact_{\delta}(X)$ is a compact subset of $X$. 
\item 
$\cusp_{\delta,\delta^{\prime}}(P,Q)$ is unbounded, for all $Q\subseteq P\not=G$.
\end{enumerate}
\end{proposition}
\begin{proof}
\begin{enumerate}
    \item The proof is immediate from the definitions.
    \item 
    Note that
    \[\min_{j=1,\ldots,\dtt-1}\eta_j(x)\geq\delta\] for any $x\in\compact_{\delta}(X)$.
    Then, by Proposition~\ref{proposition:3.4},
    \[\lambda_{1}(x)\gg \delta^{(\dtt-1)/2},\]
    and so the result follows from Mahler's compactness theorem.
    \item
    We will prove that $\cusp_{\delta}(P)$ is unbounded for $P\not=\ G$. The proof for the sets $\cusp_{\delta,\delta^{\prime}}(P,Q)$ is similar and is left to the reader.
    
    Consider
    \[\eta(P)=\{i_1,\ldots,i_k\},\ i_1 < \cdots< i_k.\]
    Let $i_0=0$ and $i_{k+1}=\dtt$,
   and
   \[h_n=\diag(\underbrace{z_1,\ldots,z_1}_{\text{$i_1-i_0$ times}},\ldots,\underbrace{z_{k+1},\ldots,z_{k+1}}_{\text{$i_{k+1}-i_{k}$ times}}),\]
    where
    \[z_m=\begin{cases}n^{-(k+1-m)} & 1\leq m\leq k \\ \prod_{s=1}^{k}z_{s}^{-(i_s-i_{s-1})/(i_{k+1}-i_k)} & m=k+1\end{cases}.\]
    Let $x_n=[h_n]\in X$ the lattice corresponding to $h_n$. Then it is easy to show that $h_n\in \cusp_{\delta}(P)$ for all $n\in\mathbb{N}$ large enough, and $\lambda_{1}(x_n)=n^{-k}\underset{n\to\infty}{\to}0.$
\end{enumerate}
\end{proof}

Next, we define regions in our partition of $X$ which correspond to lattices with a well defined orientation.
\begin{definition}\label{definition:8} Let $P\in\mathcal{P}$, $[w]_P\in W_{P,\att}$ and $0<\delta<\eta_0$. 
Then define
\[\cusp_{\delta}(P,[w]_P)=\Big\{x\in \cusp_{\delta}(P):\ \text{$x$ has orientation $[w]_P$ with respect to $P$}\Big\}.\]
\end{definition}
We further define the sets $\cusp_{\delta,\delta^{\prime}}(P,Q,[w]_{Q})$ of points $x\in \cusp_{\delta,\delta^{\prime}}(P,Q)$ with orientation $[w]_{Q}$ with respect to (the smaller group) $Q$. 
\begin{definition}\label{definition:9}
For $Q\subseteq P\in \mathcal{P}$ and $[w]_{Q}\in W_{Q,\att}$, let
\begin{align*}
    \cusp_{\delta,\delta^{\prime}}(P,Q,[w]_{Q})&
    =\cusp_{\delta}(P)\cap \cusp_{\delta^{\prime}}(Q,[w]_{Q}).
\end{align*}

\end{definition}

As we mentioned, not all lattices have a well defined orientation. We will 
use the following notations for 
these
parts of the cusp. 
\begin{definition}
For $\delta>0$ and $P\in\mathcal{P}$, let
\[\cusp_{\delta}(P,\emptyset)=\cusp_{\delta}(P)\smallsetminus\bigcup_{[w]_P\in W_{P,\att}}\cusp_{\delta}(P,[w]_P)\]
and
\[\cusp_{\delta}(X,\emptyset)=\bigcup_{P\in \mathcal{P}}\cusp_{\delta}(P,\emptyset).\]
Furthermore, for $Q\subseteq P$ and $\delta^{\prime}>\delta$, we define \[\cusp_{\delta,\delta^{\prime}}(P,Q,\emptyset)=\cusp_{\delta}(P)\cap\cusp_{\delta^{\prime}}(Q,\emptyset).\]
\end{definition}

The following proposition shows that
the
regions 
$\cusp_{\delta}(P,[w]_{P})$ 
of lattices with well defined orientations
cover a large portion of the space 
.
\begin{proposition}\label{proposition:3.9} 
Let $\mu$ be an $\att_{\bullet}$-invariant probability measure on $\SL_{\dtt}(\mathbb{R})/\SL_{\dtt}(\mathbb{Z})$. 
Let $0<\delta< \eta_0$.
Then 
\[\mu\left(\cusp_{\delta}(X,\emptyset)\right)\ll_{\att} \frac{1}{|\log\delta|}.\]
\end{proposition}

\begin{proof}
Note that $\cusp_{\delta}(G,\emptyset)=\emptyset$, because the trivial flag $0<\mathbb{R}^{\dtt}$ does have an orientation.
Fix some $P\in\mathcal{P}\smallsetminus\{G\}$ and let $x\in \cusp_{\delta}(P)$.
By applying Proposition~\ref{proposition:2.2} to $V_i(x)$, for all $i\in \eta(P)$, we see that there is a collection of $\ll 1$ intervals, whose union we denote by $L$, such that 
$\Leb(\mathbb{R}\smallsetminus L)\ll_{\att} 1$, and $\att_{t} V_{i}(x)$ has a well defined orientation for all $t\in L$ and $i\in \eta(P)$.

Note that by Proposition~\ref{proposition:3.13}, $V_i(\att_{t}x)=\att_{t}V_i(x)$ for all $t\leq \kappa|\log\delta|$, for some fixed $\kappa>0$ (which depends on $\att$), and for all $i\in\eta(P)$. Let $N=\lceil \kappa|\log\delta|\rceil$.
Therefore, if 
$\att_{n} x\in \cusp_{\delta}(P,\emptyset)$
for some $n\in [0,N-1]\cap\mathbb{Z}$ then there is some $i\in \eta(P)$ 
so that $\att_{n}V_{i}(x)$ has no well defined orientation,
hence $n\in \mathbb{R}\smallsetminus L$. 
However,  $\mathbb{R}\smallsetminus L$ contains $\ll_{\att} 1$ integer points (determined by the number of intervals and their length, both bounded by $\ll_{\att} 1$). So we deduce
\[0\leq f(x)\coloneqq\sum_{n=0}^{N-1}\mathbbm{1}_{\att_{-n}(\cusp_{\delta}(P,\emptyset))}(x)\leq \Big| (\mathbb{R}\smallsetminus L)\cap \mathbb{Z}\Big|\ll_{\att} 1,\]
for any $x\in\cusp_{\delta}(P)$.

Now, let $x^{\prime}\in\cusp_{\delta}(P^{\prime})$ for $P^{\prime}\not=P$. Assume $f(x^{\prime})\not=0$, then there is some minimal $0\leq m\leq N-1$ so that
$\att_{m}x^{\prime}\in \cusp_{\delta}(P)$.
Then it follows that $f(\att_{m}x^{\prime})\ll_{\att} 1$, and so
\begin{align*}0\leq f(x^{\prime})&=\sum_{n=0}^{m-1}\mathbbm{1}_{\att_{-n}(\cusp_{\delta}(P,\emptyset))}(x^{\prime})+f(\att_{m}x^{\prime})-\sum_{n=N}^{N+m-1}\mathbbm{1}_{\att_{-n}(\cusp_{\delta}(P,\emptyset))}(x^{\prime})\\&
\leq f(\att_{m}x^{\prime})\ll_{\att} 1.
\end{align*}
All together, we found that
\[0\leq f\ll_{\att} 1.\]

So it follows from $\att_{\bullet}$-invariance of $\mu$ that
\begin{align*}\label{eq:3.18}
\mu(\cusp_{\delta}(P,\emptyset))&=\frac{1}{N}\sum_{n=0}^{N-1}\mu(\att_{-n}(\cusp_{\delta}(P,\emptyset)))
=\frac{1}{N}\int_{X}\sum_{n=0}^{N-1}\mathbbm{1}_{\att_{-n}(\cusp_{\delta}(P,\emptyset))}d\mu\\&
\ll_{\att} \frac{1}{N}\int_{X}1d\mu\leq\frac{1}{\kappa|\log\delta|}.
\end{align*}
Then, by taking union over all $P\in\mathcal{P}$ we get
\[\mu(\cusp_{\delta}(X,\emptyset))=\mu(\bigcup_{P\in\mathcal{P}}\cusp_{\delta}(P,\emptyset))\ll_{\att} \frac{1}{|\log\delta|},\]
as required.
\end{proof}

\subsubsection{Additional definitions of cusp regions}
Our main cusp regions in this work are those defined in~\S\ref{subsec:structure}.
Still, in some cases throughout this paper 
it would be convenient to work simultaneously with another definition of the cusp regions, as follows.
For $P\in\mathcal{P}$ and $[w]_{P}\in W_{P,\att}$, let
\[\cusp_{\delta}^{+}(P)=\bigcup_{Q\subseteq P} \cusp_{\delta}(Q)=\{x\in X:\ \eta(P)\subseteq \eta(x,\delta)\},\]
and
\begin{align*}
\cusp_{\delta}^{+}(P,[w]_P)=\{x\in\cusp_{\delta}^{+}(P):\ \text{$x$ has orientation $[w]_{P}$ with respect to $P$}\}.
\end{align*}
Compared to Definition~\ref{definition:8} of $\cusp_{\delta}(P,[w]_P)$, we allow points in $\cusp_{\delta}^{+}(P,[w]_P)$ to have exceptionally small rational subspaces not only in dimensions $\eta(P)$ but possibly in others as well, but we only determine the orientation $[w]_P$ by those of dimensions $\eta(P)$.
Note that unlike our original regions, the regions $\cusp_{\delta}^{+}(P,[w]_P)$ are not necessarily mutually disjoint for different subgroups $P$. We will only use them to simplify notations.

\subsection{The dynamics in the cusp of $\SLdRZ$}
As we have seen,
the location of a point $x$ in the cusp may be characterised by the flag
$\{V_{i}(x)\}_{i\in \eta(P)}$ 
of unique rational subspaces of small covolume. 
In \S\ref{section:linear subspaces} we studied the dynamics of a given linear subspace. 
The following Lemma extend the crude bound of Corollary~\ref{corollary:3.12} to a more accurate result, using these inputs.
\begin{proposition}\label{lemma:6.1}
Let $0<\delta<\eta_0$, $x\in X$. Assume $\att_{t}x\in \cusp_{\delta}^{+}(P,[w]_P)$ for all $t\in [0,T]$.
Then 
\[\alpha_{l}(\att_{t}x)\asymp\exp(t\cdot\sum_{i=1}^{l}\upalpha_{w(i)})\cdot\alpha_{l}(x)\]
for all $l\in\eta(P)$ and $t\in [0,T]$.
\end{proposition}
\begin{proof}
Let $l\in\eta(P)$, then $\eta_l(x)<\delta$. It follows from Proposition~\ref{proposition:3.13} that $\eta_l(\att_t x)<\eta_0$ and $V_l(\att_t x)=\att_{t}V_l(x)$ for all $t\in [0,s]$ for some $s>0$. As $\att_{\min\{s,T\}} x\in \cusp_{\delta}^{+}(P,[w])$ as well, we proceed inductively and deduce that $\eta_l(\att_t x)<\eta_0$ and $V_l(\att_t x)=\att_{t} V_l(x)$ for all $0\leq t\leq T$.

Let $v\in \bigwedge^{l}\mathbb{R}^{\dtt}$ be the wedge of a $\mathbb{Z}$-basis for $x\cap V_l(x)$.
Then, using Remark~\ref{remark:3.1} and Corollary~\ref{corollary:2.4} (applied to the constant $\epsilon_0$ as in Remark~\ref{remark:2.2}), we have
\begin{align*}
\alpha_{l}(\att_{t} x)&=\covol_{\att_{t}x}(V_{l}(\att_{t} x))=\covol_{\att_{t} x}(\att_{t} V_{l}(x))=\|\att_{t}v\|_{\Ext}\\&
\asymp \exp(t\cdot\sum_{i=1}^{l}\upalpha_{w(i)})\cdot\|v\|_{\Ext}=\exp(t\cdot\sum_{i=1}^{l}\upalpha_{w(i)})\cdot\covol_{x}(V_l(x))\\&
=\exp(t\cdot\sum_{i=1}^{l}\upalpha_{w(i)})\cdot \alpha_{l}(x),
\end{align*}
where the transition from the first to the second line utilized the notations of~\S\ref{subsec:3.3}.
\end{proof}

As in~\S\ref{section:introduction}, for any $P\in\mathcal{P}$ and $[w]_P\in W_{P,\att}$ we define the projection of $\upalpha^{w}$ from $\Lie(A)$ to $\Lie(A_P)$  
by
\[\pi_P(\upalpha^{w})=\frac{1}{| W(T,P)|}\sum_{u\in W(T,P)}\upalpha^{wu}.\]
Explicitly,
if $P=P_{I}$ for $I=\{i_1,\ldots,i_k\}\subset\{1,\ldots,\dtt-1\}$, where $0=i_0<i_1<\cdots <i_k<i_{k+1}=\dtt$, then 
\[
\pi_P(\upalpha^{w})
=\left(
\begin{array}{c|c|c}
z_1 I_{i_1-i_0}  &  & \\
\hline
 &  \ddots & \\
 \hline
 & & z_{k+1} I_{i_{k+1}-i_k}
\end{array}
\right)
\]
where 
\[z_{s}=\frac{1}{i_{s}-i_{s-1}}\sum_{i=i_{s-1}+1}^{i_{s}}\upalpha_{w(i)}\] 
is the block average.

\begin{remark}\label{remark:6.2}
It can be shown as a corollary of Proposition~\ref{lemma:6.1} that if $\att_{t}x\in \cusp_{\delta}(P,[w]_{P})$ for all $t\in [0,T]$, for some $P\in \mathcal{P}$ and $w\in W$, then
\[\begin{pmatrix} \lambda_1(\att_{t} x) \\ \vdots \\ \lambda_{\dtt}(\att_{t}x) \end{pmatrix}
\asymp_{\delta}
\exp(t\pi_P(\upalpha^{w}))
\begin{pmatrix} \lambda_1(x) \\ \vdots \\ \lambda_{\dtt}(x) \end{pmatrix}\]
for all $t\in [0,T]$.
This also serves as a motivation for our definition of 
$\pi_P(\upalpha^{w})$, which can be interpreted as an averaged action.
\end{remark}

It would be convenient to give notations for the height of some $x\in X$ in the cusp.
Define
\begin{equation}\label{eq:3.4}\height(x)=\diag(-\log\lambda_1(x),\ldots,-\log\lambda_{\dtt}(x)).\end{equation}
Note that $\height(x)$ is contained in a bounded neighborhood of $\Lie(A^{+})$, for $A^{+}<A$ the subgroup of matrices whose diagonal entries are ordered from largest to smallest.
Furthermore, for any $P\in\mathcal{P}$, define the projection
\[\pi_{P}(\height(x))=\frac{1}{|W(T,P)|}\sum_{w\in W(T,P)}\height(x)^{w}.\]
Then $\pi_{P}(\height(x))$ is contained in a bounded neighborhood of $\Lie(A_{P}^{+})$, for $A_P^{+}=A_P\cap A^{+}$.

\section{Including linear functionals in the entropy bounds}\label{sec:linear functionals}

The main idea of this paper, as discussed in~\S\ref{sec:introduction}, is that upper bounds for the entropy in the cusp can be obtained by including linear functionals in the entropy computations. This allows to produce upper bounds on the entropy without having to account for all the different possible trajectories of lattices in the cusp. In this section we lay down the required settings for including linear functionals in the computations.

\subsection{Coding trajectories}\label{subsec:coding}

\subsubsection{Strategy and overview}\label{subsubsec:4.1.1}
Before diving into the technical details, we describe generally the goals of this section.
In this section, we aim to assign a coding (see Definition~\ref{def:13} later on) to any given point $x\in X$, that describes well the trajectory $\{\att_{t} x\}_{t\in [-N,N]}$, in terms of the parabolic regions the trajectory passes through, and the orientations of the lattice at these times. 
To do so, for any given $x\in X$ we perform the following procedure (which depends on $x$, but we omit this dependence from the notations for simplicity). First, we construct a partition $\mathcal{J}$ of $[-N,N]$ to finitely many intervals, and a function \[\Par:\ \mathcal{J}\to \mathcal{P}\] assigning a suitably chosen standard parabolic subgroup to each interval. This is the main task when defining a coding. When doing so, we have some freedom on choosing the parabolic subgroup for each time $t$, as we vary the parameters defining the boundary between the cusp regions to our convenience.
The key property is that for some fixed $\delta^{\prime}$, if $t\in U\in \mathcal{J}$ then \[\att_{t}x\in\cusp_{\delta^{\prime}}^{+}(\Par(U)).\]

Then, we define a finer partition $\mathcal{J}^{\prime}$ obtained by dividing each element of $\mathcal{J}$ to $\ll_{\dtt} 1$ many finer intervals, 
and define a map
\[ \Weyl :\ \mathcal{J}^{\prime}\to \bigcup_{P\in \mathcal{P}}W_{P,\att}\cup\{\emptyset\}\]
which satisfies
\[\Weyl(U)\in W_{\Par(U),\att}\cup\{\emptyset\}\]
for all $U\in \mathcal{J}^{\prime}$, so that for all $t\in U$ we have
\begin{equation}\label{eq:4.0}\att_{t}x\in \cusp_{\delta^{\prime}}^{+}(\Par(U),\Weyl(U))\end{equation}
(where for $U\in\mathcal{J}^{\prime}$ we define $\Par(U)=\Par(V)$ for $V\in\mathcal{J}$ the unique element so that $U\subseteq V$).
Once $\mathcal{J}^{\prime}$, $\Par$ and $\Weyl$ are defined, the entropy contribution of the trajectory $\{\att_{t} x\}_{t\in [-N,N]}$ will be shown in~\S\ref{sec:6}-\ref{sec:proof_1.1} to be essentially bounded by
\begin{equation*}
\sum_{\substack{U\in \mathcal{J}^{\prime}:\\ \ \emptyset\not=\Weyl(U)=[w]_{\Par(U)}}}\hspace{-1cm}|U|\cdot h(\Par(U),\att^{w})+O_{
\att}(|\mathcal{J}|).\end{equation*}
The error term of order $O_{\att}(|\mathcal{J}|)$ has two contributions of similar magnitude; First, as a consequence of Proposition~\ref{proposition:2.2} there would be in our construction $\ll |\mathcal{J}|$ intervals $U\in\mathcal{J}^{\prime}$, each of length $\ll_{\att} 1$, for which $\Weyl(U)=\emptyset$ and on which we would only have a trivial bound for the entropy.
Secondly, when the precise details are written down, one finds that the fact that we divide $[-N,N]$ to a partition adds an inherent error of $O(1)$ per interval, which accumulates to $O(|\mathcal{J}|)$.
In order for this error term to not affect the tightness of the entropy bounds, a key requirement for us is to be able to construct for any $\epsilon>0$ a partition $\mathcal{J}$ as discussed with the additional requirement
\begin{equation}\label{eq:4.1_n}|\mathcal{J}|\leq \epsilon N\end{equation}
for $N$ large enough. 

\medskip

A second, more subtle, requirement from the partition $\mathcal{J}$ (which should be satisfies simultaneously with Equation~\eqref{eq:4.1_n}) follows from our desire to include a linear functional $\phi$ in the entropy bounds. 
To obtain Theorem~\ref{theorem:1.1new}, we want to sum up a contribution \[|U|\cdot \phi(\pi_{\Par(U)}(\upalpha^{w}))\] from each interval $U\in \mathcal{J}^{\prime}$ with $\emptyset\not=\Weyl(U)=[w]_{\Par(U)}$.
Similarly to the first requirement, we ask that the sum of all of these contributions is bounded in absolute value by $\epsilon N$. 
By linearity of $\phi$, it is sufficient to require
\begin{equation}\label{eq:4.2_n}
\|\sum_{\substack{U\in \mathcal{J}^{\prime}:\\ \ \emptyset\not=\Weyl(U)=[w]_{\Par(U)}}}\hspace{-1cm}|U|\cdot \pi_{\Par(U)}(\upalpha^{w})\|\leq \epsilon N.
\end{equation}
Let us explain the nature of this sum.
For an interval $U\subseteq [-N,N]$ and $P\in \mathcal{P}$, let 
\[\Delta\height(U,P)=\pi_{P}\Big(\height(\att_{\sup U}x)\Big)-\pi_{P}\Big(\height(\att_{\inf U}x)\Big).\]
It follows from Proposition~\ref{lemma:6.1} that
\[\||U|\cdot \pi_{\Par(U)}(\upalpha^{w})-\Delta\height(U,\Par(U))\|\ll 1\]
for $U\in\mathcal{J}^{\prime}$ with $\emptyset\not=\Weyl(U)=[w]_{\Par(U)}$, and from Corollary~\ref{corollary:3.12} that
\[\|\Delta\height(U,\Par(U))\|\ll_{\att} \max(|U|,1)\ll_{\att} 1\]
for any other $U\in\mathcal{J}^{\prime}$.
So together
\begin{equation}\label{eq:4.3_n}\|\sum_{\substack{U\in \mathcal{J}^{\prime}:\\ \ \emptyset\not=\Weyl(U)=[w]_{\Par(U)}}}\hspace{-1cm}|U|\cdot \pi_{\Par(U)}(\upalpha^{w})-\sum_{U\in\mathcal{J}^{\prime}}\Delta\height(U,\Par(U))\|\ll_{\att} |\mathcal{J}^{\prime}|.\end{equation}
Therefore, the linear functional contribution essentially just counts the change in heights. Hence, using Equations~\eqref{eq:4.1_n},\eqref{eq:4.3_n}, it is sufficient to require
\begin{equation*}\|\sum_{U\in\mathcal{J}^{\prime}}\Delta\height(U,\Par(U))\|\leq \epsilon N\end{equation*}
instead of Equation~\eqref{eq:4.2_n}.
Note that we have
\[\sum_{U\in\mathcal{J}^{\prime}}\Delta\height(U,\Par(U))=\sum_{U\in\mathcal{J}}\Delta\height(U,\Par(U))\]
so this requirement is in fact equivalent to
\begin{equation}\label{eq:4.4_n}
\|\sum_{U\in\mathcal{J}}\Delta\height(U,\Par(U))\|\leq \epsilon N.
\end{equation}

Note that the full change of height in $A^{+}$, without projections, is simply
\[\sum_{U\in \mathcal{J}}\Delta\height(U,B)=\height(\att_{N}x)-\height(\att_{-N}x).\]
We take special interest in trajectories $\{\att_{t}x\}_{t\in [-N,N]}$ which begin and end at similar heights, e.g.\ $\att_{-N}x,\att_{N}x\in\compact_{\tau}(X)$ for some fixed $\tau>0$. For such $x$, the total height change in $A^{+}$ is bounded in norm by some function of $\tau$. That is, if $\Par(U)$ was equal to $B$ for all intervals,  requirement~\eqref{eq:4.4_n} would be immediately met.

However, we would like to count the change of height in each interval $U$ with respect to its assigned parabolic subgroup $\Par(U)$, as in Equation~\eqref{eq:4.4_n}, and still have it be small. 
This discrepancy between counting the full and the projected changes of heights is the essence of the construction we will show in this section. 
In fact, we will construct $\mathcal{J}$ so that is satisfies the following stronger property.
Define
\[\err(z,P)=\|\height(z)-\pi_{P}(\height(z))\|.\]
Then we 
will show that
\begin{equation}\label{eq:4.2}\sum_{U\in\mathcal{J}}\Big[\err(\att_{\inf U}x,\Par(U))+\err(\att_{\sup U}x,\Par(U))\Big]\leq \epsilon N,
\end{equation}
which is of course stronger than Equation~\eqref{eq:4.4_n} for large $N$, as
\begin{multline*}\|\sum_{U\in\mathcal{J}}\Delta\height(U,\Par(U))\|\leq \sum_{U\in\mathcal{J}}\Big[\err(\att_{\inf U}x,\Par(U))+\err(\att_{\sup U}x,\Par(U))\Big]\\
+\|\height(\att_{N}x)-\height(\att_{-N}x)\|\end{multline*}
and the latter height difference is bounded in norm as discussed above.

\subsubsection{Construction: naive approach}\label{subsubsec:4.1.2}
We first suggest a naive construction. It is satisfactory in order to bound entropy, but fails to allow the addition of linear functionals. Still, it is beneficial to first understand this simpler case.

In our constructions, we will often define parabolic subgroups $P\in \mathcal{P}$ by the set $\eta(P)$ of partial sums of the block sizes of $P$ (not including $0$ or $\dtt$). So the task in defining $\mathcal{J}$ is to declare which dimensions $l\in \{1,\ldots,\dtt-1\}$ are significant for $\att_{n}x$ at each time $n\in [-N,N]$.

Fix some $0<\delta<\eta_0$ with $|\log\delta|<N$.
In the naive approach, we
 keep track only of dimensions $l$ where the jump in the successive minima is significant with respect to $\delta$, that is $\eta_l(\att_t x)<\delta$. But once $l$ was marked as a significant jump at some time $t$, it is also considered significant at the time interval around $t$ where $\eta_l(\att_s x)< \delta^{\prime}$, for some  $\delta <\delta^{\prime}<\eta_0$.
We choose $\delta^{\prime}$ to satisfy \[\delta^{\prime}>\delta^{1/2},\]
in order to simplify the expressions we obtain.

First of all, this approach clearly guarantees that Equation~\eqref{eq:4.0} holds.
Secondly, by this approach, each $l$ that was marked as significant at some time, stays significant for a duration of at least 
\[\gg_{\att} \log \frac{\delta^{\prime}}{\delta} > \frac{1}{2}|\log\delta|.\]
Then a partition constructed this way will have at most $|\mathcal{J}|\ll_{\att} N/|\log\delta|$ elements, which satisfies requirement~\eqref{eq:4.1_n} with $\epsilon=1/|\log\delta|$.

However, requirement~\eqref{eq:4.2} would fail in general. 
Note that by construction, for all $U\in\mathcal{J}$, $t\in U$ and $l\not\in\eta(\Par(U))$, we have
\[0<(\height(\att_t x))_{l}-(\height(\att_t x))_{l+1}=-\log(\eta_l(\att_t x))<-\log \delta=|\log\delta|.\]
Recall that $\pi_{\Par(U)}(\height(\att_t x))$ just averages the entries inside each block, which in our case are $|\log\delta|$-close to each other. So the tightest bound that can a-priory be given is
\[\err(\att_t x  ,\Par(U))\ll |\log\delta|.\]
Therefore, although the number of intervals $|\mathcal{J}|\ll_{\att} \epsilon N$ was small,
the sum of this error term $|\log\delta|=1/\epsilon$ over all intervals $U\in\mathcal{J}$ fails to be small. It accumulates to $O_{\att}(N)$ which contradicts our requirement.

The goal of the detailed construction in this section is to produce intervals where the error term at the interval end points is bounded more efficiently compared to its effect of increasing the number of intervals. To do so, we would have liked to achieve a construction with
\[|\mathcal{J}|\ll_{\att} \epsilon N\] while \[\err(\att_{t}x,\Par(U))\ll 1/\epsilon^{\beta}\]
for $t=\inf U,\sup U$, and for some $\beta<1$. 
However, reducing the error for all of the intervals uniformly is too strict. Instead, in~\S\ref{subsubsec:4.1.3} we construct a degree function for intervals
\[\deg:\ \mathcal{J}\to \{1,\ldots,\dtt\},\]
so that $\mathcal{J}$ is comprised of disjoint subsets
\[\mathcal{J}=\bigcupdot_{m=1}^{\dtt}\left(\mathcal{J}\cap\deg^{-1}(m)\right).\]
In this construction, we will have an efficient bound of error for each $\mathcal{J}\cap\deg^{-1}(m)$ separately. That is, we would essentially prove
\[\err(\att_{t} x,\Par(U))\cdot |\mathcal{J}\cap\deg^{-1}(m)|\ll_{\att} \epsilon^{1-\beta} N\]
for all $1\leq m\leq \dtt$, $U\in \mathcal{J}\cap \deg^{-1}(m)$, and $t=\inf U,\sup U$. The precise statement is shown in Proposition~\ref{proposition:4.4_n}.

\subsubsection{Construction: notations}\label{subsubsec:4.1.3}
For any finite collection $\mathcal{T}$ of disjoint sub intervals of $[-N,N]$, let $\mathcal{T}^{c}$ be the complement of $\mathcal{T}$ in $[-N,N]$, i.e.\ the partition of $[-N,N]\smallsetminus\bigcup \mathcal{T}$ to maximal disjoint intervals.

Let $x\in\SLdRZ$, $0<\delta<\delta^{\prime}<\eta_0$ and $N\in \mathbb{N}$. As before we assume $\delta^{\prime}>\delta^{1/2}$ and $|\log\delta|<N$.
For any $t\in [-N,N]$ and $l\in \{1,\ldots,\dtt-1\}$, let 
\[m_{l}(t)=\sup\{s\leq t:\ \eta_l(\att_s x)\geq \delta^{\prime}\},\]
i.e.\ the previous time $s$ that $\eta_l(\att_s x)$ was larger than $\delta^{\prime}$ (or $-\infty$ if it wasn't), and let
\[M_{l}(t)=\inf\{t\leq s:\ \eta_l(\att_s x)\geq \delta^{\prime}\},\]
i.e.\ the next time $s$ when $\eta_l(\att_{s} x)$ will be larger than $\delta^{\prime}$ (or $\infty$ if it will not).
Let
\[\mathcal{T}_{\delta,l}=\left\{(m_l(t),M_l(t))\cap[-N,N]:\ t\in[-N,N]\text{ such that } \eta_l(\att_t x)<\delta\right\}.\]
Note that for $s\in (m_l(t),M_l(t))$ we have $(m_l(t),M_l(t))=(m_l(s),M_l(s))$. Hence, any $U_1\not=U_2\in T_{\delta,l}$ are disjoint. 
As follows from Corollary~\ref{corollary:3.12},
every $U\in \mathcal{T}_{\delta,l}$ is of length \[\gg_{\att} \log\frac{\delta^{\prime}}{\delta}>\frac{1}{2}|\log\delta|,\]
so in particular $\mathcal{T}_{\delta,l}$ is finite.

Each $\mathcal{T}_{\delta,l}$ keeps track of the times when there is a $\delta$-significant $l$-dimensional rational subspace for $\att_t x$.
For each time $t\in [-N,N]$, we consider the set of indices $l$ which are $\delta$-significant, i.e.\ the indices $l$ for which there is an element of $\mathcal{T}_{\delta,l}$ which covers $t$. We define
\[E_{\delta}(t)=\{l:\ t\in \bigcup \mathcal{T}_{\delta,l}\}.\]
Then we consider the partition
\[\mathcal{F}_{\delta}=\bigvee_{l=1}^{\dtt-1}(\mathcal{T}_{\delta,l}\cup\mathcal{T}_{\delta,l}^{c}),\]
which is precisely the partition of $[-N,N]$ to maximal disjoint intervals so that $E_{\delta}$ is constant on each interval. For $U\in \mathcal{F}_{\delta}$, we let $E_{\delta}(U)$ stand for this common value.
We take special interest in intervals with non-empty $E_{\delta}$, so we define 
\[\mathcal{T}_{\delta}=\{U\in \mathcal{F}_{\delta}:\ E_{\delta}(U)\not=\emptyset\}.\]
An illustration of the construction of $\mathcal{T}_{\delta}$ is shown in Figure~\ref{fig:4.1}.

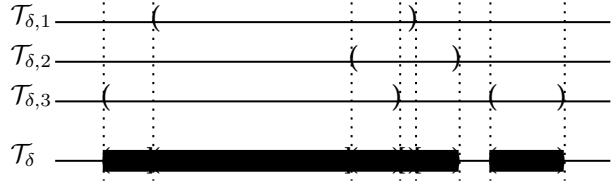
\begin{figure}
\begin{center}

\tikzset{every picture/.style={line width=0.75pt}} 
\begin{tikzpicture}[x=0.75pt,y=0.75pt,yscale=-1,xscale=1]

\draw    (264,230) -- (544,230) ;
\draw    (264,250) -- (544,250) ;
\draw    (264,270) -- (544,270) ;

\draw [
draw opacity=1 ]   (264,300) -- (544,300) ;

\draw (310,220) node [anchor=north west][inner sep=0.75pt]   [align=left] {\textbf{(}};
\draw (440,220) node [anchor=north west][inner sep=0.75pt]   [align=left] {\textbf{)}};
\draw (410,240) node [anchor=north west][inner sep=0.75pt]   [align=left] {\textbf{(}};
\draw (462,240) node [anchor=north west][inner sep=0.75pt]   [align=left] {\textbf{)}};
\draw (285,260) node [anchor=north west][inner sep=0.75pt]   [align=left] {\textbf{(}};
\draw (432,260) node [anchor=north west][inner sep=0.75pt]   [align=left] {\textbf{)}};
\draw (480,260) node [anchor=north west][inner sep=0.75pt]   [align=left] {\textbf{(}};
\draw (515,260) node [anchor=north west][inner sep=0.75pt]   [align=left] {\textbf{)}};

\draw  [dash pattern={on 0.84pt off 2.51pt}]  (285+3.5,220) -- (285+3.5,310) ;
\draw  [dash pattern={on 0.84pt off 2.51pt}]  (310+3.5,220) -- (310+3.5,310) ;
\draw  [dash pattern={on 0.84pt off 2.51pt}]  (410+3.5,220) -- (410+3.5,310) ;
\draw  [dash pattern={on 0.84pt off 2.51pt}]  (432+6,220) -- (432+6,310) ;
\draw  [dash pattern={on 0.84pt off 2.51pt}]  (440+6,220) -- (440+6,310) ;
\draw  [dash pattern={on 0.84pt off 2.51pt}]  (462+6,220) -- (462+6,310) ;
\draw  [dash pattern={on 0.84pt off 2.51pt}]  (480+3.5,220) -- (480+3.5,310) ;
\draw  [dash pattern={on 0.84pt off 2.51pt}]  (515+6,220) -- (515+6,310) ;

\draw  [draw opacity=0][fill={rgb, 255:red, 0; green, 0; blue, 0 }  ,fill opacity=0.5 ] (285+3.5,295) -- (462+5,295)  -- (462+5,305)  -- (285+3.5,305)  -- cycle ;
\draw  [draw opacity=0][fill={rgb, 255:red, 0; green, 0; blue, 0 }  ,fill opacity=0.5 ] (480+3.5,295) -- (515+5,295)  -- (515+5,305)  -- (480+3.5,305)  -- cycle ;

\draw (285,290+1.5) node [anchor=north west][inner sep=0.75pt]   [align=left] {\textcolor{black}{\textbf{(}}};

\draw (310-1.5,290+1.5) node [anchor=north west][inner sep=0.75pt]   [align=left] {\textcolor{black}{\textbf{]}}};
\draw (310,290+1.5) node [anchor=north west][inner sep=0.75pt]   [align=left] {\textcolor{black}{\textbf{(}}};

\draw (410-1.5,290+1.5) node [anchor=north west][inner sep=0.75pt]   [align=left] {\textcolor{black}{\textbf{]}}};
\draw (410,290+1.5) node [anchor=north west][inner sep=0.75pt]   [align=left] {\textcolor{black}{\textbf{(}}};

\draw (432,290+1.5) node [anchor=north west][inner sep=0.75pt]   [align=left] {\textcolor{black}{\textbf{)}}};
\draw (432+3,290+1.5) node [anchor=north west][inner sep=0.75pt]   [align=left] {\textcolor{black}{\textbf{[}}};

\draw (440,290+1.5) node [anchor=north west][inner sep=0.75pt]   [align=left] {\textcolor{black}{\textbf{)}}};
\draw (440+3,290+1.5) node [anchor=north west][inner sep=0.75pt]   [align=left] {\textcolor{black}{\textbf{[}}};

\draw (462,290+1.5) node [anchor=north west][inner sep=0.75pt]   [align=left] {\textcolor{black}{\textbf{)}}};

\draw (480,290+1.5) node [anchor=north west][inner sep=0.75pt]   [align=left] {\textcolor{black}{\textbf{(}}};

\draw (515,290+1.5) node [anchor=north west][inner sep=0.75pt]   [align=left] {\textcolor{black}{\textbf{)}}};

\draw (240,220) node [anchor=north west][inner sep=0.75pt]    {$\mathcal{T}_{\delta ,1}$};
\draw (240,240) node [anchor=north west][inner sep=0.75pt]    {$\mathcal{T}_{\delta ,2}$};
\draw (240,260) node [anchor=north west][inner sep=0.75pt]    {$\mathcal{T}_{\delta ,3}$};

\draw (240,290) node [anchor=north west][inner sep=0.75pt] 
{$\mathcal{T}_{\delta}$};

\end{tikzpicture}
\end{center}
\caption{An illustration of the construction of $\mathcal{T}_{\delta}$ for $\dtt=4$.}
\label{fig:4.1}
\end{figure}

Note that in these new notations, 
\[\mathcal{F}_{\delta}=\mathcal{T}_{\delta}\cup\mathcal{T}_{\delta}^{c}\]
is precisely the naive partition discussed in~\S\ref{subsubsec:4.1.2}.
To allow us to reduce the size of the errors,
as discussed in~\S\ref{subsubsec:4.1.2}, we define 
inductively the partition $\mathcal{J}$ by taking refinements of the naive partition by finer and finer partitions. 
To do so, we fix
a parameter $r\in(0,1)$ so that $\delta^{r^{\dtt+1}}<\delta^{\prime}$, i.e.\ $r>(\frac{\log\delta^{\prime}}{\log\delta})^{1/(\dtt+1)}$. This would allow us to perform $\dtt$ steps of the induction. The saving in the errors would turn out to be
\[\err(\att_{t}x,\Par(U))\cdot |\mathcal{J}^{\prime}\cap\deg^{-1}(m)|\ll_{\att} rN.\]

\subsubsection{Construction: full details}\label{subsubsec:4.1.4}
In the following induction, we define partitions $\mathcal{J}_{m}$ which are comprised of two sub-collections: final intervals (which would belong to the final partition $\mathcal{J}$), and temporary intervals (which could change in the next step).
We begin the induction 
with
$\mathcal{J}_{0}=\mathcal{F}_{\delta}$
and no final intervals
\[\mathcal{J}^{\mathrm{final}}_{0}=\emptyset,\ \mathcal{J}^{\mathrm{temp}}_{0}=\mathcal{J}_{0}.\]
Assume that for some $0\leq m\leq \dtt-1$ a partition \[\mathcal{J}_{m}=\mathcal{J}_{m}^{\mathrm{final}}\cupdot\mathcal{J}_{m}^{\mathrm{temp}}\] of $[-N,N]$ was constructed, so that 
$\mathcal{J}^{\mathrm{temp}}_{m}\subseteq \mathcal{F}_{\delta^{r^{m}}}$ and $|E_{\delta^{r^{m}}}(U)|\geq m$ for all $U\in \mathcal{J}_{m}^{\mathrm{temp}}$.
For $U\in \mathcal{J}_{m}^{\mathrm{temp}}$, let
\[a_U=\inf \{t\in U:\  E_{\delta^{r^{m+1}}}(t)=E_{\delta^{r^{m}}}(U)\}\]
be the first time in $U$ (or $\infty$ if such does not exist) when there are no $\delta^{r^{m+1}}$-significant indices other than those that are already $\delta^{r^{m}}$-significant.
Similarly, let
\[b_U=\sup\{t\in U:\ E_{\delta^{r^{m+1}}}(t)=E_{\delta^{r^{m}}}(U)\}.\]
Then we define
\[\mathcal{J}^{\mathrm{final}}_{m+1}=\mathcal{J}^{\mathrm{final}}_{m}\cupdot\{[a_U,b_U]\cap U:\ U\in\mathcal{J}^{\mathrm{temp}}_{m},\ a_U,b_U\not\in\{\pm\infty\}\}.\]
We also define
\[\mathcal{J}^{\mathrm{temp}}_{m+1}=(\mathcal{J}^{\mathrm{final}}_{m+1})^{c}\vee \mathcal{F}_{\delta^{r^{m+1}}}\]
and
\[\mathcal{J}_{m+1}=\mathcal{J}^{\mathrm{final}}_{m+1}\cupdot\mathcal{J}^{\mathrm{temp}}_{m+1}.\]

An illustration of this construction is shown in Figure~\ref{fig:4.2}. 
In words, at the $(m+1)$th step, we consider a finer partition of $[-N,N]$, corresponding to the larger parameter $\delta^{r^{m+1}}$. If by looking at this finer partition we find that a temporary interval $U$ is completely covered by $\bigcup_{l\not\in E_{\delta^{r^{m}}}(U)} \mathcal{T}_{\delta^{r^{m+1}},l}$,
we have no ability to save error terms, and have to divide $U$ according to $\mathcal{F}_{\delta^{r^{m+1}}}$ and carry it to the next step. Otherwise, if $U$ is not entirely covered, we take a maximal sub-interval $[a_U,b_U]\cap U$ whose two end points are not covered, and declare it final. 

\begin{lemma}\label{lemma:4.1}
The collection $\mathcal{J}_{m+1}^{\mathrm{temp}}$ in the induction step satisfies the following.
\begin{enumerate}
    \item \label{item:4.1.1}
$\mathcal{J}_{m+1}^{\mathrm{temp}}\subseteq\mathcal{F}_{\delta^{r^{m+1}}}$.
    \item \label{item:4.1.2} $|E_{\delta^{r^{m+1}}}(U)|\geq m+1$ for all $U\in\mathcal{J}_{m+1}^{\mathrm{temp}}$.
\end{enumerate}
\end{lemma}
\begin{proof}
    To prove~\eqref{item:4.1.1}, it is sufficient to show that every $V\in (\mathcal{J}_{m+1}^{\mathrm{final}})^{c}$ is a union of elements of $\mathcal{F}_{\delta^{r^{m+1}}}$.
    
    Let $V\in(\mathcal{J}_{m+1}^{\mathrm{final}})^{c}$.
    Then it is easy to see that $V$ is a union of two types of intervals -- elements of $\mathcal{J}_{m}^{\mathrm{temp}}$, and intervals of the form $U\cap(-\infty,a_{U})$ or $U\cap(b_{U},\infty)$ for some $U\in\mathcal{J}_{m}^{\mathrm{temp}}$.  
    Note that by the induction assumption $\mathcal{J}_{m}^{\mathrm{temp}}\subseteq \mathcal{F}_{\delta^{r^{m}}}$. Therefore, as $\mathcal{F}_{\delta^{r^{m+1}}}$ is a refinement of $\mathcal{F}_{\delta^{r^{m}}}$, it is clear that any union of elements of $\mathcal{J}_{m}^{\mathrm{temp}}$ is also a union of elements of $\mathcal{F}_{\delta^{r^{m+1}}}$ as desired. Therefore, it is sufficient to prove that for any $U\in\mathcal{J}_{m}^{\mathrm{temp}}$, if $a_{U},b_{U}\not\in\{\pm\infty\}$ then $U\cap(-\infty,a_{U})$ and $U\cap(b_{U},\infty)$ are both unions of elements of $\mathcal{F}_{\delta^{r^{m+1}}}$.
    Indeed,
    it follows from the definitions of $a_U,b_U$ that in this case $U\cap[a_U,b_U]$ is a union of elements of $\mathcal{F}_{\delta^{r^{m+1}}}$.
    As $U\in\mathcal{J}_m^{\mathrm{temp}}\subseteq \mathcal{F}_{\delta^{r^{m}}}$, it is too
    a union of elements of $\mathcal{F}_{\delta^{r^{m+1}}}$. Therefore, so are $U\cap (-\infty,a_U)$ and $U\cap(b_U,\infty)$.
    
    \medskip

    For~\eqref{item:4.1.2}, note that for any $U\in\mathcal{J}_{m}^{\mathrm{temp}}$, the (possibly empty) interval $[a_U,b_U]$ contains all the times $t\in U$ so that  $E_{\delta^{r^{m+1}}}(t)=E_{\delta^{r^{m}}}(U)$.
    So for any other $t\in U$ we have
\[E_{\delta^{r^{m}}}(U)\subsetneq E_{\delta^{r^{m+1}}}(t).\]
    By the induction assumption $|E_{\delta^{r^{m}}}(U)|\geq m$,
    so we deduce that
    \[|E_{\delta^{r^{m+1}}}(t)|\geq m+1\]
    for any $t\in U\smallsetminus [a_U,b_U]$.
    As any element of    $\mathcal{J}_{m+1}^{\mathrm{temp}}$ is contained in $U\smallsetminus [a_U,b_U]$ for some $U\in\mathcal{J}_{m}^{\mathrm{temp}}$, the conclusion follows.
\end{proof}

\begin{corollary}
    This construction terminates after $\dtt$ steps, with
    \[\mathcal{J}_{\dtt}^{\mathrm{final}}=\mathcal{J}_{\dtt-1},\qquad \mathcal{J}_{\dtt}^{\mathrm{temp}}=\emptyset.\]
\end{corollary}
\begin{proof}
By Lemma~\ref{lemma:4.1}, every interval $U\in\mathcal{J}_{\dtt-1}^{\mathrm{temp}}$ satisfies $E_{\delta^{r^{\dtt}}}(U)=\{1,\ldots,\dtt-1\}$, and so $[a_U,b_U]=[\inf U,\sup U]$ by definition. Therefore, $U\cap[a_U,b_U]=U\in \mathcal{J}_{\dtt}^{\mathrm{final}}$, i.e.\ \[\mathcal{J}_{\dtt}^{\mathrm{final}}=\mathcal{J}_{\dtt-1}^{\mathrm{final}}\cupdot\mathcal{J}_{\dtt-1}^{\mathrm{temp}}=\mathcal{J}_{\dtt-1}.\]
The claim $\mathcal{J}_{\dtt}^{\mathrm{temp}}=\emptyset$ follows immediately.
\end{proof}

Then, we define \[\mathcal{J}=\mathcal{J}_{\dtt}^{\mathrm{final}}=\mathcal{J}_{\dtt-1}.\]
We define a degree function which assigns to each $U\in\mathcal{J}$ the step of the induction where it was added to $\mathcal{J}$, i.e.\
\[\deg:\ \mathcal{J}\to\{1,\ldots,\dtt\}\]
given by
\[\deg(U)=\min\{m\in \{1,\ldots,\dtt\}:\ U\in\mathcal{J}_{m}^{\mathrm{final}}\}.\]
We also define a function
\[\Par:\ \mathcal{J}\to \mathcal{P}\]
which assigns a standard parabolic subgroup, according to the significant indices in the construction, as follows.
Any $U\in \mathcal{J}$ is by construction of the form
\[U=V\cap [a_{V},b_{V}]\]
for some $V\in \mathcal{J}_{\deg(U)-1}^{\mathrm{temp}}$.
Then we define
$\Par(U)$ to be
the unique standard parabolic subgroup $\Par(U)\in\mathcal{P}$ with
\[\eta(\Par(U))=E_{\delta^{r^{\deg(U)-1}}}(V).\]

The key property of the construction is that at the end points of any $U\in\mathcal{J}$ the jumps satisfy
\[\eta_l(\att_t x)\geq \delta^{r^{\deg U}},\]
for all 
$l\not\in \eta(\Par(U))$
at the endpoints
$t=\inf U,\sup U$,
despite the fact that $U$ was in essence defined in the $(\deg U-1)$th step using $\mathcal{F}_{\delta^{r^{\deg U-1}}}$ (it is just an element of $\mathcal{J}_{\deg U-1}^{\mathrm{temp}}$ trimmed at its ends).
This gap of one power of $r$ is the fact that allows to cut down the error terms.
Some of the properties of $\mathcal{J}$ are summarised in the following proposition.

\begin{proposition}\label{proposition:5.1}
Let $\mathcal{J},\Par$ be as above. 
Then
\begin{enumerate}
    \item\label{item:5.1_1} 
    $\eta(\att_t x,\delta)\subseteq \eta(\Par(U))\subseteq \eta(\att_t x,\delta^{\prime})$
    for any $t\in U\in\mathcal{J}$.  
    In particular, 
        $\att_t x\in \cusp_{\delta^{\prime}}^{+}(\Par(U))$.
    \item\label{item:5.1_2} Let $U\in \mathcal{J}$. Then $\eta_l(\att_t x)\geq \delta^{r^{\deg U}}$ for $t=\inf U, \sup U$, for all
    $l\not\in \eta(\Par(U))$. 
    \item\label{item:5.1_3} $|\{U\in \mathcal{J}:\ \deg(U)=m\}|\ll_{\att}  N/(r^{m-1}|\log\delta|)$
\end{enumerate}
\end{proposition}
\begin{proof}
Items \ref{item:5.1_1}-\ref{item:5.1_2} follow easily from the construction.

For item \ref{item:5.1_3}, we first bound the length of elements of $\mathcal{T}_{\delta^{r^{m}}}$. 
Let $U\in \mathcal{T}_{\delta^{r^{m}},l}$. Then by construction, there is some $t\in [-N,N]$ such that $\eta_{l}(\att_t x)<\delta^{r^{m}}$, and $U=(m_{l}(t),M_{l}(t))\cap [-N,N]$.
By Corollary~\ref{corollary:3.12}, 
\begin{align*}M_{l}(t)-t&\gg_{\att} \log\frac{\delta^{\prime}}{\delta^{r^{m}}}=r^{m}|\log\delta|-|\log\delta^{\prime}|=r^{m}(|\log\delta|-r^{-m}|\log\delta^{\prime}|)\\&
\geq r^{m}(|\log\delta|-r^{-\dtt}|\log\delta^{\prime}|)>r^{m}(|\log\delta|-|\log\delta^{\prime}|^{1/(\dtt+1)}|\log\delta|^{\dtt/(\dtt+1)})\\&
=r^{m}|\log\delta|^{\dtt/(\dtt+1)}(|\log\delta|^{1/(\dtt+1)}-|\log\delta^{\prime}|^{1/(\dtt+1)})\\&
\gg r^{m}(|\log\delta|-|\log\delta^{\prime}|)=r^{m}\log\frac{\delta^{\prime}}{\delta},
\end{align*}
where we used the facts that $r\in ((\frac{\log\delta^{\prime}}{\log\delta})^{1/(\dtt+1)},1)$, $m\leq \dtt$, and that for $x>y>0$
\[(x-y)=(x^{\dtt+1}-y^{\dtt+1})/(x^{\dtt}+x^{\dtt-1}y+\cdots +y^{\dtt})\geq (x^{\dtt+1}-y^{\dtt+1})/((\dtt+1) x^{\dtt}).\]
The same estimate holds for $t-m_{l}(t)$ as well.

So we see that each interval of $\mathcal{T}_{\delta^{r^{m}},l}$ is of length at least $\gg_{\att} r^{m}\log\frac{\delta^{\prime}}{\delta}$, hence the number of intervals is bounded by
\[|\mathcal{T}_{\delta^{r^{m}},l}|\ll_{\att} N/(r^{m}\log\frac{\delta^{\prime}}{\delta})\ll N/(r^{m}|\log\delta|),\]
where the assumption $2|\log\delta^{\prime}|<|\log\delta|< N$ 
was used.
Then, 
it follows that
\begin{equation}\label{eq:4.7n}|\mathcal{F}_{\delta^{r^{m}}}|\ll_{\att}  N/(r^{m}|\log\delta|) \end{equation}
as well.

Now, it follows from Lemma~\ref{lemma:4.1} and its proof that for any $m$, the partition $\mathcal{J}_{m}$ is coarser than $\mathcal{F}_{\delta^{r^{m}}}$ in the sense that it is obtained by taking unions of elements of $\mathcal{F}_{\delta^{r^{m}}}$.
Then
\begin{equation}\label{eq:4.7_n}|\mathcal{J}_{m}|\leq |\mathcal{F}_{\delta^{r^{m}}}|\ll_{\att}  N/(r^{m}|\log\delta|) .\end{equation}
Lastly, as we saw, each element $U\in\mathcal{J}$ with $\deg(U)=m$ is obtained by trimming the ends of an element of $\mathcal{J}_{m-1}^{\mathrm{temp}}$, and each element of $\mathcal{J}_{m-1}^{\mathrm{temp}}$ is trimmed this way at most once. 
Therefore,
\[|\{U\in \mathcal{J}:\ \deg(U)=m\}|\leq | \mathcal{J}_{m-1}^{\mathrm{temp}}|\ll_{\att}   N/(r^{m-1}|\log\delta|)  ,\]
which
concludes the proof.

\end{proof}

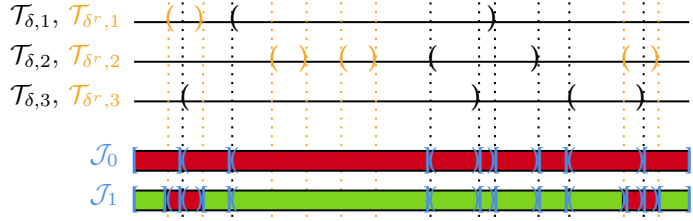
\begin{figure}
\begin{center}

\tikzset{every picture/.style={line width=0.75pt}} 

\begin{tikzpicture}[x=0.75pt,y=0.75pt,yscale=-1,xscale=1]

\draw    (264,230) -- (544,230) ;
\draw    (264,250) -- (544,250) ;
\draw    (264,270) -- (544,270) ;

\draw [color={rgb, 255:red, 74; green, 144; blue, 226 }  ,draw opacity=1 ]   (264,300) -- (544,300) ;
\draw [color={rgb, 255:red, 74; green, 144; blue, 226 }  ,draw opacity=1 ]   (264,320) -- (544,320) ;

\draw (310,220) node [anchor=north west][inner sep=0.75pt]   [align=left] {\textbf{(}};
\draw (440,220) node [anchor=north west][inner sep=0.75pt]   [align=left] {\textbf{)}};
\draw (410,240) node [anchor=north west][inner sep=0.75pt]   [align=left] {\textbf{(}};
\draw (462,240) node [anchor=north west][inner sep=0.75pt]   [align=left] {\textbf{)}};
\draw (285,260) node [anchor=north west][inner sep=0.75pt]   [align=left] {\textbf{(}};
\draw (432,260) node [anchor=north west][inner sep=0.75pt]   [align=left] {\textbf{)}};
\draw (480,260) node [anchor=north west][inner sep=0.75pt]   [align=left] {\textbf{(}};
\draw (515,260) node [anchor=north west][inner sep=0.75pt]   [align=left] {\textbf{)}};

\draw (277.5,220) node [anchor=north west][inner sep=0.75pt]   [align=left] {\textbf{\textcolor[rgb]{0.96,0.65,0.14}{(}}};
\draw (292.5,220) node [anchor=north west][inner sep=0.75pt]   [align=left] {\textbf{\textcolor[rgb]{0.96,0.65,0.14}{)}}};
\draw (330,240) node [anchor=north west][inner sep=0.75pt]   [align=left] {\textbf{\textcolor[rgb]{0.96,0.65,0.14}{(}}};
\draw (345,240) node [anchor=north west][inner sep=0.75pt]   [align=left] {\textbf{\textcolor[rgb]{0.96,0.65,0.14}{)}}};
\draw (365,240) node [anchor=north west][inner sep=0.75pt]   [align=left] {\textbf{\textcolor[rgb]{0.96,0.65,0.14}{(}}};
\draw (380,240) node [anchor=north west][inner sep=0.75pt]   [align=left] {\textbf{\textcolor[rgb]{0.96,0.65,0.14}{)}}};
\draw (507.5,240) node [anchor=north west][inner sep=0.75pt]   [align=left] {\textbf{\textcolor[rgb]{0.96,0.65,0.14}{(}}};
\draw (522.5,240) node [anchor=north west][inner sep=0.75pt]   [align=left] {\textbf{\textcolor[rgb]{0.96,0.65,0.14}{)}}};

\draw  [dash pattern={on 0.84pt off 2.51pt}]  (285+3.5,220) -- (285+3.5,330) ;
\draw  [dash pattern={on 0.84pt off 2.51pt}]  (310+3.5,220) -- (310+3.5,330) ;
\draw  [dash pattern={on 0.84pt off 2.51pt}]  (410+3.5,220) -- (410+3.5,330) ;
\draw  [dash pattern={on 0.84pt off 2.51pt}]  (432+6,220) -- (432+6,330) ;
\draw  [dash pattern={on 0.84pt off 2.51pt}]  (440+6,220) -- (440+6,330) ;
\draw  [dash pattern={on 0.84pt off 2.51pt}]  (462+6,220) -- (462+6,330) ;
\draw  [dash pattern={on 0.84pt off 2.51pt}]  (480+3.5,220) -- (480+3.5,330) ;
\draw  [dash pattern={on 0.84pt off 2.51pt}]  (515+6,220) -- (515+6,330) ;

\draw [color={rgb, 255:red, 245; green, 166; blue, 35 }  ,draw opacity=1 ] [dash pattern={on 0.84pt off 2.51pt}]  (278+3.5,220) -- (277.5+3.5,330) ;
\draw [color={rgb, 255:red, 245; green, 166; blue, 35 }  ,draw opacity=1 ] [dash pattern={on 0.84pt off 2.51pt}]  (293+6,220) -- (292.5+6,330) ;
\draw [color={rgb, 255:red, 245; green, 166; blue, 35 }  ,draw opacity=1 ] [dash pattern={on 0.84pt off 2.51pt}]  (330+3.5,220) -- (330+3.5,330) ;
\draw [color={rgb, 255:red, 245; green, 166; blue, 35 }  ,draw opacity=1 ] [dash pattern={on 0.84pt off 2.51pt}]  (345+6,220) -- (345+6,330) ;
\draw [color={rgb, 255:red, 245; green, 166; blue, 35 }  ,draw opacity=1 ] [dash pattern={on 0.84pt off 2.51pt}]  (365+3.5,220) -- (365+3.5,330) ;
\draw [color={rgb, 255:red, 245; green, 166; blue, 35 }  ,draw opacity=1 ] [dash pattern={on 0.84pt off 2.51pt}]  (380+6,220) -- (380+6,330) ;
\draw [color={rgb, 255:red, 245; green, 166; blue, 35 }  ,draw opacity=1 ] [dash pattern={on 0.84pt off 2.51pt}]  (507.5+3.5,220) -- (507.5+3.5,330) ;
\draw [color={rgb, 255:red, 245; green, 166; blue, 35 }  ,draw opacity=1 ] [dash pattern={on 0.84pt off 2.51pt}]  (522.5+6,220) -- (522.5+6,330) ;

\draw  [draw opacity=0][fill={rgb, 255:red, 208; green, 2; blue, 27 }  ,fill opacity=0.5 ] (264,295) -- (544,295)  -- (544,305)  -- (264,305)  -- cycle ;
\draw  [draw opacity=0][fill={rgb, 255:red, 208; green, 2; blue, 27 }  ,fill opacity=0.5 ] (277.5+3,315) -- (292.5+6,315)  -- (292.5+6,325)  -- (277.5+3,325)  -- cycle ;
\draw  [draw opacity=0][fill={rgb, 255:red, 208; green, 2; blue, 27 }  ,fill opacity=0.5 ] (507.5+3,315) -- (522.5+6,315)  -- (522.5+6,325)  -- (507.5+3,325)  -- cycle ;

\draw  [draw opacity=0][fill={rgb, 255:red, 126; green, 211; blue, 33 }  ,fill opacity=0.5 ] (264,315) -- (277.5+3,315)  -- (277.5+3,325)  -- (264,325)  -- cycle ;
\draw  [draw opacity=0][fill={rgb, 255:red, 126; green, 211; blue, 33 }  ,fill opacity=0.5 ] (292.5+6,315) -- (507.5+3,315)  -- (507.5+3,325)  -- (292.5+6,325)  -- cycle ;
\draw  [draw opacity=0][fill={rgb, 255:red, 126; green, 211; blue, 33 }  ,fill opacity=0.5 ] (522.5+6,315) -- (544,315)  -- (544,325)  -- (522.5+6,325)  -- cycle ;

\draw (264-3.5,290+1.5) node [anchor=north west][inner sep=0.75pt]   [align=left] {\textcolor[rgb]{0.29,0.56,0.88}{\textbf{[}}};

\draw (285-1.5,290+1.5) node [anchor=north west][inner sep=0.75pt]   [align=left] {\textcolor[rgb]{0.29,0.56,0.88}{\textbf{]}}};
\draw (285,290+1.5) node [anchor=north west][inner sep=0.75pt]   [align=left] {\textcolor[rgb]{0.29,0.56,0.88}{\textbf{(}}};

\draw (310-1.5,290+1.5) node [anchor=north west][inner sep=0.75pt]   [align=left] {\textcolor[rgb]{0.29,0.56,0.88}{\textbf{]}}};
\draw (310,290+1.5) node [anchor=north west][inner sep=0.75pt]   [align=left] {\textcolor[rgb]{0.29,0.56,0.88}{\textbf{(}}};

\draw (410-1.5,290+1.5) node [anchor=north west][inner sep=0.75pt]   [align=left] {\textcolor[rgb]{0.29,0.56,0.88}{\textbf{]}}};
\draw (410,290+1.5) node [anchor=north west][inner sep=0.75pt]   [align=left] {\textcolor[rgb]{0.29,0.56,0.88}{\textbf{(}}};

\draw (432,290+1.5) node [anchor=north west][inner sep=0.75pt]   [align=left] {\textcolor[rgb]{0.29,0.56,0.88}{\textbf{)}}};
\draw (432+3,290+1.5) node [anchor=north west][inner sep=0.75pt]   [align=left] {\textcolor[rgb]{0.29,0.56,0.88}{\textbf{[}}};

\draw (440,290+1.5) node [anchor=north west][inner sep=0.75pt]   [align=left] {\textcolor[rgb]{0.29,0.56,0.88}{\textbf{)}}};
\draw (440+3,290+1.5) node [anchor=north west][inner sep=0.75pt]   [align=left] {\textcolor[rgb]{0.29,0.56,0.88}{\textbf{[}}};

\draw (462,290+1.5) node [anchor=north west][inner sep=0.75pt]   [align=left] {\textcolor[rgb]{0.29,0.56,0.88}{\textbf{)}}};
\draw (462+3,290+1.5) node [anchor=north west][inner sep=0.75pt]   [align=left] {\textcolor[rgb]{0.29,0.56,0.88}{\textbf{[}}};

\draw (480-1.5,290+1.5) node [anchor=north west][inner sep=0.75pt]   [align=left] {\textcolor[rgb]{0.29,0.56,0.88}{\textbf{]}}};
\draw (480,290+1.5) node [anchor=north west][inner sep=0.75pt]   [align=left] {\textcolor[rgb]{0.29,0.56,0.88}{\textbf{(}}};

\draw (515,290+1.5) node [anchor=north west][inner sep=0.75pt]   [align=left] {\textcolor[rgb]{0.29,0.56,0.88}{\textbf{)}}};
\draw (515+3,290+1.5) node [anchor=north west][inner sep=0.75pt]   [align=left] {\textcolor[rgb]{0.29,0.56,0.88}{\textbf{[}}};

\draw (544-3.5,290+1.5) node [anchor=north west][inner sep=0.75pt]   [align=left] {\textcolor[rgb]{0.29,0.56,0.88}{\textbf{]}}};

\draw (264-3.5,310+1.5) node [anchor=north west][inner sep=0.75pt]   [align=left] {\textcolor[rgb]{0.29,0.56,0.88}{\textbf{[}}};

\draw (277.5-1.5,310+1.5) node [anchor=north west][inner sep=0.75pt]   [align=left] {\textcolor[rgb]{0.29,0.56,0.88}{\textbf{]}}};
\draw (277.5,310+1.5) node [anchor=north west][inner sep=0.75pt]   [align=left] {\textcolor[rgb]{0.29,0.56,0.88}{\textbf{(}}};

\draw (285-1.5,310+1.5) node [anchor=north west][inner sep=0.75pt]   [align=left] {\textcolor[rgb]{0.29,0.56,0.88}{\textbf{]}}};
\draw (285,310+1.5) node [anchor=north west][inner sep=0.75pt]   [align=left] {\textcolor[rgb]{0.29,0.56,0.88}{\textbf{(}}};

\draw (292.5,310+1.5) node [anchor=north west][inner sep=0.75pt]   [align=left] {\textcolor[rgb]{0.29,0.56,0.88}{\textbf{)}}};
\draw (292.5+3,310+1.5) node [anchor=north west][inner sep=0.75pt]   [align=left] {\textcolor[rgb]{0.29,0.56,0.88}{\textbf{[}}};

\draw (310-1.5,310+1.5) node [anchor=north west][inner sep=0.75pt]   [align=left] {\textcolor[rgb]{0.29,0.56,0.88}{\textbf{]}}};
\draw (310,310+1.5) node [anchor=north west][inner sep=0.75pt]   [align=left] {\textcolor[rgb]{0.29,0.56,0.88}{\textbf{(}}};

\draw (410-1.5,310+1.5) node [anchor=north west][inner sep=0.75pt]   [align=left] {\textcolor[rgb]{0.29,0.56,0.88}{\textbf{]}}};
\draw (410,310+1.5) node [anchor=north west][inner sep=0.75pt]   [align=left] {\textcolor[rgb]{0.29,0.56,0.88}{\textbf{(}}};

\draw (432,310+1.5) node [anchor=north west][inner sep=0.75pt]   [align=left] {\textcolor[rgb]{0.29,0.56,0.88}{\textbf{)}}};
\draw (432+3,310+1.5) node [anchor=north west][inner sep=0.75pt]   [align=left] {\textcolor[rgb]{0.29,0.56,0.88}{\textbf{[}}};

\draw (440,310+1.5) node [anchor=north west][inner sep=0.75pt]   [align=left] {\textcolor[rgb]{0.29,0.56,0.88}{\textbf{)}}};
\draw (440+3,310+1.5) node [anchor=north west][inner sep=0.75pt]   [align=left] {\textcolor[rgb]{0.29,0.56,0.88}{\textbf{[}}};

\draw (462,310+1.5) node [anchor=north west][inner sep=0.75pt]   [align=left] {\textcolor[rgb]{0.29,0.56,0.88}{\textbf{)}}};
\draw (462+3,310+1.5) node [anchor=north west][inner sep=0.75pt]   [align=left] {\textcolor[rgb]{0.29,0.56,0.88}{\textbf{[}}};

\draw (480-1.5,310+1.5) node [anchor=north west][inner sep=0.75pt]   [align=left] {\textcolor[rgb]{0.29,0.56,0.88}{\textbf{]}}};
\draw (480,310+1.5) node [anchor=north west][inner sep=0.75pt]   [align=left] {\textcolor[rgb]{0.29,0.56,0.88}{\textbf{(}}};

\draw (507.5-1.5,310+1.5) node [anchor=north west][inner sep=0.75pt]   [align=left] {\textcolor[rgb]{0.29,0.56,0.88}{\textbf{)}}};
\draw (507.5+1.5,310+1.5) node [anchor=north west][inner sep=0.75pt]   [align=left] {\textcolor[rgb]{0.29,0.56,0.88}{\textbf{[}}};

\draw (515,310+1.5) node [anchor=north west][inner sep=0.75pt]   [align=left] {\textcolor[rgb]{0.29,0.56,0.88}{\textbf{)}}};
\draw (515+3,310+1.5) node [anchor=north west][inner sep=0.75pt]   [align=left] {\textcolor[rgb]{0.29,0.56,0.88}{\textbf{[}}};

\draw (522.5,310+1.5) node [anchor=north west][inner sep=0.75pt]   [align=left] {\textcolor[rgb]{0.29,0.56,0.88}{\textbf{)}}};
\draw (522.5+3,310+1.5) node [anchor=north west][inner sep=0.75pt]   [align=left] {\textcolor[rgb]{0.29,0.56,0.88}{\textbf{[}}};

\draw (544-3.5,310+1.5) node [anchor=north west][inner sep=0.75pt]   [align=left] {\textcolor[rgb]{0.29,0.56,0.88}{\textbf{]}}};

\draw (200,220) node [anchor=north west][inner sep=0.75pt]    {$\mathcal{T}_{\delta ,1}$, {\textcolor[rgb]{0.96,0.65,0.14}{$\mathcal{T}_{\delta^{r},1}$}}};
\draw (200,240) node [anchor=north west][inner sep=0.75pt]    {$\mathcal{T}_{\delta ,2}$, {\textcolor[rgb]{0.96,0.65,0.14}{$\mathcal{T}_{\delta^{r},2}$}}};
\draw (200,260) node [anchor=north west][inner sep=0.75pt]    {$\mathcal{T}_{\delta ,3}$, {\textcolor[rgb]{0.96,0.65,0.14}{$\mathcal{T}_{\delta^{r},3}$}}};

\draw (240,290) node [anchor=north west][inner sep=0.75pt]  [color={rgb, 255:red, 74; green, 144; blue, 226 }  ,opacity=1 ]  {$\mathcal{J}_{0}$};
\draw (240,310) node [anchor=north west][inner sep=0.75pt]  [color={rgb, 255:red, 74; green, 144; blue, 226 }  ,opacity=1 ]  {$\mathcal{J}_{1}$};

\end{tikzpicture}
\end{center}
\caption{An illustration of the construction of $\mathcal{J}_{1}$ for $\dtt=4$. Final intervals are colored in green, while temporary intervals are colored in red.}
\label{fig:4.2}
\end{figure}

The main task of defining the partition $\mathcal{J}$ and $\Par$ is complete. Now, as discussed in~\S\ref{subsubsec:4.1.1}, we define a refined partition $\mathcal{J}^{\prime}$ with an orientation assigned to each sub-interval, and show that it meets our requirements.
By Proposition~\ref{proposition:5.1}, we have 
\[\att_{t} x\in \cusp_{\delta^{\prime}}^{+}(\Par(U))\]
for any $t\in U\in \mathcal{J}$. Therefore, by Proposition~\ref{proposition:2.2}, each $U\in \mathcal{J}$ can be divided to $\ll 1$ many intervals during each the trajectory of $x$ has either a fixed orientation $[w]_{\Par(U)}\in W_{\Par(U),\att}$ with respect to $\Par(U)$, or does not have an orientation at all. The collection of all these sub-intervals of elements $U\in \mathcal{J}$ form a refinement $\mathcal{J}^{\prime}$ of $\mathcal{J}$ 
with the following property. For any $U\in \mathcal{J}^{\prime}$, let $V\in\mathcal{J}$ be the unique element so that $U\subseteq V$, and define $\Par(U)=\Par(V)$. Then, there is
a function
\[\Weyl:\ \mathcal{J}^{\prime}\to \bigcup_{P\in\mathcal{P}}W_{P,\att}\cup\{\emptyset\}\]
so that 
\[\att_{t}x\in \cusp_{\delta^{\prime}}^{+}(\Par(U),\Weyl(U))\]
for all $t\in U\in\mathcal{J}^{\prime}$.

We now show that this partition indeed meets the requirements we set in~\S\ref{subsubsec:4.1.1}.
\begin{proposition}\label{proposition:4.4_n}
Let $\delta,\delta^{\prime},r,N$ be as before.
Then
    \begin{enumerate}
    \mathitem\label{item:4.4_1}
        \[|\mathcal{J}^{\prime}|\ll_{\att}  \frac{1}{r^{\dtt-1}|\log\delta|}N 
         \]
        \mathitem\label{item:4.4_2} \[\sum_{U\in\mathcal{J}}\Big[\err(\att_{\inf U}x,\Par(U))+\err(\att_{\sup U}x,\Par(U))\Big]\ll_{\att} 
         rN
         .\]
    \end{enumerate}
\end{proposition}
\begin{proof}
Clearly, \[|\mathcal{J}^{\prime}|\ll |\mathcal{J}|.\]
    As $\mathcal{J}=\mathcal{J}_{\dtt-1}$, 
    we deduce from Equation~\eqref{eq:4.7_n}
    that
        \[|\mathcal{J}|\leq |\mathcal{F}_{\delta^{r^{\dtt-1}}}|\ll_{\att}  N/(r^{\dtt-1}|\log\delta|),\]
    concluding item~\eqref{item:4.4_1}. 

    \medskip

    For~\eqref{item:4.4_2}, let $U\in\mathcal{J}$ be of degree $\deg(U)=m$ for some $m$.
    Then, as discussed in~\S\ref{subsubsec:4.1.2}, the fact (which follows from Proposition~\ref{proposition:5.1}) that
    \[\eta_{l}(\att_{t} x)\geq \delta^{r^{m}}\]
    for all 
    \[l\in\{1,\ldots,\dtt-1\}\smallsetminus\eta(\Par(U))\]
    for $t=\inf U,\sup U$, implies that
    \[\err(\att_t x,\Par(U))\ll |\log\delta^{r^{m}}|=r^{m}|\log\delta|\]
    for such $t$.

    Then, using Proposition~\ref{proposition:5.1}, it follows that
    \begin{align*}\sum_{U\in \mathcal{J}\cap\deg^{-1}(m)}&\Big[\err(\att_{\inf U}x,\Par(U))+\err(\att_{\inf U}x,\Par(U))\Big]\ll r^{m}|\log\delta|\cdot|\mathcal{J}\cap\deg^{-1}(m)|\\&
    \ll_{\att} r^{m}|\log\delta|\cdot N/(r^{m-1}|\log\delta|)
    =rN
    .\end{align*}
    As this holds for all $m$, the result follows.
    
\end{proof}

\begin{corollary}\label{corollary:4.5}
For any $x\in X$,
    \[\Bigl\|\!\!\!\!\!\!\!\!\!\!\!\sum_{\substack{U\in \mathcal{J}^{\prime}:\\ \ \emptyset\not=\Weyl(U)=[w]_{\Par(U)}}}\hspace{-1cm}|U|\cdot \pi_{\Par(U)}(\upalpha^{w})\Bigr\|\leq  O_{\att}(rN)+\|\height(\att_{N}x)-\height(\att_{-N}x)\|.\]
\end{corollary}
\begin{proof}
By the discussion in~\S\ref{subsubsec:4.1.1}, we have
\begin{align*}\Bigl\|\!\!\!\!\!\!\!\!\!\!\!\sum_{\substack{U\in \mathcal{J}^{\prime}:\\ \ \emptyset\not=\Weyl(U)=[w]_{\Par(U)}}}\hspace{-1cm}|U|\cdot \pi_{\Par(U)}(\upalpha^{w})\Bigr\|&
\leq \sum_{U\in\mathcal{J}}\Big[\err(\att_{\inf U_j}x,\Par(U_j))+\err(\att_{\sup U_j}x,\Par(U_{j}))\Big]\\&
+\|\height(\att_{N}x)-\height(\att_{-N}x)\|+O_{\att}(|\mathcal{J}^{\prime}|)
.\end{align*}
Thus the corollary follows from Proposition~\ref{proposition:4.4_n}, considering that $r$ was chosen so that
\[r^{\dtt}>r^{\dtt+1}>\frac{|\log\delta^{\prime}|}{|\log\delta|}>\frac{1}{|\log\delta|}.\]
\end{proof}

\subsection{Codings and discrete partitions}

When computing entropy, we will consider the more convenient discrete dynamical system obtained by the time-one map $\att\coloneqq\att_{1}$, instead of the continuous flow.
The analogous data to the partition of $[-N,N]$ and the assignments of $\Par$ and $\Weyl$ is called a coding.
\begin{definition}\label{def:13}
Let $N\in \mathbb{N}$. A 
function 
\[\mathcal{C}:\ \{-N,\ldots,N\}\to\{(P,[w]_P):\ P\in\mathcal{P},\ [w]_P\in W_{P,\att}\}\cup\big(\mathcal{P}\times\{\emptyset\}\big)\]
is called a \textit{coding} over the time interval $[-N,N]$.
The set of all codings over $[-N,N]$ is denoted by $\mathcal{C}_{N}$.
\end{definition}

\begin{definition}
Let $x\in X$, $N\in \mathbb{N}$, and $\delta,\delta^{\prime},r>0$ which satisfy $\delta^{1/2}<\delta^{\prime}<\eta_0$, $|\log\delta|<N$, and $(\frac{\log\delta^{\prime}}{\log\delta})^{1/(\dtt+1)}<r<1$.
Let $\mathcal{J}^{\prime}$, $\Par$ and $\Weyl$ be as constructed in~\S\ref{subsubsec:4.1.4}.
Then, define the coding
$\mathcal{C}_{x}^{\delta,\delta^{\prime},r,N}\in\mathcal{C}_{N}$ by
\[\mathcal{C}_{x}^{\delta,\delta^{\prime},r,N}(n)=(\Par(U_n),\Weyl(U_n))\]
for the unique $U_{n}\in\mathcal{J}^{\prime}$ so that $n\in U_{n}$.
\end{definition}
\begin{definition}
    For any subset $U\subseteq\mathbb{R}$, let
    \[U_{\mathbb{Z}}=U\cap\mathbb{Z}.\]
    For any collection $\mathcal{A}$ of real subsets, let
    \[\mathcal{A}_{\mathbb{Z}}=\{U_{\mathbb{Z}}:\ U\in\mathcal{A}\}.\]
    If $U\subseteq\mathbb{R}$ is an interval, we still consider $U_{\mathbb{Z}}$ as an  interval in terms of terminology.
\end{definition}
\begin{remark}\label{remark:4.6}
    For an interval $U\subseteq\mathbb{R}$, 
    \[\Big||U|-|U_{\mathbb{Z}}|\Big|\leq 1.\]
\end{remark}
\begin{corollary}[of Corollary~\ref{corollary:4.5}]\label{corollary:4.7_n}
Let $\mathcal{C}=\mathcal{C}_{x}^{\delta,\delta^{\prime},r,N}$. Then 
    \[\Bigl\|\!\!\!\sum_{\substack{P\in \mathcal{P},\\ [w]_{P}\in W_{P,\att}}}\bigl|\mathcal{C}^{-1}(P,[w]_{P})\bigr|\cdot \pi_P(\upalpha^{w})\Bigr\|\leq O_{\att}(rN)+\|\height(\att_{N}x)-\height(\att_{-N}x)\|.
    \]
\end{corollary}

\begin{proof}
Using
Proposition~\ref{proposition:4.4_n} together with Corollary~\ref{corollary:4.5} and Remark~\ref{remark:4.6}, we have    \begin{align*}\Bigl\|\!\!\!\sum_{\substack{P\in \mathcal{P},\\ [w]_{P}\in W_{P,\att}}}\bigl|\mathcal{C}^{-1}(P,[w]_{P})\bigr|\cdot \pi_P(\upalpha^{w})\Bigr\|&\leq O_{\att}(|\mathcal{J}^{\prime}|)+\Bigl\|\!\!\!\sum_{\substack{U\in \mathcal{J}^{\prime}:\\ \ \emptyset\not=\Weyl(U)=[w]_{\Par(U)}}}\hspace{-1cm}|U|\cdot \pi_{\Par(U)}(\upalpha^{w})\Bigr\|
    \\&
    \leq 
    O_{\att}(rN)+\|\height(\att_{N}x)-\height(\att_{-N}x)\|.
    \end{align*}
    \end{proof}

\subsection{The relation to the cusp regions}

In order to prove Theorem~\ref{theorem:1.1new}, we need to understand the relation between the parabolic groups assigned by a coding of a trajectory, and the regions of the cusp where the trajectory in fact passes through.
This relation is not as simple as one could hope for, due to our complicated construction of $\mathcal{J}$, where each interval $U\in\mathcal{J}$ is essentially defined using a different parameter $\delta^{r^{\deg U-1}}$.  
We study this relation using the cusp regions $\cusp_{\delta,\delta^{\prime}}(P,Q,[w]_{Q})$ and eventually obtain an entropy bound using these regions in Theorem~\ref{theorem:1.1}.
Constructing an open cover as in Theorem~\ref{theorem:1.1new} would follow easily from Theorem~\ref{theorem:1.1}, as will be shown in~\S\ref{subsec:8.3}.

\begin{lemma}\label{proposition:5.2}
Let $x,\delta,\delta^{\prime},r,N$ as before.
Let $n\in U\in \mathcal{J}^{\prime}$.
Let $Q\subseteq P\in\mathcal{P}$ be such that $\att_{n} x\in \cusp_{\delta,\delta^{\prime}}(P,Q)$.
Then
\begin{enumerate}
    \item\label{item:5.2_1}
    $Q\subseteq \Par(U)\subseteq P$.
    \item\label{item:5.2_2}
    If $\att_n x\in \cusp_{\delta,\delta^{\prime}}(P,Q,[w]_{Q})$ for some $[w]_{Q}\in W_{Q,\att}$, then
    $\Weyl(U)=[w]_{\Par(U)}$.
    \item\label{item:5.2_3} 
    \[\sum_{P\in\mathcal{P}}
    \bigl|(\mathcal{C}_{x}^{\delta,\delta^{\prime},r,N})^{-1}(P,\emptyset)\bigr|\ll_{\att} 
    \frac{1}{r^{\dtt-1}|\log\delta|}N
    .\]
\end{enumerate}
\end{lemma}
\begin{proof}
Item~\ref{item:5.2_1} follows from Proposition~\ref{proposition:5.1}.
Item~\ref{item:5.2_2} follows easily from the fact that the flag $(V_{l}(\att_n x))_{l\in \eta(Q)}$ is a refinement of the flag $(V_{l}(\att_n x))_{l\in \eta(\Par(U))}$.
For Item~\ref{item:5.2_3}, we defined $\mathcal{J}^{\prime}$ as a refinement of $\mathcal{J}$ obtained by splitting each interval $V\in \mathcal{J}$ to $\ll 1$ sub-intervals $U\subseteq V$, each with an assignment by the function $\Weyl$, using Proposition~\ref{proposition:2.2}. For any such $V$, the total length of its sub-intervals $U$ with $\Weyl(U)=\emptyset$ is $\ll_{\att} 1$.
Taking into account Remark~\ref{remark:4.6}, and using Proposition~\ref{proposition:4.4_n}, 
we obtain 
    \[\sum_{P}\bigl|(\mathcal{C}_{x}^{\delta,\delta^{\prime},r,N})^{-1}(P,\emptyset)\bigr|\ll_{\att} |\mathcal{J}|\ll_{\att}\frac{1}{r^{\dtt-1}|\log\delta|}N,\]
    as required.
\end{proof}

\begin{proposition}\label{corollary:5.3}
Let $x\in X$ and $\mathcal{C}=\mathcal{C}_{x}^{\delta,\delta^{\prime},r,N}$.
For any $Y\subseteq X$, let 
\[T_{Y}^{x}=\{n\in [-N,N]_{\mathbb{Z}}:\ \att_{n} x\in Y\}.\]

Let \[\chi:\ \bigcupdot_{P\in \mathcal{P}}W_{P,\att}\to \mathbb{R}\] be a
function with a non-negative maximal value.
Then 
\begin{multline*}\sum_{H\in \mathcal{P}}\sum_{[\tau]_{H}\in W_{H,\att}}\bigl|\mathcal{C}^{-1}(H,[\tau]_H)\bigr|\chi([\tau]_{H})\leq\\
\leq\sum_{\substack{P,Q\in\mathcal{P}:\\ Q\subseteq P}}\sum_{[w]_{Q}\in W_{Q,\att}}\bigl|T^{x}_{\cusp_{\delta,\delta^{\prime}}(P,Q,[w]_{Q})}\bigr|\max_{\substack{H\in\mathcal{P}:\\ Q\subseteq H\subseteq P}}\chi([w]_{H})
+\bigl|T^{x}_{\cusp_{\delta^{\prime}}(X,\emptyset)}\bigr|\max_{\substack{H\in \mathcal{P},\\ [w]_{H}\in W_{H,\att}}}\chi([w]_{H}).
\end{multline*}

\end{proposition}
\begin{remark}
The 
function
$\chi$ could simply be the entropy $\chi([w]_{P})=h(P,\att^{w})$, but could also be include an additional linear functional $\chi([w]_P)=h(P,\att^{w})-\phi(\pi_P(\upalpha^{w}))$,
as will be in later parts of this paper. 
\end{remark}
\begin{proof}
For simplicity of notation, we will denote
\[T_{P,Q,\sigma}=T^{x}_{\cusp_{\delta,\delta^{\prime}}(P,Q,\sigma)}\]
for any $P,Q\in\mathcal{P}$ with $Q\subseteq P$ and $\sigma\in W_{Q,\att}\cup\{\emptyset\}$.

Recall that
\[\left\{\cusp_{\delta,\delta^{\prime}}(P,Q,\sigma):\ P,Q\in\mathcal{P},\ Q\subseteq P,\ \sigma\in W_{Q,\att}\cup\{\emptyset\}\right\}\]
is a partition of $X$.
Then, by Lemma~\ref{proposition:5.2},
\[\mathcal{C}^{-1}(H,[\tau]_{H})\subseteq\bigcupdot_{\substack{P,Q\in\mathcal{P}:\\  Q\subseteq H\subseteq P}}(T_{P,Q,\emptyset}\cupdot\bigcupdot_{\substack{[w]_{Q}\in W_{Q,\att}:\\ [w]_{H}=[\tau]_{H}}}T_{P,Q,[w]_{Q}})\]
Then

\begin{align*}
\numberthis\label{eq:5.1}
\Big|\mathcal{C}^{-1}(H,[\tau]_{H})\Big|\chi([\tau]_{H})&
= \sum_{\substack{P,Q\in\mathcal{P}:\\  Q\subseteq H\subseteq P}}\Big|T_{P,Q,\emptyset}\cap \mathcal{C}^{-1}(H,[\tau]_{H})\Big|\chi([\tau]_{H})\\&
+\sum_{\substack{P,Q\in\mathcal{P}:\\  Q\subseteq H\subseteq P}}\sum_{\substack{[w]_{Q}\in W_{Q,\att}:\\ [w]_{H}=[\tau]_{H}}} \Big|T_{P,Q,[w]_{Q}}\cap \mathcal{C}^{-1}(H,[w]_{H})\Big|\chi([w]_{H}).
\end{align*}
Note that the following sums are the same, only in different order,  
\[\sum_{H\in \mathcal{P}}\sum_{\substack{P,Q\in\mathcal{P}:\\  Q\subseteq H\subseteq P}}=\sum_{\substack{P,Q\in\mathcal{P}:\\  Q\subseteq P}}\sum_{\substack{H\in \mathcal{P}:\\ Q\subseteq H\subseteq P}},\]
and that for $Q\subseteq H$ 
\[\bigcupdot_{[\tau]_{H}\in W_{H,\att}}\Big\{[w]_{Q}\in W_{Q,\att}:\  [w]_{H}=[\tau]_H\Big\}=W_{Q,\att}.\]
So summing Equation~\eqref{eq:5.1} over all $H$ and $[\tau]_{H}$ in fact shows that
\begin{align*}
\sum_{H\in \mathcal{P}}\sum_{[\tau]_{H}\in W_{H,\att}}
\Big|\mathcal{C}^{-1}(H,[\tau]_{H})\Big|\chi([\tau]_{H})&
\leq \sum_{\substack{P,Q\in \mathcal{P}:\\Q\subseteq P}}\Big|T_{P,Q,\emptyset}
\Big|\max_{\substack{H\in \mathcal{P},\\ [w]_{H}\in W_{H,\att}}}\chi([w]_{H})
\\&\hspace{-5cm}
+\sum_{\substack{P,Q\in \mathcal{P}:\\Q\subseteq P}}\sum_{\substack{H\in\mathcal{P}:\\Q\subseteq H\subseteq P}}\sum_{[w]_{Q}\in W_{Q,\att}} \Big|T_{P,Q,[w]_{Q}}\cap \mathcal{C}^{-1}(H,[w]_{H})\Big|\chi([w]_{H})
\\&\hspace{-5cm}
\leq \Big|T^{x}_{\cusp_{\delta^{\prime}}(X,\emptyset)}\Big|\max_{\substack{H\in \mathcal{P},\\ [w]_{H}\in W_{H,\att}}}\chi([w]_{H})
+\sum_{\substack{P,Q\in \mathcal{P}:\\Q\subseteq P}}\sum_{[w]_{Q}\in W_{Q,\att}}\Big|T_{P,Q,[w]_{Q}}\Big|\max_{\substack{H\in\mathcal{P}:\\ Q\subseteq H\subseteq P}}\chi([w]_{H}),
\end{align*}
where we used the fact that
\[T_{P,Q,[w]_{Q}}\subseteq \bigcupdot_{\substack{H\in\mathcal{P}:\\ Q\subseteq H\subseteq P}} \mathcal{C}^{-1}(H,[w]_{H}).\]
This concludes the proof.
\end{proof}

\subsection{The number of codings}
The following proposition shows that there is only a small number of codings which can be obtained as a result of our construction in~\S\ref{subsec:coding}.
\begin{proposition}\label{lemma:num_of_V}
For $0<\delta<\delta^{\prime}<\eta_0$ small enough, we have 
\begin{equation*}
\left| \{\mathcal{C}_{x}^{\delta,\delta^{\prime},r,N}:\ x\in X\} \right |\leq 
\exp\Big(\kappa\frac{\log\big( r^{\dtt-1}|\log\delta|\big)}{r^{\dtt-1}|\log\delta|}N\Big),
\end{equation*}
for some constant $0<\kappa\ll_{\att} 1$.
\end{proposition}

\begin{proof}

Let $\mathcal{C}=\mathcal{C}_{x}^{\delta,\delta^{\prime},r,N}$ for some $x\in X$.
The coding is determined by the partition $\mathcal{J}^{\prime}_{\mathbb{Z}}$ and the functions $\Par,\Weyl$, defined in the construction before.
For simplicity of notation, we will denote by $\Opt(Y)$ the number of different values the variable $Y$ could take, for instance  $\Opt(\mathcal{J}^{\prime}_{\mathbb{Z}})$ is the number of possible outcomes for a partition $\mathcal{J}^{\prime}_{\mathbb{Z}}$.

\medskip

We start by counting the number of different possible partitions $\mathcal{J}_{\mathbb{Z}}$, i.e.\ bounding $\Opt(\mathcal{J}_{\mathbb{Z}})$.
For any collection of sets $\mathcal{A}$, let
\[E^{-}_{\mathcal{A}}=\{\inf U:\ U\in \mathcal{A}\},\]
\[E^{+}_{\mathcal{A}}=\{\sup U:\ U\in \mathcal{A}\}\]
and
\[Z_{\mathcal{A}}=\{\lfloor r\rfloor:\ r\in E^{+}_{\mathcal{A}}\cup E^{-}_{\mathcal{A}}\}.\]

First of all, note that the sets $E_{\mathcal{J}_{\mathbb{Z}}}^{+},E_{\mathcal{J}_{\mathbb{Z}}}^{-}$ determine the partition $\mathcal{J}_{\mathbb{Z}}$, so
\[\Opt(\mathcal{J}_{\mathbb{Z}})\leq \Opt(E_{\mathcal{J}_{\mathbb{Z}}}^{+})\Opt(E_{\mathcal{J}_{\mathbb{Z}}}^{-}).\]
Secondly, it is easy to see that for any possible $\mathcal{J}$,
\begin{equation*}
E_{\mathcal{J}_{\mathbb{Z}}}^{+}\subseteq \{\lfloor r\rfloor, \lfloor r\rfloor- 1:\ r\in E_{\mathcal{J}}^{+}\}\subseteq\{m,m- 1:\ m\in Z_{\mathcal{J}}\},\end{equation*}
and similarly
\[E_{\mathcal{J}_{\mathbb{Z}}}^{-}\subseteq\{m,m+ 1:\ m\in Z_{\mathcal{J}}\}.\]
So it follows that $E_{\mathcal{J}_{\mathbb{Z}}}^{+}$ and $E_{\mathcal{J}_{\mathbb{Z}}}^{-}$ are each determined by $Z_{\mathcal{J}}$ up to
\[\leq 2^{2|Z_J|}\leq 2^{2(|\mathcal{J}|+1)}\]
many options. Therefore
\begin{align*}\label{eq:n1}\numberthis
\Opt(\mathcal{J}_{\mathbb{Z}})&\leq \exp\Big(O(1)|\mathcal{J}|\Big)\cdot\Opt(Z_{\mathcal{J}})^{2}\leq \exp\Big(\frac{O_{\att}(1)}{r^{\dtt-1}|\log\delta|}N\Big)\cdot\Opt(Z_{\mathcal{J}})^{2}
,
\end{align*}
where Proposition~\ref{proposition:4.4_n} was used. 

Furthermore, note that 
for any possible $\mathcal{J}$, the partition $\mathcal{F}_{\delta^{r^{\dtt-1}}}$ is a refinement of $\mathcal{J}$ and so
\[Z_{\mathcal{J}}\subseteq Z_{\mathcal{F}_{\delta^{r^{\dtt-1}}}}.\] Therefore, using Equation~\eqref{eq:4.7n}, $Z_{\mathcal{J}}$ is determined by $Z_{\mathcal{F}_{\delta^{r^{\dtt-1}}}}$ up to
\[\leq 2^{|Z_{\mathcal{F}_{\delta^{r^{\dtt-1}}}}|}\leq \exp(\frac{O_{\att}(1)}{r^{\dtt-1}|\log\delta|}N)\]
options,
i.e.\
\begin{equation}\label{eq:n2}\Opt(Z_{\mathcal{J}})\leq
\exp(\frac{O_{\att}(1)}{ r^{\dtt-1}|\log\delta|}N)\cdot \Opt(Z_{\mathcal{F}_{\delta^{r^{\dtt-1}}}})
.\end{equation}
Lastly, as
\[Z_{\mathcal{F}_{\delta^{r^{\dtt-1}}}}=\bigcup_{l=1}^{\dtt-1}Z_{\mathcal{T}_{\delta^{r^{\dtt-1}},l}},\]
we have
\begin{equation}\label{eq:n3}\Opt(Z_{\mathcal{F}_{\delta^{r^{\dtt-1}}}})\leq\prod_{l=1}^{\dtt-1}
\Opt(Z_{\mathcal{T}_{\delta^{r^{\dtt-1}},l}}).\end{equation}
All together, we find from Equations~\eqref{eq:n1}-\eqref{eq:n3}
\begin{equation}\label{eq:4.9}
\Opt(\mathcal{J}_{\mathbb{Z}})\leq
\exp(\frac{O_{\att}(1)}{r^{\dtt-1}|\log\delta|}N)
\max_{l=1,\ldots,\dtt-1}\Opt(Z_{\mathcal{T}_{\delta^{r^{\dtt-1}},l}})^{2(\dtt-1)}.
\end{equation}

Let us now bound $\Opt(Z_{\mathcal{T}_{\delta^{r^{\dtt-1}},l}})$.
As discussed in~\S\ref{subsubsec:4.1.3}, $\mathcal{T}_{\delta^{r^{\dtt-1}},l}$ is a union of intervals of length $>\kappa_1 r^{\dtt-1}\log\frac{\delta^{\prime}}{\delta}$ each, for 
some $\kappa_1\in(0,1)$ which depends on $\att$.
It follows, by length considerations,
that
for any integral interval $U\subseteq [-N,N]_{\mathbb{Z}}$ of cardinality $\leq \lceil\kappa_1 r^{\dtt-1}\log\frac{\delta^{\prime}}{\delta}\rceil $
(hence with diameter $\lceil\kappa_1 r^{\dtt-1}\log\frac{\delta^{\prime}}{\delta}\rceil-1$), 
we have
\[|U\cap Z_{\mathcal{T}_{\delta^{r^{\dtt-1}},l}}|\leq 2.\]
The number of possibilities for these at most two intersection points is $|U|^{2}+1$.
As we can divide $[-N,N]_{\mathbb{Z}}$ to
\[\leq \big\lceil (2N+1)/\lceil\kappa_1 r^{\dtt-1}\log\frac{\delta^{\prime}}{\delta}\rceil\big\rceil\ll_{\att}
\frac{1}{r^{\dtt-1}|\log\delta|}N
\]
intervals of cardinality
\[\leq \lceil\kappa_1 r^{\dtt-1}\log\frac{\delta^{\prime}}{\delta}\rceil ,\]
we deduce
\begin{align*}\Opt(Z_{\mathcal{T}_{\delta^{r^{\dtt-1}},l}})&\leq (\lceil\kappa_1 r^{\dtt-1}\log\frac{\delta^{\prime}}{\delta}\rceil^{2}+1)^{O_{\att}(1)\cdot N/(r^{\dtt-1}|\log\delta|)}\\&
\leq\exp\Big(O_{\att}(1)\frac{\log\big(r^{\dtt-1}|\log\delta|\big)}{r^{\dtt-1}|\log\delta|}N\Big).
\end{align*}

Then we find from Equation~\eqref{eq:4.9} that in total
\begin{equation}\label{eq:4.11}
\Opt(\mathcal{J}_{\mathbb{Z}})\leq
\exp\Big(O_{\att}(1)\cdot\frac{\log\big(r^{\dtt-1}|\log\delta|\big)}{r^{\dtt-1}|\log\delta|}N\Big).
\end{equation}

\medskip

Next, we now count the number of codings $\mathcal{C}$ with a given partition $\mathcal{J}_{\mathbb{Z}}$, i.e\ the number to refine $\mathcal{J}_{\mathbb{Z}}$ to $\mathcal{J}^{\prime}_{\mathbb{Z}}$ and then assign the functions $\Par, \Weyl$.

By construction, as was previously discussed, $\mathcal{J}^{\prime}_{\mathbb{Z}}$ is obtained from $\mathcal{J}_{\mathbb{Z}}$ by splitting each interval $V\in\mathcal{J}_{\mathbb{Z}}$ to sub-intervals $U_1,\ldots U_k$, each with an assignment by the function $\Weyl$. The sequence $(\Weyl(U_j))_{j=1}^{k}$ is a strictly increasing sequence in $W_{\Par(V),\att}$ with respect to some some partial order, with possible occurrences of $\emptyset$ between the relative Weyl group elements in the sequence (c.f.\ Proposition~\ref{proposition:2.2}). 
So the length of this sequence is  \[k\leq 2|W_{P,\att}|+1\leq 2\dtt!+1.\]
It follows that the number of possible restrictions $\mathcal{C}|_{V}$ is at most
\[2^{\dtt-1}\cdot(\dtt!+1)^{2\dtt!+1}\cdot|V|^{2\dtt!}\leq ((\dtt+1) !)^{2\dtt !} |V|^{2\dtt !},\]
as follows. 
This bound is the product of the $|\mathcal{P}|=2^{\dtt-1}$ many different values for $\Par(V)$, the up to $(\dtt!+1)^{2\dtt!+1}$ sequences $\{\Weyl(U_j)\}_{j=1}^{k}$ of elements of $W_{\Par(V),\att}\cup\{\emptyset\}$, and most importantly the term $|V|^{2\dtt!}$ which bounds the number of ways an integer interval of length $|V|$ can be divided to $2\dtt!+1$ sub-intervals.

As this holds for all $V\in \mathcal{J}_{\mathbb{Z}}$, we find that 
the coding
$\mathcal{C}$ is one of 
\begin{align*}
\numberthis\label{eq:5.2}
\leq &
\prod_{V\in \mathcal{J}_{\mathbb{Z}}}((\dtt+1)! \cdot |V|)^{2\dtt!}
\leq
\big(\frac{\sum_{V\in \mathcal{J}_{\mathbb{Z}}} (\dtt+1)!\cdot |V|}{|\mathcal{J}_{\mathbb{Z}}|}\big)^{2\dtt!\cdot |\mathcal{J}_{\mathbb{Z}}|}\leq \big(\frac{(\dtt+2)!\cdot N}{|\mathcal{J}_{\mathbb{Z}}|}\big)^{2\dtt!\cdot |\mathcal{J}_{\mathbb{Z}}|}.
\end{align*}
different many possibilities, where we used the inequality of arithmetic and geometric means. Let us further bound this term.
Note that by
Proposition~\ref{proposition:4.4_n}, 
together with the assumption $r>(\frac{|\log\delta^{\prime}|}{|\log\delta|})^{1/(\dtt+1)}$,
we have
\begin{equation}\label{eq:4.13}|\mathcal{J}_{\mathbb{Z}}|\ll_{\att}   \frac{1}{ r^{\dtt-1}|\log\delta|}N< N.\end{equation}
Consider the $\mathbb{R}^{+}$ valued function $x\mapsto (z/x)^{x}$ for some fixed $z>0$. It is monotonically increasing for $x<z/e$. Then, 
we deduce from
Equation~\eqref{eq:5.2} and Equation~\eqref{eq:4.13} that there is some constant $\kappa_2>1$ which depends on $\att$, so that 
there are at most 
\begin{align*}
\numberthis\label{eq:4.14}
\leq& \big(\frac{(\dtt+2)!\cdot N}{|\mathcal{J}_{\mathbb{Z}}|}\big)^{2\dtt!\cdot|\mathcal{J}_{\mathbb{Z}}|}\leq \big(\frac{(\dtt+2)!\cdot \kappa_{2} N}{\kappa_2 N/(r^{\dtt-1}|\log\delta|)}\big)^{2\dtt!\kappa_2 N/(r^{\dtt-1}|\log\delta|)}\\&
\leq \exp\Big(O_{\att}(1)\frac{\log(O(1) r^{\dtt-1}|\log\delta|)}{ r^{\dtt-1}|\log\delta|}N\Big)
\leq 
\exp\Big(O_{\att}(1)\frac{ \log(r^{\dtt-1}|\log\delta|)}{ r^{\dtt-1}|\log\delta|}N\Big)
\end{align*}
different functions $\mathcal{C}$ for a given partition $\mathcal{J}_{\mathbb{Z}}$.

\medskip

Putting it all together, the number of functions $\mathcal{C}$ is bounded by the product of the number of possible partitions $\mathcal{J}_{\mathbb{Z}}$ with the number of possible of functions $\mathcal{C}$ for each $\mathcal{J}_{\mathbb{Z}}$. So we multiply the bounds in Equations~\eqref{eq:4.11} and \eqref{eq:4.14}, and find that 
\[\Opt(\mathcal{C})\leq \exp\Big(O_{\att}(1)\cdot\frac{\log\big(r^{\dtt-1}|\log\delta|)}{r^{\dtt-1}|\log\delta|}N\Big).\]
as required.
\end{proof}

\section{Bases in the cusp}\label{sec:bases}
We proceed towards the proof of the covering lemma in~\S\ref{sec:6}. In this section we construct bases for lattices in the cusp, that will simplify the proof later on.
\subsection{Metrics}\label{subsec:metrics_on_SL}
We will often require two notions of metrics on $G=\SL_{\dtt}(\mathbb{R})$, which induce the same topology on $G$. The first metric is the one induced from the standard Euclidean norm on $\mathbb{R}^{\dtt^{2}}$ (often referred to as the Frobenius norm), where we consider $\SL_{\dtt}(\mathbb{R})\subseteq M_{\dtt}(\mathbb{R})\cong \mathbb{R}^{\dtt^{2}}$. Namely, the distance between $g,g^{\prime}\in G$ is $\|g-g^{\prime}\|_{F}$, where for any $M\in M_{\dtt}(\mathbb{R})$ we denote 
\[\|M\|_{F}=\sqrt{\sum_{i,j=1}^{\dtt}M_{i,j}^{2}}.\]
The other notion is a right invariant metric, which we denote $d$. Such metric can be derived for $G=\SL_{\dtt}(\mathbb{R})$, as well as for  any connected Lie group, from a right-invariant Riemannian metric. Such construction can be found for instance in~\cite[\S~9.3.2]{einsiedler_ward_book}.

Each of our two metrics is useful for a different purpose. The key feature that allows us to use both of them is that they are strongly equivalent at a neighborhood of the identity (unit) matrix $e=I\in G$ (c.f.~\cite[\S~9.3.2]{einsiedler_ward_book}). 
That is to say that there are $c_{1},c_{2},\rho>0$, such that
\[c_1\|g-I\|_{F}\leq d(g,I)\leq c_2 \|g-I\|_{F}\] 
for all $g\in G$ with $d(g,I)<\rho$. 

To avoid any confusion, for $\rho>0$ and $H\leq G$ we will denote by $B_{\rho}^{H}$  the ball in $H$ with center $e$ and radius $\rho$ with respect to $d$, and by $D_{\rho}^{H}$ the same ball but with respect to the Frobenius metric.

\subsection{Bases}
When studying trajectories in the cusp, and in particular computing the entropy contribution of a parabolic subgroup $P$, 
it would be useful if the basis we use for the flag of unique small rational subspaces is compatible with the structure of $P$ and the orientation $[w]_{P}$. We develop such bases in this section. 

\begin{definition}
Let $\{0\}=V_0<V_1<\cdots<V_{k}<V_{k+1}=\mathbb{R}^{\dtt}$ be a flag of subspaces of $\mathbb{R}^{\dtt}$. 
A basis for 
$\{V_{j}\}_{j=0}^{k+1}$
is 
a set $\{v_i\}_{i=1}^{\dtt}$
such that $\{v_1,\ldots,v_{\dim V_l}\}$ is a basis for $V_l$, for all $1\leq l\leq k+1$.
\end{definition}

\begin{lemma}\label{lemma:5.1}
Let $\epsilon>0$ be sufficiently small. Let $\{W_j\}_{j=0}^{k+1}$ be a flag in $\mathbb{R}^{\dtt}$, and $\{w_i\}_{i=1}^{\dtt}$ an orthonormal basis for the flag.  Let $\{V_j\}_{j=0}^{k+1}$ be another flag, with $d_{\Gr}(V_j,W_j)<\epsilon$ for all $1\leq j\leq k$.
Then there exists a basis $\{v_i\}_{i=1}^{\dtt}$ for the flag $\{V_j\}_{j=0}^{k+1}$ of the form $v_i=w_i+u_i$ for $u_i\in W_{j}^{\perp}$ with $\|u_i\|\ll\epsilon$, where $j$ is the unique integer so that $\dim W_{j-1}<i\leq \dim W_{j}$.
\end{lemma}
\begin{remark}
Consider the matrix $v\coloneqq(v_1\vline \cdots \vline v_\dtt)$, represented in the basis $(w_1,\ldots,w_{\dtt})$. It is a lower-triangular block-matrix, with block sizes \[\dim V_{1}-\dim V_{0},\ldots,\dim V_{k+1}-\dim V_{k}.\]
The values are $1$ on the diagonal, and real numbers of absolute value $\ll \epsilon$ below the diagonal. This is the form we aim for in this lemma, and this picture is useful in order to follow the proof.
\end{remark}

\begin{proof}
We prove the claim by induction on $k$. For $k=0$, the claim is satisfied by the basis $\{v_i\}_{i=1}^{\dtt}$ where $v_i=w_i$ for all $i$.
Assume the claim holds for $k-1$ for some $k\geq 1$, and we prove it for $k$. 
We apply the claim for the flag \[0=V_0<V_1<\cdots<V_{k-1}<V_{k+1}=\mathbb{R}^{\dtt},\] whose length is shorter by $1$ compared to the original flag.
We obtain vectors $v_1,\ldots,v_{\dim V_{k-1}}$ which satisfy the required property. We complete them to a basis as follows.
First, for all 
\[\dim V_{k-1} < i\leq \dim V_k,\]
let $\tilde{v}_i\in V_{k}$ be a vector with $\|\tilde{v}_i-w_i\|<\epsilon$.
We use these vectors to construct the desired basis by the following process of $\dim V_k$ steps, In each step $j$, for all 
\[1\leq i\leq \dim V_k\]
we construct vectors $v_{i,j}$, by performing column operations. 
In this notation of $v_{i,j}$, the first index stands for the serial number of the vector, and the second index for the step number.
For the $0$th step, define 
\[v_{i,0}=\begin{cases} 
v_i & 1\leq i\leq \dim V_{k-1} \\
\tilde{v}_i & \dim V_{k-1}<i\leq \dim V_{k}
\end{cases}.\]
For the $j$th step, for $1\leq j\leq \dim V_{k}$, we define
\begin{equation*}
v_{i,j}=\begin{cases} 
v_{i,j-1} & 1\leq i\leq \dim V_{k-1}\\
v_{i,j-1}-\frac{\langle v_{i,j-1},w_j\rangle-\delta_{i,j}}{\langle v_{j,j-1},w_j\rangle}v_{j,j-1}
& \dim V_{k-1}<i\leq \dim V_{k} 
\end{cases}.
\end{equation*}
That is, for $i\leq\dim V_{k-1}$ we do not change the initial vector $\tilde{v}_{i}$ in any step, and for $i>\dim V_{k-1}$, at the $j$th step we either eliminate the component of $v_{i,j-1}$ in the direction of $w_j$ using $v_{j,j-1}$ (for $j\not=i$), or normalize this component (for $j=i$). 
Note that if $\epsilon$ is initially chosen to be small enough then $\langle v_{j,j-1},w_j\rangle\not=0$ for any $j$, hence the procedure is well defined.
This procedure can be thought of as just a sequence of column operations on the matrix $(v_{1,0}|\cdots |v_{\dim V_{k},0})$, where in each step $j$ we make sure that the $j$th row has only zeros except for the $j$th entry, and this $j$th entry is equal to $1$. 

Let
\[v_{i}=\begin{cases} v_{i,\dim V_{k}} & \dim V_{k-1}<i\leq \dim V_{k} \\
w_i & \dim V_{k}<i\leq \dtt
\end{cases}\]
be the result of this procedure.
It is not difficult to show (by induction, for instance) that for any $\dim V_{k-1}<i\leq \dim V_{k}$, we have 
\begin{enumerate}
\item $v_i\in V_{k}$
\item
$\|v_i-w_i\|\ll \epsilon$
\item
$\langle v_i,w_s\rangle=\delta_{i,s}$ for all $1\leq s\leq \dim V_{k}$.
\end{enumerate}

Finally, note that in general, a small perturbation of an orthonormal basis is still a linearly independent set. So we get that if $\epsilon$ is small enough, then
$\{v_1,\ldots,v_{\dtt}\}$ is linearly independent, and so  a basis of the flag $\{V_j\}_{j=0}^{k+1}$, which clearly is of the correct form.
\end{proof}

We can now write down the bases which we will use in the covering lemma.
Here, let $K_{\att}$ be the compact subgroup $K\cap C_{G}(\att)$, for $K=\SO(\dtt)$ and $C_{G}(\att)$ the centralizer of $\att$. Furthermore, for $P\in\mathcal{P}$, let $U_{P}^{-}=U_{P^t}$ be the unipotent radical of $P^{t}$, which is the set of lower-triangular block matrices, of the same blocks as $P$, with units on the diagonal. 

\begin{proposition}\label{proposition:bases}
Let $x\in X$. Assume there are $P\in\mathcal{P}$, $[w]_P\in W_{P,\att}$, $T\in\mathbb{R}_{>0}$ and $0<\delta<\eta_0$ such that 
\[\att_{t} x\in \cusp_{\delta}^{+}(P,[w]_P)\qquad\text{for all $t\in[0, T]$}.
\]
Then there are some $O\in K_{\att}$, and $u\in U_{P}^{-}$ satisfying 
\[\|\att_{t}^{w}u\att_{-t}^{w}-I\|_{F}\ll\epsilon_{0}\qquad\text{for all $t\in [0,T]$},\]
where $\epsilon_0$ was defined in Remark~\ref{remark:2.2}, so that the columns of the matrix $Owu$ are a basis for the flag $\{V_{l}(x)\}_{l\in\eta(P)}$ of exceptionally small $x$-rational subspaces. 
\end{proposition}

\begin{proof}
Consider the flag of exceptionally small $x$-rational subspaces $\{V_{l}(x)\}_{l\in \eta(P)}$, add the trivial subspaces $\{0\}$ and $\mathbb{R}^{\dtt}$, and denote it by
\[\{0\}=Y_0<Y_1<\cdots<Y_{k+1}=\mathbb{R}^{\dtt}.\]
For all $1\leq m\leq k$, there is an $\att_{\bullet}$-invariant subspace $Z_m$ such that $d_{\Gr}(Y_m,Z_m)<\epsilon_0$.
Consider the flag
\[\{0\}<Z_1<\cdots<Z_{k+1}=\mathbb{R}^{\dtt}.\]
Note that we can clearly find and fix an orthonormal basis $\{z_1,\ldots,z_\dtt\}$ for the flag 
$\{Z_{m}\}_{m=0}^{k+1}$,
consisting of $\att_\bullet$-eigenvectors, so that the matrix 
$z=(z_1\vline\cdots\vline z_{\dtt})$, whose $i$th column is $z_i$, is of the form $z=Ow$ for $O\in K_{\att}$.

Let $\{y_1,\ldots,y_{\dtt}\}$ be the basis for
$\{Y_{m}\}_{m=0}^{k+1}$ 
guaranteed in Lemma~\ref{lemma:5.1} with respect to $\{z_1,\ldots,z_{\dtt}\}$.
Then, writing $y=(y_1\vline \cdots\vline y_\dtt)$, it follows that $y=zu$ for $u\in U_{P}^{-}$ with $\|u-I\|_{F}\ll\epsilon_0$.
Namely, the columns of the matrix $Owu$ are a basis for the flag, as we desired.

We still need to show that $u$ satisfies $\|\att_{t}^{w}u\att_{-t}^{w}\|_{F}\ll \epsilon_0$, i.e.\ that it is exponentially small.
By assumption, $\att_{t}x\in \cusp_{\delta}^{+}(P,[w]_P)$ for all $t\in [0, T]$, and in particular
$\att_{t} x$ has the same well defined orientation $[w]_{P}$ with respect to $P$ for all such $t$. So by Proposition~\ref{proposition:2.2}, the trajectory of the flag $\{Y_{m}\}_{m=0}^{k+1}$ is close to $\{Z_{m}\}_{m=0}^{k+1}$ at all times, not only at $t=0$. That is, we have 
$d_{\Gr}(\att_{t}Y_m,Z_m)<\epsilon_0$ for all $t\in [0,T]$ 
and
for all $1\leq m\leq k$.

Note that the columns of $\att_{t}y=\att_{t} zu$ are a basis for the flag 
$\{\att_{t} Y_{l}\}_{l=0}^{k+1}$,
and hence so are the columns of $\att_{t} z u \att_{-t}^{w} =z\att_{t}^{w}u\att_{-t}^{w}$.
Fix $1\leq j\leq \dtt$, and let $\tilde{y}_j$ be the $j$th columns of $z\att_{t}^{w}u\att_{-t}^{w}$, and $l$ the unique integer which satisfies $\dim Z_{l-1}<j\leq \dim Z_{l}$.
Note that 
\begin{equation}\label{eq:4.3}\tilde{y}_{j}=z_j+\sum_{i=l+1}^{\dtt} (\att_{t}^{w}u\att_{-t}^{w})_{i,j} z_i,\end{equation}
and in particular the projection of $\tilde{y}_{j}$ to $Z_{l}$ is  $\pi_{Z_{l}}(\tilde{y}_j)=z_j$.
So by definition of $d_{\Gr}$ we get that
\[ \|\tilde{y}_j-z_j\|\leq d_{\Gr}(\att_{t} Y_i,Z_i) \|\tilde{y}_j\|<\epsilon_0 (1+\|\tilde{y}_j-z_j\|^2)^{1/2}\]
so $\|\tilde{y}_j-z_j\|\ll \epsilon_0$.
Then, as $\{z_i\}_{i=1}^{\dtt}$ is an orthonormal basis, it follows from Equation~\eqref{eq:4.3} that the coefficients 
\[(\att_{t}^{w}u\att_{-t}^{w})_{i,j}\]
are indeed bounded by $\ll \epsilon_0$ in absolute value, for all $i\geq l+1$.
Therefore, as $u\in U_{P}^{-}$, it follows that $\|\att_{t}^{w}u\att_{-t}^{u}-I\|_{F}\ll\epsilon_0$, as required.
\end{proof}

\section{The covering lemma}\label{sec:6}
In this section we state and prove the covering lemma. We consider the following sets, comprised of points whose trajectory agrees with a given coding $\mathcal{C}$.  
\begin{definition}\label{definition:10}
For any $\delta,\delta^{\prime},r,N$ as before,
and a coding $\mathcal{C}\in \mathcal{C}_N$, let
\begin{equation*}
    Z(\mathcal{C},\delta,\delta^{\prime},r)=\Big\{x\in X:\ \mathcal{C}_{x}^{\delta,\delta^{\prime},r,N}=\mathcal{C}\Big\}.
\end{equation*}
\end{definition}
The purpose of the covering lemma is to cover sets of the form $Y\cap Z(\mathcal{C},\delta,\delta^{\prime},r)$, for compact $Y\subset X$, by a suitable amount of Bowen balls (see~\S\ref{subsec:6.1}). It is a key ingredient of the proof of Theorem~\ref{theorem:1.1new}.
\subsection{Bowen balls}\label{subsec:6.1}
We will often use Bowen balls to describe small neighborhoods which behave `nicely' with respect to the flow $\att_{\bullet}$. The definition is as follows.
For $\rho>0$ and $N\in \mathbb{N}$, and $w\in W$, let
\[\mathcal{B}_{N,\rho,w}\coloneqq \bigcap_{n=-N}^{N}\att_{-n}^{w}B_{\rho}^{G}\att_{n}^{w},\]
where $B_{\rho}^{G}$ is the ball of $G$ of radius $\rho$ with respect to the right-invariant metric $d$ (see~\S\ref{subsec:metrics_on_SL}).
A {\textbf{Bowen $(N,\rho,w)$-ball}} is a set of the form
\[\mathcal{B}_{N,\rho,w}x\]
for $x\in X$.
Similarly, let \[\mathcal{B}_{N,\rho,w}^{+}\coloneqq \bigcap_{n=0}^{N}\att_{-n}^{w}B_{\rho}^{G}\att_{n}^{w}.\]
A {\textbf{forward Bowen $(N,\rho,w)$-ball}} is a set of the form
\[\mathcal{B}_{N,\rho,w}^{+}x.\]
 
We will often restrict to the case where $w$ is the identity element of $W$, and in these cases we omit the subscript $w$.

It is important for us to note that small
neighborhoods of unity in
$P$ may be covered by a suitable amount of Bowen balls, corresponding to entropy. In the following lemma, $D_{r}^{P}$ is a ball in $P$ of radius $r$ around the identity, with respect to the Frobenius metric (see~\S\ref{subsec:metrics_on_SL}).
\begin{lemma}\label{lemma:5.8}
For any $w\in W$, the ball $D_{r}^{P}$ may be covered by $\ll_{r,\rho} \exp(N\cdot h(P,\att^{w}))$  
sets of the form $\mathcal{B}_{N,\rho,w}^{+}p$ for $p\in P$,
assuming $\rho$ is small enough.
\end{lemma}
\begin{proof}
First, note that since $D_{r}^{P}$ is bounded, it may be covered by $\ll_{r} 1$ many sets $B_{r}^{P}q$ for $q\in P$, so it is sufficient to cover $B_{r}^{P}$ by the desired amount of sets $\mathcal{B}_{N,\rho,w}^{+} p$.

Let $\rho$ be small enough so that the Frobenious metric and the right-invariant metric $d$ are equivalent on a $\rho$-neighborhood of the identity (with respect to $d$).
Let $S\subset B_{r}^{P}$ be a (finite) maximal $(N,\rho/3)$-separated set with respect to $\att^{w}$, i.e.\ a set that is maximal with respect to the property that for all $s_1,s_2\in S$ there is some $0\leq n\leq N$ such that \[d(\att^{w}_{n}s_1,\att^{w}_{n}s_2)\geq \rho/3.\] 
Such a set exists because
$B_{r}^{P}$ is bounded. 

First, we claim that \[B_{r}^{P}\subseteq\bigcup_{s\in S}\mathcal{B}^{+}_{N,\rho/3,w}s.\]
Indeed, if $g\in B_{r}^{P}$ is not contained in the RHS, then in particular $g\not\in S$ so $S\cup\{g\}$ cannot be $(N,\rho)$-separated (by maximality of $S$).
So There is some $s\in S$ such that
\[d(\att^{w}_n g,\att^{w}_n s)<\rho/3\]
for all $n$.
In particular 
\[gs^{-1}\in \mathcal{B}^{+}_{N,\rho/3,w}.\]
So
\[g=(gs^{-1}) s\in \mathcal{B}^{+}_{N,\rho/3,w}s,\]
in contradiction.

Next, we claim that there is a sub-collection $S^{\prime}\subseteq S$ so that \[B_{r}^{P}\subseteq \bigcup_{s\in S^{\prime}} \mathcal{B}^{+}_{N,\rho,w}s\] and $\{\mathcal{B}_{N,\rho/3,w}^{+}s:\ s\in S^{\prime}\}$ are disjoint.
Indeed, if \[\mathcal{B}_{N,\rho/3,w}^{+}s_1\cap \mathcal{B}_{N,\rho/3,w}^{+}s_2\not=\emptyset\] for $s_1,s_2\in S$,  let $g$ be in this intersection. Then $g$ satisfies $g s_{1}^{-1},g s_{2}^{-1}\in \mathcal{B}_{N,\rho/3,w}^{+}$.
Let $h\in \mathcal{B}_{N,\rho/3,w}^{+}s_2$ be arbitrary. Then
\begin{align*}
    d(\att^{w}_n h,\att^{w}_n s_1)&\leq d(\att^{w}_n h,\att^{w}_n s_2)+d(\att^{w}_n s_2, \att^{w}_n g)+d(\att^{w}_n g, \att^{w}_n s_1)\\&
    \leq d(\att^{w}_n(h s_{2}^{-1})\att^{w}_{-n},e)
    +d(\att^{w}_n(g s_{2}^{-1})\att^{w}_{-n},e)
    +d(\att^{w}_n(g s_{1}^{-1})\att^{w}_{-n},e)\\&
    \leq \rho/3+\rho/3+\rho/3=\rho.
\end{align*}
It follows that $hs_{1}^{-1}\in \mathcal{B}_{N,\rho,w}^{+}$, so
\[\mathcal{B}_{N,\rho/3,w}^{+}s_2\subseteq \mathcal{B}_{N,\rho,w}^{+}s_1,\]
hence the ball $\mathcal{B}_{N,\rho/3,w}^{+}s_2$ may be excluded from the collection. We proceed this way until the collection is disjoint.

To finish, we only need to show that $S^{\prime}$ is of the correct cardinality. To do so, let $m$ stand for the right Haar measure on $P$. 
Note that
\[B_{r+\rho/3}^{P}\supseteq P\cap\bigcupdot_{s\in S^{\prime}}\mathcal{B}^{+}_{N,\rho/3,w}s,\]
so
\[m (B_{r+\rho/3}^{P})\geq \sum_{s\in S^{\prime}}m(P\cap\mathcal{B}_{N,\rho/3,w}^{+}s)=|S^{\prime}|m(P\cap \mathcal{B}_{N,\rho/3,w}^{+}).\]
The measure $m(P\cap\mathcal{B}_{N,\rho/3,w}^{+})$ may be computed using the Lebesgue measure on $P$ embedded in Euclidean space.
Write
\[P=(I+\bigoplus_{(i,j)\in E}U_{ij})\cap G\]
for some $E\subseteq \{1,\ldots,\dtt\}^{2}$ (where $U_{ij}$ is the set of matrices with zeros at all entries except possibly the $(i,j)$'th entry).
Using the choice of $\rho$ to be small enough so that the two metrics are equivalent on $\mathcal{B}_{N,\rho/3,w}^{+}$, we deduce 
that the width (with respect to the absolute value) of the set
$P\cap \mathcal{B}_{N,\rho/3,w}^{+}$ 
in the $(i,j)$'th entry
is
\[\asymp \rho \cdot \exp(-N\cdot(\upalpha_{w(i)}-\upalpha_{w(j)})^{+}),\]
for any $(i,j)\in E$.
Then
we obtain that
\[m(P\cap\mathcal{B}_{N,w,\rho/3}^{+})\asymp_{\rho} \prod_{(i,j)\in E}\exp(-N\cdot(\upalpha_{w(i)}-\upalpha_{w(j)})^{+})=\exp(-N\cdot h(P,\att^{w})).\]
Hence 
\[|S^{\prime}|\ll_{\rho} m(B_{r+\rho/3}^{P})\cdot \exp(N\cdot h(P,\att^{w})),\]
so
\[\{\mathcal{B}_{N,\rho,w}^{+}s:\ s\in S^{\prime}\}\]
is a covering of $B_{r}^{P}$ of the required cardinality.
\end{proof}

We now cover balls in $G$, rather than in $P$, by similar sets. To do so, we use the bases introduced in~\S\ref{sec:bases}.
\begin{proposition}\label{prop:bases2}
Let $\delta>0$ be small enough. Let $P\in\mathcal{P}$ and $[w]_P\in W_{P,\att}$. Let $[g]\in\SLdRZ$ and $T\in\mathbb{N}$.
Assume that $\att_{t} [g]\in \cusp_{\delta}^{+}(P,[w]_P)$ for all $t\in [0,T]$.
Then the set
\[Z=\left\{hg:\ h\in B_{1}^{G},\ \att_{t}[hg]\in \cusp_{\delta}^{+}(P,[w]_P)\ \forall t\in[0,T]\right\}\]
may be covered by $\ll_{\rho} \exp(T\cdot h(P,\att^{w}))$ sets of the form 
$\mathcal{B}_{T,\rho}^{+}z$ (where $z\in G$), for $\rho$ small enough.
\end{proposition}
\begin{proof}
Let $hg\in Z$. Denote $x=[g]\in X$. First, note that because $h$ is small, $h^{-1}$ is small as well, so multiplying a subspace of $\mathbb{R}^{\dtt}$ by them changes its volume by a bounded amount. In particular, for any $x$-rational subspace $U$, $hU$ is $hx$-rational with $\covol_{hx}(hU)\asymp \covol_{x}(U)$ and likewise if $V$ is $hx$-rational then $h^{-1}V$ is $x$-rational and $\covol_{x}(h^{-1} V)\asymp \covol_{hx}(V)$.
In particular, it follows from Proposition~\ref{proposition:3.6} that the unique rational subspaces of small covolume satisfy
$V_l(hx)=h V_l(x)$, for all $l\in\eta(P)$.

Let $O_1, O_2\in K_{\att}$ and $u_1,u_2\in U_{P}^{-}$ be as in Proposition~\ref{proposition:bases}, so that the columns of $O_1 w u_1$ and $O_2 w u_2$ form bases for the flags $\{V_l(x)\}_{l\in\eta(P)}$ and $\{V_{l}(hx)\}_{l\in\eta(P)}$, respectively. 
It follows that the columns of $O_{1} w u_1$ and $h^{-1}O_{2} w u_{2}$ are bases for the same flag, hence there is some $p\in P$ so that 
\[O_{1} w u_{1}=h^{-1}O_{2} w u_{2}p.\]
Note that
\[p=u_{2}^{-1} w^{-1}O_{2}^{-1} h O_{1} w u_{1},\]
and so it satisfies $\|p-I\|_{F}\leq r$ for some $r\ll 1$, i.e.\ $p\in D_{r}^{P}$.
By Lemma~\ref{lemma:5.8}, $D_{r}^{P}$ is contained in the union of some sets  $\{\mathcal{B}_{T,\rho,w}^{+}q_i\}_{i=1}^{m}$ where \[m\ll_{\rho}\exp(T\cdot h(P,\att^{w})).\] These sets do not depend on $p$.
Then, we may write $p=bq_i$ for some $b\in B_{T,\rho,w}^{+}$ and $1\leq i\leq m$. 
It follows that
\[hg=(O_{2} wu_{2} pu_{1}^{-1}w^{-1}O_{1}^{-1})g=(O_2wu_2bw^{-1})(wq_i u_{1}^{-1}w^{-1}O_{1}^{-1}g).\]

Since $K_{\att}$ is compact, we can cover it by sets $\{k_{j} B_{\rho}^{K_{\att}}\}_{j=1}^{n}$ for $n\ll_{\rho} 1$, where $k_j\in K_{\att}$ for all $j$.
Then there is some $1\leq j\leq m$ so that $O_{2}\in k_j B_{\rho}^{K_{\att}}$.
Note that conjugation by an element of $K$ preserves the Frobenius norm, which is equivalent to the invariant metric $d$ on a neighborhood of the identity (see~\S\ref{subsec:metrics_on_SL}). Hence, if $\rho$ is small enough, we have
by construction that 
\[wu_2w^{-1},wbw^{-1}\in B_{T,O(\rho)}^{+}.\]
Using the fact that $K_{\att}$ commutes with $\att$, it then follows that
\[O_{2} w u_{2} b w^{-1}\in k_{j} B_{T,O(\rho)}^{+}.\]
Since $k_{j}\in K_{\att}$, we also similarly have 
\[k_{j} B_{T,O(\rho)}^{+}\subseteq B_{T,O(\rho)}^{+}k_j.\]

A similar argument to the proof of Lemma~\ref{lemma:5.8} shows that $\mathcal{B}_{T,O(\rho)}^{+}$ can be covered by $\ll 1$ sets of the form $\mathcal{B}_{T,\rho}^{+}z$.
So all together, after splitting the balls
\[\{\mathcal{B}_{T,O(\rho)}^{+}k_j wq_iu_{1}^{-1}w^{-1}O_{1}^{-1}g:\ 1\leq i\leq m,\ 1\leq j\leq n\},\]
we obtain
\[\ll mn \ll_{\rho} \exp(T\cdot h(P,\att^{w}))\] sets of the correct form, one of which contains $hg$.

Note that the collection $\{k_jwq_iu_{1}^{-1}w^{-1}O_{1}^{-1}g\}_{i,j}$, which serves as the set of centers for the balls we constructed, depends only on $g$ and not on $hg$. Hence this construction gives a cover for the set $Z$ as required. 
\end{proof}

\subsection{Statement and proof of covering lemma}
In this subsection we state and prove the covering lemma, which gives a covering of 
the set $Y\cap Z(\mathcal{C},\delta,\delta^{\prime},r)$ by a suitable number of Bowen balls,  for any compact set $Y\subset X$.

Recall that a coding $\mathcal{C}\in \mathcal{C}_{N}$ is defined as a function from the set $\{-N,\ldots,N\}$. It would be more convenient for us to first study the forward version of the set $Z(\mathcal{C},\delta,\delta^{\prime},r)$, defined by:
\[Z_{+}(\mathcal{C},\delta,\delta^{\prime},r)=\Big\{x\in X:\ 
\mathcal{C}_{x}^{\delta,\delta^{\prime},r,N}(n)=\mathcal{C}(n)\ \forall 0\leq n\leq N\Big\}.\]

The main step of the covering lemma is as follows.
\begin{lemma}[main step]\label{lemma:main step}
Assume $Z_{+}(\mathcal{C},\delta,\delta^{\prime},r)\not=\emptyset$. Let $\mathcal{J}^{\prime}_{\mathbb{Z}}$ be the partition of $[0,N]_{\mathbb{Z}}$ corresponding to $\mathcal{C}$, with $\mathcal{C}|_{L_j}=(P_j,\sigma_j)$ for $P_j\in\mathcal{P}$ and $\sigma_j\in W_{P_j,\att}\cup\{\emptyset\}$.
Write $\mathcal{J}^{\prime}_{\mathbb{Z}}=\{L_1,\ldots,L_{|\mathcal{J}^{\prime}_{\mathbb{Z}}|}\}$, where $\max L_{l}+1=\min L_{l+1}$ for all $l<|\mathcal{J}^{\prime}_{\mathbb{Z}}|$.
Then for all $1\leq l\leq |\mathcal{J}^{\prime}_{\mathbb{Z}}|
$, and any compact set $Y\subset X$, the set $Y\cap Z_{+}(\mathcal{C},\delta,\delta^{\prime},r)$ can be covered by
\begin{equation*}
    \ll N_{\rho}(Y)\cdot C^{l}\exp\Big(\sum_{\substack{1\leq j\leq l:\\\sigma_j=[w_j]_{P_j}\in W_{P_j,\att}}}|L_j|\cdot h(P_j,\att_{t}^{w_j})+\sum_{\substack{1\leq j\leq l:\\ \sigma_j=\emptyset}}|L_j|\cdot h(G,\att)\Big)
\end{equation*}
many forward Bowen $(\max L_l,\rho)$-balls, for some  
$0<C\ll_{\att,\rho} 1$, where $N_{\rho}(Y)$ is the covering number of $Y$ by balls of radius $\rho$.
\end{lemma}
\begin{proof}
We prove the claim by induction on $l$.
We start with a covering $\mathcal{O}$ of $Y$ by $N_{\rho}(Y)$ many balls of the form $ B_{\rho}^{G} x$. 

Let $1\leq l \leq |\mathcal{J}^{\prime}_{\mathbb{Z}}|$ and write
$L_{l}=[T,T+M]_{\mathbb{Z}}$ for $T,M\in\mathbb{Z}$.
If $l>1$ then we assume by induction that the claim holds for $l-1$, and start with a covering of $Y\cap Z_{+}(\mathcal{C},\delta,\delta^{\prime},r)$ by the correct amount of Bowen $(T-1,\rho)$-balls (as $\max L_{l-1}=T-1$). We split each of them to $\ll_{\att} 1$ many Bowen $(T,\rho)$-balls. If $l=1$ (i.e.\ $T=0$) we also have an initial covering by the correct number of Bowen $(T,\rho)$-balls, namely $\mathcal{O}$.  

Let $\mathcal{B}_{T,\rho}^{+} x$ be one of these balls. 
Without loss of generality, we may assume $x\in Z_{+}(\mathcal{C},\delta,\delta^{\prime},r)$,  
otherwise we split the ball to $\ll 1$ Bowen balls around points in $\mathcal{B}^{+}_{T,\rho} x \cap Z_{+}(\mathcal{C},\delta,\delta^{\prime},r)$. 
Write $\mathcal{C}|_{L_l}=(P_l,\sigma_l)$. If  $\sigma_l=[w_l]_P\in W_{P_l,\att}$ then we give the simpler notations $P=P_l$ and $[w]_P=[w_l]_{P_l}$. Otherwise, we let $P=G$ and $[w]_{G}$ the unique element of $W_{G,\att}$.
Then in any case, by the construction in~\S\ref{subsec:coding},
\[\att_{t}x\in \cusp_{\delta^{\prime}}^{+}(P,[w]_P)\] for all $t\in [T,T+M]$.

Write $x=[g]$. 
Consider \[\att_T (\mathcal{B}_{T,\rho}^{+}g)\subseteq B_{\rho}^{G}(\att_{T}g).\]
By Proposition~\ref{prop:bases2}, there is a collection of $k\ll_{\rho}\exp(M\cdot h(P,\att^{w}))$ points $g_i$ so that
\[\{f\in \mathcal{B}_{T,\rho}^{+}g:\ [f]\in Z_{+}(\mathcal{C},\delta,\delta^{\prime},r)\} \subseteq \att_{-T}\bigcup_{i=1}^{k}  \mathcal{B}_{M,\rho}^{+} g_{i}.\]
It follows that
\begin{equation}\label{eq:5.3}
\mathcal{B}_{T,\rho}^{+}x\cap Z_{+}(\mathcal{C},\delta,\delta^{\prime},r)\subseteq \pi_{X}\Big(\bigcup_{i=1}^{k}\att_{-T}\mathcal{B}_{M,\rho}^{+}g_i\cap \mathcal{B}_{T,\rho}^{+}g\Big)
\end{equation}
for $\pi_{X}$ the projection from $G$ to $X=G/\Gamma$.

Consider
\begin{equation*}\label{eq:5.12}
    \att_{-T}\mathcal{B}_{M,\rho}^{+} g_{i}\cap \mathcal{B}_{T,\rho}^{+}g.
\end{equation*}
We claim that if this intersection is non-empty, then it is contained in a union of $\ll 1$ sets of the form $\mathcal{B}_{T+M,\rho}^{+}q$. Indeed, let \[h,q\in \att_{-T}\mathcal{B}_{M,\rho}^{+}g_i\cap \mathcal{B}_{T,\rho}^{+}g.\]
First, for any $0\leq n\leq T$ we have
\begin{align*}
\numberthis\label{eq:5.4}
d(\att_{n}hq^{-1}\att_{-n},e)&=d(\att_n q h^{-1} \att_{-n},e)\leq d(\att_{n}qh^{-1}\att_{-n},\att_{n} g h^{-1}\att_{-n})+d(\att_{n} g h^{-1}\att_{-n},e)\\&
=d(\att_{n}qg^{-1}\att_{-n},e)+d(\att_{n}hg^{-1}\att_{-n},e)<\rho+\rho=2\rho
\end{align*}
where we used the fact that $h,q\in \mathcal{B}_{T,\rho}^{+}g$.
Next, for $0\leq n\leq M$, we have
\begin{align*}
\numberthis\label{eq:5.5}
d(\att_{T+n}hq^{-1}\att_{-(T+n)},e)&=d(\att_n (\att_{T} q)(\att_{T}h)^{-1}\att_{-n},e)\\&\hspace{-1.5cm}
\leq d(\att_{n}(\att_{T} q)(\att_{T}h)^{-1}\att_{-n},\att_{n}g_{i}(\att_{T}h)^{-1}\att_{-n})+d(\att_{n}g_{i}(\att_{T}h)^{-1}\att_{-n},e)
\\&\hspace{-1.5cm}
=d(\att_{n}(\att_{T}qg_{i}^{-1})\att_{-n},e)+d(\att_{n}(\att_{T}hg_{i}^{-1})\att_{-n},e)\\&\hspace{-1.5cm}
<\rho+\rho=2\rho
\end{align*}
where we used the fact that $h,q\in \att_{-T}\mathcal{B}_{M,\rho}^{+}g_i$.
All together, we deduce that from Equations~\eqref{eq:5.4}-\eqref{eq:5.5} that
\[h\in \mathcal{B}_{T+M,2\rho}^{+}q,\]
so in fact 
\[\att_{-T}\mathcal{B}_{M,\rho}^{+} g_{i}\cap \mathcal{B}_{T,\rho}^{+}g\subseteq \mathcal{B}^{+}_{T+M,2\rho}q.\]
Then it follows 
from Equation~\eqref{eq:5.3} 
that the intersection of $Z_{+}(\mathcal{C},\delta,\delta^{\prime},r)$ with each of the Bowen-$(T,\rho)$ balls we started the step with, is covered by $\ll_{\rho} \exp(M\cdot h(P,\att^{w}))$ many Bowen-$(T+M,\rho)$ balls.
This gives the correct amount of Bowen balls for the induction step.
\end{proof}

The covering lemma is a straightforward generalization of Lemma~\ref{lemma:main step}, for non-forward Bowen balls.
\begin{lemma}[covering lemma]\label{covering lemma}
Let $\mathcal{C}\in \mathcal{C}_{N}$ and assume $Z(\mathcal{C},\delta,\delta^{\prime},r)\not=\emptyset$. Then  for any compact $Y\subset X$, the set $\att_{N}Y\cap Z(\mathcal{C},\delta,\delta^{\prime},r)$ may be covered by 
\[\ll N_{\rho}(Y)\cdot \exp(\frac{D}{r^{\dtt-1}|\log\delta|}N)\cdot \exp\left(\sum_{\substack{P\in\mathcal{P},\\ [w]_P\in W_{P,\att}}} |\mathcal{C}^{-1}(P,[w]_P)|\cdot h(P,\att^{w})\right)\]
many Bowen $(N,\rho)$-balls, for some constant $0<D\ll_{\att,\rho} 1$.
\end{lemma}
\begin{proof}
Define $\mathcal{C}^{\prime}\in \mathcal{\mathcal{C}}_{2N}$ by 
\[\mathcal{C}^{\prime}(n)=\begin{cases}\mathcal{C}(-N) & n\leq 0 \\ \mathcal{C}(n-N) & n>0\end{cases}.\]
Then clearly 
\[\att_{-N}Z(\mathcal{C},\delta,\delta^{\prime},r)=Z_{+}(\mathcal{C}^{\prime},\delta,\delta^{\prime},r).\]
Consider the partition of $[-2N,2N]$ to maximal intervals on which $\mathcal{C}^{\prime}$ is constant. It is just the partition of $[-N,N]$, shifted by $N$ and extended to $-2N$ for the first interval. The number of such intervals, by construction in~\S\ref{subsec:coding}, is $\ll_{\att} N/(r^{\dtt-1}|\log\delta|)$. 
Then, by Lemma~\ref{lemma:main step}, $Y\cap\att_{-N}Z(\mathcal{C},\delta,\delta^{\prime},r)$ may be covered by 
\begin{multline*}\ll  N_{\rho}(Y)\cdot \exp(\frac{D} {r^{\dtt-1}|\log\delta|}N)\\
\cdot\exp\Bigl(\!\!\!\sum_{\substack{P\in\mathcal{P},\\ [w]_P\in W_{P,\att}}} |\mathcal{C}^{-1}(P,[w]_{P})|\cdot h(P,\att^{w})+|\mathcal{C}^{-1}(\mathcal{P}\times\{\emptyset\})|\cdot h(G,\att)\Bigr)\end{multline*}
many forward Bowen $(2N,\rho)$-balls, for some $D\asymp_{\rho,\att} 1$.
The image of this covering by $\att_{N}$ gives a covering of $\att_{N}Y\cap Z(\mathcal{C},\delta,\delta^{\prime},r)$ by the same number of (non-forward) Bowen $(N,\rho)$-balls.
Furthermore, by Proposition~\ref{proposition:5.2},
\[|\mathcal{C}^{-1}(\mathcal{P}\times \{\emptyset\})|\ll_{\att} \frac{1}{r^{\dtt-1}|\log\delta|}N,\]
so the proposition follows by adjusting $D\asymp_{\att,\rho} 1$.
\end{proof}

\section{Proving an entropy bound}\label{sec:proof_1.1}
In this section we prove the following theorem, which implies Theorem~\ref{theorem:1.1new} as will be explained in~\S\ref{subsec:8.3}. This theorem gives an entropy bound depending on our specific choice of cusp regions $\cusp_{\delta,\delta^{\prime}}(P,Q,[w]_Q)$.

\begin{theorem}\label{theorem:1.1}
Let $0<\delta<\delta^{\prime}$ small enough, with $\delta^{\prime}>\delta^{1/2}$. Fix $r$ in the range
\[(\frac{\log\delta^{\prime}}{\log\delta})^{\tfrac1{\dtt+1}}<r<1.\]
Let $\mu$ be an $\att$-invariant probability measure for $\att\in A$, and $\phi\in \Lie(A)^{\ast}$.
Then 
\begin{multline*}
h_{\mu}(\att)\leq
\sum_{\substack{P,Q\in\mathcal{P}:\\Q\subseteq P}}\sum_{[w]_{Q}\in W_{Q,\att}}\mu(\cusp_{\delta,\delta^{\prime}}(P,Q,[w]_{Q}))\cdot\max_{\substack{H\in\mathcal{P}:\\Q\subseteq H\subseteq P}}(h-\phi)([w]_{H})\\
+C_{\att}f(\delta,r,\phi)+C_{\att,\phi}^{\prime}\frac{1}{|\log\delta^{\prime}|}
\end{multline*}
where
$C_{\att}$ is a constant which depends only on $\att$, $C^{\prime}_{\att,\phi}$ a constant  which depends only on $\att$ and $\phi$, 
and
\[f(\delta,r,\phi)=\frac{\log\big(r^{\dtt-1}|\log\delta|\big)}{r^{\dtt-1}|\log\delta|}+\|\phi\|r.\]
\end{theorem}
\begin{remark}\label{remark:1.2}
    The error term $\frac{1}{|\log\delta^{\prime}|}$ can clearly  be made arbitrarily small, as the limit $\delta^{\prime}\to 0^{+}$ is taken (and with it the limit $\delta\to 0^{+}$ is taken as $\delta<(\delta^{\prime})^{2})$.
    As for $f(\delta,r,\phi)$, note that 
    $r$ needs to be small for the error term to be insignificant. 
    As $r>(\frac{\log\delta^{\prime}}{\log\delta})^{1/(\dtt+1)}$,
    this requires
    $\delta$ to be considerably smaller than $\delta^{\prime}$.
    We would later on use this theorem with $\delta^{\prime}$ being fixed as 
    \[\delta^{\prime}=\exp(-|\log\delta|^{\beta})\]
    for some $0<\beta<1$, which guarantees that all error terms become arbitrarily small as $\delta$ approaches zero.
\end{remark}

The covering lemma (Lemma~\ref{covering lemma}) is used to prove Theorem~\ref{theorem:1.1} by a standard procedure which was utilized in some previous works~\cite{einsiedler2012entropy,einsiedler2011distribution,einsiedler2015escape,kadyrov2017singular,mor2021excursions}. 
In our case, we also need to include linear functionals in the estimates, and to do so we use Corollary~\ref{corollary:4.7_n}.
A final ingredient of the proof is the following lemma, which gives an upper bound for the entropy.
When the space is compact, this lemma follows from a result of Katok \cite{katok1980lyapunov}.
Proof of this lemma in settings similar to ours may be found in~\cite{einsiedler2011distribution,mor2021excursions}. The proof translates to the setup in this paper with little change, so we do not repeat the proof here as well. 
In this section, $T$ stands for the time-one map of the flow $\att_{\bullet}$, namely $Tx=\att_{1} x$.
\begin{lemma}\label{lemma:7.1}
Let $\mu$ be an $\att_{\bullet}$-invariant probability measure on $X$. For any given integer $1\leq N\in \mathbb{N}$ and $\rho,\epsilon>0$, let $\BC_{\rho}(N,\epsilon)$ be the minimal number of Bowen $(N,\rho)$-balls needed to cover any subset of $X$ of measure larger than $1-\epsilon$. Then for any $\rho>0$,
\begin{equation*}
    h_{\mu}(T)\leq\lim_{\epsilon\to 0^{+}}\liminf_{N\to\infty}\frac{\log\BC_{\rho}(N,\epsilon)}{2N}.
\end{equation*}
\end{lemma}
\begin{proof}
Omitted.
\end{proof}

Let us now prove Theorem~\ref{theorem:1.1}.
\begin{proof}[Proof of Theorem~\ref{theorem:1.1}]
First of all, let us assume that $\mu$ is ergodic.
Let $\epsilon>0$ be arbitrary.
Let $0<\tau<\eta_0$ be small enough (depending on $\mu$ and $\epsilon$) so that the compact set $Y=\compact_{\tau}(X)$ satisfies $\mu(Y)>1-\frac{\epsilon}{3}$.

From Birkhoff's pointwise ergodic theorem, for all $(P,Q,[w]_{Q})$, where $P,Q\in\mathcal{P}$ satisfy $Q\subseteq P$ and $[w]_{Q}\in W_{Q,\att}$, the sequence of functions
\begin{equation*}
    f_{N,P,Q,[w]_Q}\coloneqq\frac{1}{2N+1}\sum_{n=-N}^{N}\mathbbm{1}_{T^{-n}(\cusp_{\delta,\delta^{\prime}}(P,Q,[w]_Q))}
\end{equation*}
converges almost everywhere (as $N\to\infty$) to the constant function $\mu(\cusp_{\delta,\delta^{\prime}}(P,Q,[w]_{Q}))$. 
By Egorov's theorem, there is a subset $E_{P,Q,[w]_{Q}}\subseteq X$ with 
\[\mu(E_{P,Q,[w]_{Q}})>1-\frac{\epsilon}{3|\mathcal{P}|^{2}|W|}\]
such that $(f_{N,P,Q,[w]_Q})_{N\in\mathbb{N}}$ converges uniformly on $E_{P,Q,[w]_{Q}}$. Therefore, we may choose $N_{P,Q,[w]_{Q}}$ such that for all $N>N_{P,Q,[w]_{Q}}$ and $x\in E_{P,Q,[w]_{Q}}$:
\begin{equation*}
    \frac{1}{2N+1}\sum_{n=-N}^{N}\mathbbm{1}_{T^{-n}(\cusp_{\delta,\delta^{\prime}}(P,Q,[w]_{Q}))}(x)>\kappa_{P,Q,[w]_{Q}}\coloneqq\mu(\cusp_{\delta,\delta^{\prime}}(P,Q,[w]_{Q}))-\epsilon.
\end{equation*}

Set \[N_0=\max_{\substack{P,Q\in\mathcal{P}:\ Q\subseteq P,\\ [w]_{Q}\in W_{Q,\att}}}N_{P,Q,[w]_{Q}}\] and \[E=\bigcap_{\substack{P,Q\in\mathcal{P}:\ Q\subseteq P,\\ [w]_{Q}\in W_{Q,\att}}}E_{P,Q,[w]_{Q}}.\]
Let $N>N_0$.
Define \begin{equation*}
F=E\cap T^{N}Y\cap T^{-N}Y.
\end{equation*}
Note that 
\begin{equation*}
\mu(F)> 1-|\mathcal{P}|^{2}|W|\frac{\epsilon}{3|\mathcal{P}|^{2}|W|}-\frac{\epsilon}{3}-\frac{\epsilon}{3}=1-\epsilon
\end{equation*}
so in view of Lemma~\ref{lemma:7.1}, it would suffice to cover $F$ by a suitable amount of Bowen balls.
Assume now that $N>|\log\delta|$ as well. 
By the covering lemma (Lemma~\ref{covering lemma}), for any $\mathcal{C}\in\mathcal{C}_{N}$, the set
\[T^{N}(Y)\cap Z(\mathcal{C},\delta,\delta^{\prime},r)\]
can be covered by
\begin{equation}\label{eq:7.2}\ll N_{\rho}(Y)\cdot\exp(\frac{D}{ r^{\dtt-1}|\log\delta|}N)\cdot\exp\Big(\sum_{\substack{P\in \mathcal{P},\\ [w]_P\in W_{P,\att}}}|\mathcal{C}^{-1}(P,[w]_P)|\cdot h(P,\att^{w})\Big)\end{equation}
Bowen $(N,\rho)$-balls, for some small fixed $\rho>0$.
Let
\[\mathcal{A}=\{\mathcal{C}\in \mathcal{C}_{N}:\ F\cap Z(\mathcal{C},\delta,\delta^{\prime},r)\not=\emptyset\}.\]
Clearly \[F\subseteq \bigcup_{\mathcal{C}\in \mathcal{A}}T^{N}(Y)\cap Z(\mathcal{C},\delta,\delta^{\prime},r),\]
and it follows, using Equation~\eqref{eq:7.2}, that
$F$ can be covered by
\begin{align*}&\ll_{\tau}
\sum_{\mathcal{C}\in\mathcal{A}}\exp(\frac{D}{ r^{\dtt-1}|\log\delta|}N)\cdot\exp\Big(\sum_{\substack{P\in \mathcal{P},\\ [w]_P\in W_{P,\att}}}|\mathcal{C}^{-1}(P,[w]_P)|\cdot h(P,\att^{w})\Big)\\&
\leq 
|\mathcal{A}|
\exp(\frac{D}{ r^{\dtt-1}|\log\delta|}N)\cdot\max_{\mathcal{C}\in \mathcal{A}}\Big[\exp\Big(\sum_{\substack{P\in \mathcal{P},\\ [w]_P\in W_{P,\att}}}|\mathcal{C}^{-1}(P,[w]_P)|\cdot h(P,\att^{w})\Big)\Big]
\end{align*}
many Bowen $(N,\rho)$-balls (note that we omitted the dependence on $\rho$ from the $\ll$ notation, since $\rho$ is considered fixed).
By Proposition~\ref{lemma:num_of_V},  
\begin{equation*}
|\mathcal{A}|\leq \exp(\kappa\frac{\log\big(r^{\dtt-1}|\log\delta|\big)}{r^{\dtt-1}|\log\delta|}N),\end{equation*}
so the covering we found for $F$ by Bowen $(N,\rho)$-balls is in fact of cardinality
\[\ll_{\tau} 
\exp(\kappa^{\prime}\frac{\log\big(r^{\dtt-1}|\log\delta|\big)}{r^{\dtt-1}|\log\delta|}N)\cdot
\max_{\mathcal{C}\in \mathcal{A}}\Big[
\exp\Big(\sum_{\substack{P\in \mathcal{P},\\ [w]_P\in W_{P,\att}}}|\mathcal{C}^{-1}(P,[w]_P)|\cdot h(P,\att^{w})\Big)\Big]
\]
for some $\kappa^{\prime}>0$ which depends on $\att$.

\medskip

We now add the linear functional.
Note that for $x\in F$ we have 
\[\|\height(\att_{N}x)-\height(\att_{-N}x)\|\ll_{\tau} 1,\]
with $\height$ defined as in Equation~\eqref{eq:3.4}. 
Therefore, by 
Corollary~\ref{corollary:4.7_n}, for any $\mathcal{C}\in\mathcal{A}$ we have
\[\exp\Big(-\sum_{\substack{P\in \mathcal{P},\\ [w]_P\in W_{P,\att}}}|\mathcal{C}^{-1}(P,[w]_P)|\phi(\pi_P(\upalpha^{w}))+ \tilde{D}\|\phi\|\cdot rN\Big)\geq 1,\]
for some $\tilde{D}\asymp_{\att} 1$ and large enough $N$.
So the covering of $F$ 
is certainly of cardinality at most 
\begin{align*}
    &\ll_{\tau} 
\exp(E\cdot f(\delta,r,\phi)N)\cdot
\max_{\mathcal{C}\in\mathcal{A}}\Big[
    \exp\Big(\sum_{\substack{P\in \mathcal{P},\\ [w]_P\in W_{P,\att}}}|\mathcal{C}^{-1}(P,[w]_P)|\big(h(P,\att^{w})-\phi(\pi_P(\upalpha^{w}))\big)\Big)\Big],
\end{align*}
 for some $0<E\ll_{\att} 1$, where the notation 
\[f(\delta,r,\phi)=\frac{\log\big(r^{\dtt-1}|\log\delta|\big)}{r^{\dtt-1}|\log\delta|}+\|\phi\|r\]
was introduced.

For any $x\in X$, let 
\[T_{P,Q,[w]_Q}^{x}=\{n\in [-N,N]\cap\mathbb{Z}:\ T^{n}x\in\cusp_{\delta,\delta^{\prime}}(P,Q,[w]_Q)\}.\]
Let $\mathcal{C}\in\mathcal{A}$ and $x\in F\cap Z(\mathcal{C},\delta,\delta^{\prime},r)$.
Recall that by Proposition~\ref{corollary:5.3}, we have
\begin{align*}
&\sum_{H\in \mathcal{P}}\sum_{[\tau]_{H}\in W_{H,\att}}
\Big|\mathcal{C}^{-1}(H,[\tau]_{H})\Big|(h-\phi)([\tau]_{H})
\leq \\&
\sum_{\substack{P,Q\in\mathcal{P}:\\ Q\subseteq P}}\sum_{[w]_{Q}\in W_{Q,\att}}\Big|T^{x}_{P,Q,[w]_{Q}}\Big|\max_{\substack{H\in\mathcal{P}:\\ Q\subseteq H\subseteq P}}(h-\phi)([w]_{H})
+\Big|T^{x}_{\cusp_{\delta^{\prime}}(X,\emptyset)}\Big|\max_{\substack{H\in \mathcal{P},\\ [w]_{H}\in W_{H,\att}}}(h-\phi)([w]_{H}).
\end{align*}
Furthermore, note that
\[\cusp_{\delta^{\prime}}(X,\emptyset)\cupdot\bigcupdot_{\substack{Q\subseteq P\\ [w]_{Q}\in W_{Q,\att}}}\cusp_{\delta,\delta^{\prime}}(P,Q,[w]_{Q})=X\]
is a partition, so in particular
\[\Big|T^{x}_{\cusp_{\delta^{\prime}}(X,\emptyset)}\Big|=2N+1-\sum_{\substack{P,Q\in \mathcal{P}:\\Q\subseteq P}}\sum_{[w]\in W_{Q}}\Big|T^{x}_{P,Q,[w]_{Q}}\Big|.\]
Moreover, using Proposition~\ref{proposition:3.9}, we have
\[1-\sum_{\substack{P,Q\in\mathcal{P}:\\ Q\subseteq P}}\sum_{[w]_{Q}\in W_{Q,\att}}\mu(\cusp_{\delta,\delta^{\prime}}(P,Q,[w]_{Q}))=\mu(\cusp_{\delta^{\prime}}(X,\emptyset))\ll_{\att}\frac{1}{|\log\delta^{\prime}|}.\]
Finally, recall that the set $F$ satisfies, by its definition, 
\begin{align*}
F\subseteq\Big\{&x\in T^{N}(Y)\cap T^{-N}(Y):\\&\:  \left|T_{P,Q,[w]_Q}^{x}\right|\geq(2N+1)\kappa_{P,Q,[w]_Q},\ \forall Q,P\in\mathcal{P}\text{ with $Q\subseteq P$ and } [w]_{Q}\in W_{Q,\att}\Big\}.
\end{align*}

Using all of these estimates, we find that
our cover of $F$ is in fact of size at most
\begin{align*}
&\ll_{\tau}\exp\Big[E\cdot f(\delta,r,\phi)N+\sum_{\substack{P,Q\in\mathcal{P}:\\ Q\subseteq P}}\sum_{[w]_{Q}\in W_{Q,\att}}(2N+1)\kappa_{P,Q,[w]_{Q}}\cdot\max_{\substack{H\in\mathcal{P}:\\ Q\subseteq H\subseteq P}}(h-\phi)([w]_{H})\\&\hspace{1cm}
+\Big(2N+1-\sum_{\substack{P,Q\in\mathcal{P}:\\ Q\subseteq P}}\sum_{[w]_{Q}\in W_{Q,\att}}(2N+1)\kappa_{P,Q,[w]_{Q}}\Big)\max_{\substack{H\in \mathcal{P},\\ [w]_{H}\in W_{H,\att}}}(h-\phi)([w]_{H})\Big]
\\&
=\exp\Big[(2N+1)\Big(\sum_{\substack{P,Q\in\mathcal{P}:\\ Q\subseteq P}}\sum_{[w]_{Q}\in W_{Q,\att}}\mu(\cusp_{\delta,\delta^{\prime}}(P,Q,[w]_{Q}))\max_{\substack{H\in\mathcal{P}:\\ Q\subseteq H\subseteq P}}(h-\phi)([w]_{H})\\&\hspace{1cm}
+\mu(\cusp_{\delta^{\prime}}(X,\emptyset))\max_{\substack{H\in \mathcal{P},\\ [w]_{H}\in W_{H,\att}}}(h-\phi)([w]_{H})+O_{\att,\phi}(\epsilon)\Big)
+E\cdot f(\delta,r,\phi)N\Big]\\&
\leq \exp\Big[(2N+1)\Big(\sum_{\substack{P,Q\in\mathcal{P}:\\ Q\subseteq P}}\sum_{[w]_{Q}\in W_{Q,\att}}\mu(\cusp_{\delta,\delta^{\prime}}(P,Q,[w]_{Q}))\max_{\substack{H\in\mathcal{P}:\\ Q\subseteq H\subseteq P}}(h-\phi)([w]_{H})\\&\hspace{1cm}
+O_{\att,\phi}(\frac{1}{|\log\delta^{\prime}|}+\epsilon)\Big)+E\cdot f(\delta,r,\phi)N\Big].
\end{align*}

\medskip

Recall that $\mu(F)>1-\epsilon$,
so the amount of Bowen balls we obtained is an upper bound for $\BC_{\rho}(N,\epsilon)$.
Then, using Lemma~\ref{lemma:7.1}, we get
\begin{align*}
h_{\mu}(T)&\leq \lim_{\epsilon\to 0^{+}}\underset{N\to\infty}{\liminf}\frac{\log\BC_{\rho}(N,\epsilon)}{2N+1}
\\&\leq
\sum_{\substack{P,Q\in\mathcal{P}:\\ Q\subseteq P}}\sum_{[w]_{Q}\in W_{Q,\att}}
\mu(\cusp_{\delta,\delta^{\prime}}(P,Q,[w]_{Q}))\max_{\substack{H\in\mathcal{P}:\\ Q\subseteq H\subseteq P}}(h-\phi)([w]_{H})\\&
+O_{\att}(f(\delta,r,\phi))+O_{\att,\phi}(\frac{1}{|\log\delta^{\prime}|}),
\end{align*}
as desired.

In order to deduce the theorem for non-ergodic $\mu$, we can decompose $\mu$ to its ergodic
components $\mu=\int\mu_{t}d\tau(t)$ over some measure space. As
\begin{equation*}
h_{\mu}(T)=\int h_{\mu_{t}}(T)d\tau(t),
\end{equation*}
the result follows immediately from the ergodic case.
\end{proof}

\section{Deducing the results}
In this section we deduce from Theorem~\ref{theorem:1.1} the 
results of this paper. We start with Theorem~\ref{corollary:1.2} and then discuss all other results, which require the notion of the RBS compactification. 
\subsection{Proof of Theorem~\ref{corollary:1.2}}
\begin{proof}[Proof of Theorem~\ref{corollary:1.2}]
Consider the bound given by Theorem~\ref{theorem:1.1}, with $\phi=0$.
The entropy for a parabolic subgroup increases as the subgroup increases, so
\[\max_{\substack{H\in\mathcal{P}:\\ Q\subseteq H\subseteq P}}(h-\phi)([w]_{H})=\max_{\substack{H\in\mathcal{P}:\\ Q\subseteq H\subseteq P}}h(H,\att^{w})=h(P,\att^{w})\]
for any $P,Q\in\mathcal{P}$ with $Q\subseteq P$ and $[w]_{Q}\in W_{Q,\att}$.

Let $P\in\mathcal{P}$ and $[w]_{P}\in W_{P,\att}$.
It is clear that
\[\bigcupdot_{\substack{Q\in \mathcal{P}:\\Q\subseteq P}}\bigcupdot_{\substack{[\tau]_{Q}\in W_{Q,\att}:\\ [\tau]_P=[w]_P}}\cusp_{\delta,\delta^{\prime}}(P,Q,[\tau]_{Q})\subseteq\cusp_{\delta}(P,[w]_{P}).\]
Then Theorem~\ref{corollary:1.2} follows from this inclusion together with
the choices
\[\delta^{\prime}=\exp(-|\log\delta|^{1/2})\]
and
\[r=(\frac{|\log\delta^{\prime}|}{|\log\delta|})^{1/(\dtt+2)}.\]
\end{proof}

\subsection{The reductive Borel-Serre compactification}\label{subsec:8.2}
The rest of the results involve the reductive Borel-Serre (RBS) compactification of $G/\Gamma$.
We first introduce it in the generality where $G=\mathbf{G}(\mathbb{R})$ is the real locus of a semisimple linear algebraic group $\mathbf{G}$ defined over $\mathbb{Q}$, and $\Gamma$ is an arithmetic subgroup. 
Then we relate this definition to our work on the particular case $\mathbf{G}=\SL_{\dtt}$ and $\Gamma=\mathbf{G}(\mathbb{Z})$.
We follow the construction of~\cite{borel2006compactifications}.

Fix some maximal torus in $\mathbf{G}$.
For $P=\mathbf{P}(\mathbb{R})$ the real locus of a parabolic subgroup $\mathbf{P}<\mathbf{G}$,  consider the rational Langlands decomposition
\[P=M_{\mathbf{P}}A_{\mathbf{P}}N_{P}\]
with respect to the fixed maximal torus,
where $N_{P},A_{\mathbf{P}},M_{\mathbf{P}}$ are the real loci of the unipotent radical of $\mathbf{P}$, of the torus component of the Levi subgroup of $\mathbf{P}$, and of the semisimple component of the Levi subgroup of $\mathbf{P}$, respectively (see~\cite[\S.III.1.9]{borel2006compactifications} for precise definitions). 
For our special case where $G=\SL_{\dtt}(\mathbb{R}$) and $P$ a standard parabolic subgroup, we can write $P=M_P A_P N_P$ for
$N_{P}$ the subgroup of block upper triangular matrices with unit matrices on the diagonal, 
$A_P$ the subgroup of block scalar matrices, 
and $M_P$ the subgroup of block diagonal matrices with determinant $\pm 1$ on each block.

We also consider the relative rational Langlands decomposition of a rational parabolic subgroup $\mathbf{Q}$ with respect to a rational parabolic subgroup $\mathbf{P}\subset\mathbf{Q}$. 
In these settings, there is a uniquely determined $\mathbb{Q}$-parabolic subgroup $\mathbf{P}^{\prime}<\mathbf{M}_{\mathbf{Q}}$, for $\mathbf{M}_{\mathbf{Q}}$ the semisimple component
of the Levi subgroup of
$\mathbf{Q}$ \cite[Lemma 2]{harish1968automorphic}, which satisfies $M_{\mathbf{P}^\prime}=M_{\mathbf{P}}$ and gives rise to the decomposition \[KM_{\mathbf{Q}}=KM_{\mathbf{P}^{\prime}}A_{\mathbf{P}^{\prime}}N_{P^\prime}\]
where $K$ is a maximal compact subgroup of $G$. For the case $\mathbf{G}=\SL_{\dtt}$, we will choose $K=\SO(\dtt)$.
Lastly, for a $\mathbb{Q}$-parabolic subgroup $\mathbf{P}$, we require the notation $\Phi(P,A_{\mathbf{P}})$, which is the set of characters of $A_{\mathbf{P}}$ occurring in its action on $\mathfrak{n}_{\mathbf{P}}=\Lie N_P$ by the adjoint representation.

Now, we define the reductive Borel-Serre (RBS) compactification of $G$.
For a $\mathbb{Q}$-parabolic subgroup $\mathbf{P}$, define the boundary face \[e(\mathbf{P})=KM_{\mathbf{P}}\cong G/A_{\mathbf{P}}N_P \] and the \textbf{reductive Borel-Serre compactification}
\[\overline{G}^{\RBS}=\coprod_{\textup{$\mathbf{P}<\mathbf{G}$ $\mathbb{Q}$-parabolic}}e(\mathbf{P}).\]
The topology on $\overline{G}^{\RBS}$ can be given by a convergence class of sequences \cite[Definition I.8.11]{borel2006compactifications}, i.e.\ by determining the limit of sequences in $\overline{G}^{\RBS}$ in a suitable manner. We give two classes of converging sequences in $\overline{G}^{\RBS}$ which together form a convergence class:

\begin{enumerate}
    \item For any $\mathbb{Q}$-parabolic subgroup $\mathbf{P}$, an unbounded sequence \[y_j=(m_j,a_j,n_j)\in (KM_{\mathbf{P}})\times A_{\mathbf{P}}\times N_P\cong G\] converges to a points $m_{\infty}\in e(\mathbf{P})$ if and only if the following conditions are satisfied:
    \begin{enumerate}
        \item $\alpha(a_j)\to\infty$ for all $\alpha\in\Phi(P,A_{\mathbf{P}})$.
        \item $m_j\to m_\infty$ in $M_{\mathbf{P}}K$.
    \end{enumerate}
    \item For two $\mathbb{Q}$-parabolic subgroups $\mathbf{P}\subsetneq \mathbf{Q}$, a sequence \[y_j=(m_j,a_j,n_j)\in e(\mathbf{P})\times A_{\mathbf{P}^{\prime}}\times N_{P^\prime}=e(\mathbf{Q})\] converges to $m_{\infty}\in e(\mathbf{P})$ if and only if
    \begin{enumerate}
        \item $\alpha(a_j)\to\infty$ for all $\alpha\in\Phi(P^\prime,A_{\mathbf{P}^{\prime}})$.
        \item $m_j\to m_{\infty}$ in $e(\mathbf{P})$.
    \end{enumerate}
\end{enumerate}

We are interested in the compactification of $G/\Gamma$ rather than $G$. To obtain such a compactification, we note that $\mathbf{G}(\mathbb{Q})$ acts continuously from the right on $\overline{G}^{\RBS}$ as follows \cite[Proposition III.14.4]{borel2006compactifications}. For $g\in\mathbf{G}(\mathbb{Q})$, write $g=namk$ for $k\in K$, $m\in M_{\mathbf{P}}$, $a\in A_{\mathbf{P}}$, $n\in N_P$. Of course $k$ and $m$ are not uniquely defined, but $mk$ is. Then $g$ acts on $m^{\prime}\in KM_{\mathbf{P}}$ by \[m^{\prime}g=m^{\prime}mk\in e(k^{-1}\mathbf{P}k).\]
Therefore, the quotient space \[\overline{ G/\Gamma}^{\RBS}\coloneqq\overline{G}^{\RBS}/\Gamma\] is well defined for any arithmetic subgroup $\Gamma<\mathbf{G}(\mathbb{Q})$. It is a compact Hausdorff space which contains $G/\Gamma$ as an open dense set \cite[Proposition III.14.5]{borel2006compactifications}.

We aim to give a detailed form of this compactification. First, recall that $\mathbf{G}(\mathbb{Q})$ acts on the set of $\mathbb{Q}$-parabolic subgroups of $\mathbf{G}$ by conjugation. Choose some set of representatives $\{\mathbf{P}_i\}_{i=1}^{k}$ for the $\Gamma$-conjugacy classes of the $\mathbb{Q}$-parabolic subgroups. 
Moreover, let  \[\pi_{M_\mathbf{P}}:\:P=M_{\mathbf{P}} A_{\mathbf{P}} N_P\to M_{\mathbf{P}}\] be the natural projection, and define \[\Gamma_{M_{\mathbf{P}}}=\pi_{M_{\mathbf{P}}}(\Gamma\cap P).\] 
Now, it follows \cite[Proposition III.14.8]{borel2006compactifications} that
\[\overline{ G/\Gamma}^{\RBS}=\coprod_{i=1}^{k} KM_{\mathbf{P}_{i}}/\Gamma_{M_{\mathbf{P}_i}}=\coprod_{i=1}^{k}G/A_{\mathbf{P}_{i}}N_{P_{i}}\Gamma_{M_{\mathbf{P}_i}}.\]

Back to the main case of interest which is $G=\SL_{\mathtt{d}}(\mathbb{R})$ and $\Gamma=\SL_{\dtt}(\mathbb{Z})$, we may choose the set of representatives to be the standard parabolic subgroups $\mathcal{P}$, so that
\[\overline{ G/\Gamma}^{\RBS}=\coprod_{P\in\mathcal{P}}e_{\infty}(P),\]
where \[e_{\infty}(P)\coloneqq K M_{\mathbf{P}}/\Gamma_{M_{\mathbf{P}}}.\]
In this case, if $P$ has block sized $l_1,\ldots,l_k$, then we simply have 
\[ M_{\mathbf{P}}/\Gamma_{M_{\mathbf{P}}}=\prod_{j=1}^{k}\SL_{l_j}(\mathbb{R})/ \SL_{l_j}(\mathbb{Z}).\]

\subsection{Proof of Theorem~\ref{theorem:1.1new}}\label{subsec:8.3}
\begin{proof}[Proof of Theorem~\ref{theorem:1.1new}]
Let $\epsilon>0$. Let $\delta,\delta^{\prime},r$ be as in Theorem~\ref{theorem:1.1}, so that additionally
\[f(\delta,r,\phi)<\frac{\epsilon}{3}(1+\|\phi\|)\]
for all $\phi\in\Lie(A)^{\ast}$,
and
\[\frac{1}{|\log\delta^{\prime}|}<\frac{\epsilon}{3}.\]
Then \[C_{\att}f(\delta,r,\phi)+C^{\prime}_{\att,\phi}\frac{1}{|\log\delta^{\prime}|}<D_{\att,\phi}\epsilon\]
for all $\att\in A,\ \phi\in\Lie(A)^{\ast}$, where 
$D_{\att,\phi}=\max\{C_{\att},C_{\att}\|\phi\|,C^{\prime}_{\att,\phi}\}.$

For all $H\in\mathcal{P}$, let
\[Y_H=\bigcup_{\substack{P,Q\in\mathcal{P}:\\ Q\subseteq H\subseteq P}}\cusp_{\delta,\delta^{\prime}}(P,Q)= \{x\in X:\ \eta(x,\delta)\subseteq \eta(H)\subseteq\eta(x,\delta^{\prime})\}.\]
It is easy to see by the definition of the RBS compactification that $\overline{Y_H}\cap e_{\infty}(Q)=\emptyset$ holds for any $Q\in\mathcal{P}$ with $H\not\subseteq Q$, where the closure is in the RBS compactification. 
Indeed, $Y_{H}$ is unbounded (in terms of the jumps $\eta_i$ in the successive minima) precisely in the directions $\eta(H)$ while the boundary component $e_{\infty}(Q)$ is obtained by approaching infinity at all the directions $\eta(Q)$ together. Therefore, in order for the intersection to be non-empty we must have $\eta(Q)\subseteq\eta(H)$, hence $H\subseteq Q$. Note that by the same argument, there clearly exists an open set $U_{H}\subseteq\overline{G/\Gamma}^{\RBS}$ which contains $\overline{Y_H}$ and satisfies the same property $U_{H}\cap e_{\infty}(Q)=\emptyset$ for all $Q\in\mathcal{P}$ with $H\not\subseteq Q$.

Now, assume $\att$ and $\phi$ are given. Then 
for any $Q\subseteq P$ and any $[w]_{Q}\in W_{Q,\att}$, we let $H_{P,Q,[w]_{Q}}$ be a subgroup $Q\subseteq H_{P,Q,[w]_{Q}}\subseteq P$ so that
\[\max_{Q\subseteq H\subseteq P}
(h-\phi)([w]_{H})=
(h-\phi)([w]_{H_{P,Q,[w]_Q}}).\]
Then we join together cusp regions with the same $H_{P,Q,[w]_{Q}}$, and define
\[V_{H,[\tau]_H}=\bigcup
\Big\{\cusp_{\delta,\delta^{\prime}}(P,Q,[w]_{Q}):\ H_{P,Q,[w]_{Q}}=H,\ [w]_{H}=[\tau]_H \Big\}.\]
These sets are clearly mutually disjoint, and $V_{H,[\tau]_{H}}\subseteq U_{H}$.
Note that
\begin{equation}\label{eq:8.1n}\bigcupdot_{\substack{H\in \mathcal{P},\\ [\tau]_H\in W_{H,\att}}}V_{H,[\tau]_H}=\bigcupdot_{\substack{Q\subseteq P,\\ [w]_{Q}\in W_{Q,\att}}}\cusp_{\delta,\delta^{\prime}}(P,Q,[w]_{Q})).\end{equation}

Then, it follows from Theorem~\ref{theorem:1.1}, using Equation~\eqref{eq:8.1n}, that
\begin{align*}
\numberthis\label{eq:8.2}
h_{\mu}(\att)&\leq
\sum_{\substack{P,Q\in\mathcal{P}:\\Q\subseteq P}}\sum_{[w]_{Q}\in W_{Q,\att}}\mu(\cusp_{\delta,\delta^{\prime}}(P,Q,[w]_{Q}))\cdot\max_{\substack{H\in\mathcal{P}:\\Q\subseteq H\subseteq P}}(h-\phi)([w]_{H})+D_{\att,\phi}\epsilon\\&
=
\sum_{\substack{P,Q\in\mathcal{P}:\\Q\subseteq P}}\sum_{[w]_{Q}\in W_{Q,\att}}\mu(\cusp_{\delta,\delta^{\prime}}(P,Q,[w]_{Q}))\cdot(h-\phi)([w]_{H_{P,Q,[w]_{Q}}})+D_{\att,\phi}\epsilon\\&
=\sum_{H\in \mathcal{P}}\sum_{[\tau]_{H}\in W_{H,\att}}\mu(V_{H,[\tau]_H})\cdot(h-\phi)([\tau]_H)+D_{\att,\phi}\epsilon
.
\end{align*}

This is indeed the correct bound we desired. However, the collection of mutually disjoint sets \[\{V_{H,[\tau]_{H}}:\ H\in\mathcal{P},\ [\tau]_{H}\in W_{H,\att}\},\] 
is not a partition of $X$, because it does not cover the space, so we need to extend it without significantly affecting the entropy bound.
It follows from Equation~\eqref{eq:8.1n} that
\[\bigcupdot_{\substack{H\in \mathcal{P}:\\ [\tau]_H\in W_{H,\att}}}V_{H,[\tau]_H}=X\smallsetminus\bigcupdot_{H\in\mathcal{P}\smallsetminus\{G\}}\cusp_{\delta^{\prime}}(H,\emptyset),\]
so we only need for all $H\in\mathcal{P}\smallsetminus\{G\}$ to redefine $V_{H,[\tau]_{H}}$, for some $[\tau]_{H}\in W_{H,\att}$, to be the union of the original set  $V_{H,[\tau]_{H}}$ with $\cusp_{\delta^{\prime}}(H,\emptyset)$. Indeed, for any choice of $[\tau]_{H}$ the redefined set $V_{H,[\tau]_H}$ would still be a subset of $U_{H}$. Furthermore, by Proposition~\ref{proposition:3.9}
\[\mu(\cusp_{\delta^{\prime}}(H,\emptyset))\ll_{\att} \frac{1}{|\log\delta^{\prime}|}<\frac{\epsilon}{3}\]
for any $\mu$,
so adding $\mathcal{N}_{\delta^{\prime}}(H,\emptyset)$ would not affect the entropy by much because
\[\mu(\cusp_{\delta^{\prime}}(H,\emptyset)) \cdot (h-\phi)([\tau]_{H})\gg_{\att,\phi} - \epsilon,\]
so the same bound as in Equation~\eqref{eq:8.2} would still hold, only with a re-adjusted constant $D_{\att,\phi}\ll_{\att,\phi} 1$.
Note that the fact that $\mathcal{N}_{\delta^{\prime}}(H,\emptyset)$ has small measure allowed us to choose for $H\in\mathcal{P}$ the set $V_{H,[\tau]_{H}}$ to which we add $\mathcal{N}_{\delta^{\prime}}(H,\emptyset)$ arbitrarily, but instead we could have avoided this argument and just chose some $[\tau]_{H}$ so that $(h-\phi)([\tau]_{H})\geq 0$. 
\end{proof}

\subsection{Proof of Theorem~\ref{corollary:1.3} and Corollaries~\ref{corollary:1.4}-\ref{corollary:1.5}}
\begin{proof}[Proof of Theorem~\ref{corollary:1.3}]
Let $\delta,\delta^{\prime},r$ be as in the statement of  Theorem~\ref{theorem:1.1}. Then
\begin{multline*}
h_{\mu_n}(\att)\leq
\sum_{\substack{P,Q\in\mathcal{P}:\\Q\subseteq P}}\sum_{[w]_{Q}\in W_{Q,\att}}\mu_n(\cusp_{\delta,\delta^{\prime}}(P,Q,[w]_{Q}))\cdot\max_{\substack{H\in\mathcal{P}:\\Q\subseteq H\subseteq P}}(h-\phi)([w]_{H})\\
+C_{\att}f(\delta,r,\phi)+C_{\att,\phi}^{\prime}\frac{1}{|\log\delta^{\prime}|}
\end{multline*}
for all $n\in \mathbb{N}$,
and in particular
\begin{multline*}
h_{\mu_n}(\att)\leq
\sum_{\substack{P,Q\in\mathcal{P}:\\Q\subseteq P}}\mu_n(\cusp_{\delta,\delta^{\prime}}(P,Q))\cdot\max_{[w]_{Q}\in W_{Q,\att}}\max_{\substack{H\in\mathcal{P}:\\Q\subseteq H\subseteq P}}(h-\phi)([w]_{H})\\
+C_{\att}f(\delta,r,\phi)+C_{\att,\phi}^{\prime}\frac{1}{|\log\delta^{\prime}|},
\end{multline*}
where we used implicitly the easy to check fact that for any $\phi\in\Lie(A)^{\ast}$ and $H\in\mathcal{P}$ there is some $[w]_{H}\in W_{H,\att}$ so that $\phi([w]_{H})\geq 0$.

Taking the limsup as $n\to\infty$, we obtain
\begin{multline}\label{eq:8.1}
\limsup_{n\to\infty} h_{\mu_n}(\att)\leq
\sum_{\substack{P,Q\in\mathcal{P}:\\Q\subseteq P}}\nu(\overline{\cusp_{\delta,\delta^{\prime}}(P,Q)})\cdot\max_{[w]_{Q}\in W_{Q,\att}}\max_{\substack{H\in\mathcal{P}:\\Q\subseteq H\subseteq P}}(h-\phi)([w]_{H})\\
+C_{\att}f(\delta,r,\phi)+C_{\att,\phi}^{\prime}\frac{1}{|\log\delta^{\prime}|}
\end{multline}
where the closure is in the RBS compactification.

Next, we take the limit as $\delta\to 0^{+}$. We take a positive sequence $(\delta_{j})_{j=1}^{\infty}$ converging to $0$, and choose as in Remark~\ref{remark:1.2} suitable $\delta_{j}^{\prime},r_{j}$ so that the error terms $f(\delta_{j},r_{j},\phi)$ and $\frac{1}{|\log\delta^{\prime}_{j}|}$ vanish at the limit.
By the description of the RBS compactification in~\S\ref{subsec:8.2}, it follows that 
\[\limsup_{j\to\infty} \overline{\cusp_{\delta_j,\delta_j^{\prime}}(P,Q)}=\begin{cases}
e_{\infty}(P) & P=Q \\ 
\emptyset & P\not=Q
\end{cases}.
\]
Therefore, taking the limit in Equation~\eqref{eq:8.1},
we get
\[\limsup_{n\to\infty}h_{\mu_n}(\att)\leq \sum_{P\in \mathcal{P}}\nu(e_{\infty}(P))\cdot\max_{[w]_{P}\in W_{P,\att}}(h-\phi)([w]_{P})\]
as required.
\end{proof}

Corollaries~\ref{corollary:1.4}-\ref{corollary:1.5} follow immediately from Theorem~\ref{corollary:1.3}.

\bibliographystyle{plain}

\end{document}